\documentclass[11pt]{amsart}
\usepackage{amssymb, amsmath, amscd, eucal, rotating}

\input xy
 \xyoption{all}

\newif\ifPDF
\ifx\pdfoutput\undefined %Step to check pdf on older tetex version
        \ifx\pdfversion\undefined
                \PDFfalse
        \fi
\else %Newer version of tetex defines pdfoutput in both cases, so need
        \ifnum\pdfoutput > 0
                \PDFtrue
        \else
                \PDFfalse
        \fi
\fi

\numberwithin{equation}{section}

\theoremstyle{plain}
				
\newtheorem{proposition}{Proposition}[section]
\newtheorem{theorem}[proposition]{Theorem}		
\newtheorem{corollary}[proposition]{Corollary}
\newtheorem{lemma}[proposition]{Lemma}

\theoremstyle{definition}

\newtheorem{definition}[proposition]{Definition}
\newtheorem{remark}[proposition]{Remark}
\newtheorem{example}[proposition]{Example}

\newcommand{\Vbold}{{\bf V}}

\newcommand{\ebold}{{\bf e}}

\newcommand{\lbold}{{\bf l}}

\newcommand{\vbold}{{\bf v}}

\newcommand{\CBbb}{\mathbb C}

\newcommand{\HBbb}{\mathbb H}

\newcommand{\PBbb}{\mathbb P}

\newcommand{\RBbb}{\mathbb R}

\newcommand{\ZBbb}{\mathbb Z}

\newcommand{\Acal}{\mathcal A}
\newcommand{\Bcal}{\mathcal B}
\newcommand{\Ccal}{\mathcal C}
\newcommand{\Dcal}{\mathcal D}
\newcommand{\Ecal}{\mathcal E}
\newcommand{\Fcal}{\mathcal F}
\newcommand{\Gcal}{\mathcal G}

\newcommand{\Kcal}{\mathcal K}
\newcommand{\Lcal}{\mathcal L}
\newcommand{\Mcal}{\mathcal M}
\newcommand{\Ncal}{\mathcal N}
\newcommand{\Ocal}{\mathcal O}

\newcommand{\Qcal}{\mathcal Q}
\newcommand{\Rcal}{\mathcal R}
\newcommand{\Scal}{\mathcal S}

\newcommand{\Vcal}{\mathcal V}
\newcommand{\Wcal}{\mathcal W}

\newcommand{\Mfrak}{\mathfrak M}

\newcommand{\gfrak}{\mathfrak g}

\newcommand{\SL}{\mathsf{SL}}
\newcommand{\PSL}{\mathsf{PSL}}
\newcommand{\GL}{\mathsf{GL}}

\newcommand{\SU}{\mathsf{SU}}

\newcommand{\U}{\mathsf{U}}

\DeclareMathOperator{\End}{End}
\DeclareMathOperator{\Hom}{Hom}

\DeclareMathOperator{\rank}{rank}

\DeclareMathOperator{\vol}{vol}

\DeclareMathOperator{\coker}{coker}

\DeclareMathOperator{\ad}{ad}

\DeclareMathOperator{\tr}{tr}

\newcommand{\dbar}{\bar\partial}

\newcommand{\lra}{\longrightarrow}
\newcommand{\doubleslash}{\bigr/ \negthinspace\negthinspace \bigr/}
\newcommand{\simrightarrow}{\buildrel \sim\over\longrightarrow }

\newcommand{\Op}{{\rm Op}}
\newcommand{\Gr}{{\rm Gr}}

\newcommand{\Tor}{{\rm Tor}}

\DeclareMathOperator{\YMH}{YMH}
\DeclareMathOperator{\hol}{hol}
\DeclareMathOperator{\Ric}{Ric}
\DeclareMathOperator{\Riem}{Riem}

\begin{document}

% Topmatter

\title[Higgs bundles and local systems]
{Higgs bundles and local systems on Riemann surfaces}

\author[Richard A. Wentworth]{Richard A. Wentworth}

\address{Department of Mathematics,
   University of Maryland,
   College Park, MD 20742, USA}
\email{raw@umd.edu}
\thanks{R.W. supported in part by NSF grants DMS-1037094 and DMS-1406513.  
The author also acknowleges support from NSF grants DMS 1107452, 1107263, 1107367 ``RNMS: GEometric structures And Representation varieties" (the GEAR Network).}

%\subjclass[2000]{Primary: 58D15; Secondary: 14D20, 32G13}
\date{\today}

%\begin{abstract} 
%\end{abstract}

%$\hbox{}$
%\centerline{\fbox{\bf *** Preliminary Version ***}}
%\vskip .75in

% End Topmatter

\maketitle
\setcounter{tocdepth}{2}
\tableofcontents
\thispagestyle{empty}

%\newpage

\baselineskip=16pt

%\thispagestyle{plain}
%\setcounter{page}{1}
%\setcounter{section}{-1}
%\setcounter{footnote}{0}

%%%%%%%%%%%%%%%%%%%%%%%%%%%%%%%%%%%%%%%%%%%%%%%%
%%%  BODY OF THE ARTICLE HERE  %%%%%%%%%%%%%%%%%%%%%%%%%%%%

\section{Preface}

These notes are based on lectures given at
 the Third International School on
Geometry and Physics at the Centre de Recerca Matem\`atica in
Barcelona, March 26--30, 2012.
The aim of the School's four  lecture series was to give a rapid introduction  to  Higgs bundles, representation varieties,
 and mathematical physics. 
 While the scope of these subjects is very broad,
that of these notes is far more modest.
The main topics covered here are:
\begin{itemize}
\item The Hitchin-Kobayashi-Simpson correspondence for Higgs bundles
on Riemann surfaces.
\item The Corlette-Donaldson theorem relating the moduli
spaces of Higgs bundles and semisimple representations of the
fundamental group.
%\item The equivariant cohomology of moduli spaces.
\item  A description of the oper moduli space and its
relationship to systems of holomorphic differential equations, Higgs bundles,  and the Eichler-Shimura isomorphism.
\end{itemize}
These topics have been treated extensively in the literature.  I have tried to condense the 
key ideas into a presentation that requires as little background as possible.
With regard to the first item, 
I give a complete proof of the Hitchin-Simpson theorem (Theorem \ref{thm:hitchin}) that combines techniques that have emerged since Hitchin's seminal paper \cite{Hitchin87a}.  In the case of Riemann surfaces a direct proof for arbitrary rank which avoids introduction of the Donaldson functional can be modeled on Donaldson's proof of the Narasimhan-Seshadri theorem  in  \cite{Donaldson83} (such a proof was suggested in \cite{Simpson87}).  Moreover, the Yang-Mills-Higgs flow can be used to extract minimizing sequences with desirable properties.  A similar idea is used in the Corlette-Donaldson proof of the existence of equivariant harmonic maps (Theorem \ref{thm:corlette}).  Indeed, I have sought in these notes to exhibit the parallel structure of the proofs of these two fundamental results.
Continuity of the two flows is the key to the relationship between the equivariant cohomology of the moduli space of semistable Higgs bundles on the one hand and the moduli space of representations on the other.
On first sight the last item in the list above is a rather different topic from the others, but it is nevertheless  deeply related in ways that are perhaps still not completely understood.  Opers 
\cite{BeilinsonDrinfeld05} play an important role in the literature on the Geometric Langlands program \cite{Frenkel07}. My intention here is to give  fairly complete  proofs of the basic facts about opers and their relationship to differential equations  and Higgs bundles (see also \cite{Simpson10}).

Due to the limited amount of time for the lectures
I have necessarily omitted many important aspects of this subject.
Two in particular are worth mentioning. 
 First, I deal only with vector bundles and do not 
consider principal bundles with more general structure groups.  For example, there is no discussion of representations into the various real forms of a complex Lie group.  Since some of the other lectures at this introductory school will treat this topic in great detail I hope this omission will  not be serious.
 Second, I deal only with closed Riemann surfaces and
do not consider extra ``parabolic'' structures at marked points.   In some sense this ignores an important aspect at  the heart of the classical literature on holomorphic differential equations (cf.\ \cite{Simpson90, Boalch01}).  Nevertheless, for the purposes of introducing the global structure of  moduli spaces, I feel it is better to first treat the case of closed surfaces.
While much of the current research in the field is 
directed toward the two generalizations above,
 these topics are left for further reading.

I have tried to give  references to essential results in these notes. 
 Any omissions or incorrect attributions are due solely to my own ignorance of the 
extremely rich and vast literature, and 
for these I extend my sincere apologies.  Also,
there is  no claim to originality of the proofs given here.  A perusal of Carlos Simpson's foundational contributions to this subject  is highly recommended for 
anyone wishing to learn about Higgs bundles (see \cite{Simpson87, Simpson88, Simpson92, Simpson94a, Simpson94b}).  In addition, the original articles of Corlette \cite{Corlette88}, Donaldson \cite{Donaldson83, Donaldson87}, and of course Hitchin \cite{Hitchin87a, Hitchin87b, Hitchin92} are indispensable.  Finally, I also mention more  recent survey articles  \cite{BurgerIozziWienhard, BradlowGothenPrada, Guichard11} which  treat  especially the case of representations to general Lie groups.
  I am grateful to the organizers, Luis \'Alvarez-C\'onsul,
Peter Gothen, and Ignasi Mundet i Riera, for inviting me to
give these lectures, and to the CRM for its hospitality. Additional thanks to Bill Goldman, Fran\c cois Labourie, Andy Sanders, and Graeme Wilkin for discussions related to the topics presented here, and to Beno\^ it Cadorel for catching several typos. The anonymous referee also made very useful suggestions, for which I owe my gratitude.

\bigskip

\centerline{\sc Notation}

\medskip  

\begin{itemize}
\item  $X=$ a compact Riemann surface of genus $g\geq 2$.
\item $\pi=\pi_1(X,p)=$ the fundamental group of $X$.
\item $\HBbb=$ the upper half plane in $\CBbb$.
\item $\Ocal=\Ocal_X=$  the sheaf of  germs of holomorphic functions on $X$.
\item $\Kcal=\Kcal_X=$  the canonical sheaf of $X$.
\item 
$E=$ a complex vector bundle on $X$. 
\item $H=$ a hermitian metric on $E$.
\item $\nabla=$ a connection on $E$.
\item $A$ (or $d_A$) $=$  a unitary connection on $(E, H)$.
\item $\Ccal_E=$ the space of connections on a rank $n$ bundle $E$.
\item $\Acal_E=$ the space of unitary connections on $E$.
\item $\Bcal_E=$ the space of Higgs bundles.
\item $\Bcal_E^{ss}=$ the space of semistable Higgs bundles.
\item $\Gcal_E$ (resp.\ $\Gcal_E^\CBbb$) $=$ the unitary (resp.\ complex) gauge group.
\item $\dbar_E=$ a Dolbeault operator on $E$, which is equivalent to a holomorphic structure. 
\item $(\dbar_E,H)=$ the Chern connection.
\item $\Ecal=$ sheaf of germs of holomorphic sections of a holomorphic bundle $(E,\dbar_E)$.
\item $\gfrak_E=$ the bundle of skew-hermitian endomorphisms of $E$.
\item $\End E=\gfrak_E^\CBbb$ the endomorphism bundle of $E$.
\item 
$\Vbold=$ a local system on $X$. 
\item $\Vbold_\rho=$ the local system associated to a representation $\rho:\pi\to \GL_n(\CBbb)$.
\item $\underline R =$  the locally constant sheaf modeled on a ring $R$.  
\item $L^p_k=$ the Sobolev space of functions/sections with
$k$ derivatives in $L^p$.
\item $C^{k,\alpha}=$ the space of functions/sections with $k$ 
derivatives being H\"older continuous with exponent $\alpha$.
\end{itemize}

\section{The Dolbeault Moduli Space} \label{sec:dolbeault}

\subsection{Higgs bundles}

\subsubsection{Holomorphic bundles and stability} \label{sec:stability}
Throughout these notes, $X$ will denote a closed Riemann
surface of genus $g\geq 2$ and $E\to X$ a  complex
vector bundle.  We begin with a discussion of the basic differential geometry of complex vector bundles. Good references for this material are Kobayashi's book \cite{Kobayashi87} and Griffiths and Harris \cite{GriffithsHarris78}.
 A holomorphic structure on $E$ is equivalent
to a choice of {\bf $\dbar$-operator}, i.e.\ a $\CBbb$-linear map
$$
\dbar_E: \Omega^0(X,E)\lra \Omega^{0,1}(X,E)
$$
satisfying the Leibniz rule: $\dbar_E(fs)=\dbar f\otimes
s+f\dbar_E s$, for a function $f$ and a section $s$ of $E$.
  Indeed, if $\{s_i\}$ is a local holomorphic
frame of a holomorphic bundle, then the Leibniz rule uniquely determines the $\dbar$-operator on the underlying complex vector bundle. Conversely, since there is no integrability condition on Riemann surfaces, given a
$\dbar$-operator as defined above one
can always find local holomorphic frames (cf.\ \cite[\S 5]{AtiyahBott82}).  When we want to specify the holomorphic
structure we write $(E,\dbar_E)$.  We also introduce the
notation $\Ecal$ for a sheaf of germs of holomorphic sections
of $(E,\dbar_E)$. We will sometimes confuse the terminology and call $\Ecal$ a holomorphic bundle. 

If $\Scal\subset\Ecal$ is a holomorphic subbundle with
quotient $\Qcal$, then a smooth splitting $E=S\oplus Q$ allows
us to represent the $\dbar$-operators as
\begin{equation} \label{eqn:dbar}
\dbar_E=\left( \begin{matrix} \dbar_S&\beta\\
0&\dbar_Q\end{matrix}\right)
\end{equation}
where $\beta\in \Omega^{0,1}(X,\Hom(Q,S))$ is called the
{\bf second fundamental form}.
A hermitian metric $H$ on $E$ gives  an orthogonal splitting.
In this case the subbundle $S$ is determined by its orthogonal projection operator $\pi$, which is an
endomorphism of $E$ satisfying
\begin{enumerate}
\item $\pi^2=\pi$;
\item $\pi^\ast=\pi$;
\item $\tr\pi$ is constant.
\end{enumerate}
The statement  that $S\subset E$ be holomorphic is equivalent
to the further condition 
$$
\hspace{-11.85cm} {\rm (iv)}\ (I-\pi)\dbar_E\,\pi=0\ .
$$
Notice that (i) and (iv) imply (iii), and that $\beta=-\dbar_E\pi$. Hence, there is a 1-1 correspondence between holomorphic subbundles of $\Ecal$ and endomorphisms $\pi$ of the hermitian bundle $E$ satisfying conditions (i), (ii), and (iv).  I should point out that the generalization of this description of holomorphic subsheaves to higher dimensions is a key idea of Uhlenbeck and Yau \cite{UhlenbeckYau86}.

A {\bf connection} $\nabla$ on $E$ is a $\CBbb$-linear map
$$
\nabla: \Omega^0(X,E)\lra \Omega^{1}(X,E)\ ,
$$
satisfying the Leibniz rule: $\nabla(fs)=df\otimes
s+f\nabla s$, for a function $f$ and a section $s$.
Given a hermitian metric $H$, we call a connection {\bf unitary}
(and we will always then denote it by $A$ or $d_A$)
if it preserves $H$, i.e.
\begin{equation} \label{eqn:unitary}
d\langle s_1, s_2\rangle_H=\langle d_A s_1, s_2\rangle_H
+\langle s_1, d_A s_2\rangle_H\ .
\end{equation}
The curvature of a connection $\nabla$ 
is $F_\nabla=\nabla^2$ (perhaps more precise notation: $\nabla\wedge \nabla$).
If $\gfrak_E$ denotes the bundle of  skew-hermitian endomorphisms of $E$ and $\gfrak_E^\CBbb$ its complexification, then $F_A\in \Omega^2(X,\gfrak_E)$ for a unitary connection, and $F_\nabla\in \Omega^2(X,\gfrak_E^\CBbb)$ in general.

\begin{remark} \label{rem:traceless}
We will mostly be dealing with connections on bundles that induce a fixed connection on the  determinant bundle. These will correspond, for example, to representations into $\SL_n$ as opposed to $\GL_n$. In this case, the bundles $\gfrak_E$ and $\gfrak_E^\CBbb$ should be taken to consist of traceless endomorphisms.  
\end{remark}

Finally, note that a connection always induces a
$\dbar$-operator by taking its $(0,1)$ part.  Conversely, a
$\dbar$-operator gives a unique unitary connection, called the
{\bf Chern connection}, which we will sometimes denote by 
$d_A=(\dbar_E, H)$.  The complex structure on $X$ splits $\Omega^1(X)$  into $(1,0)$ and $(0,1)$ parts, and hence also splits the connections.  We denote these by, for example, $d'_A$ and $d''_A$, respectively.  So for $d_A=(\dbar_E,H)$, $d''_A=\dbar_E$, and $d'_A$ is determined by 
$
\partial\langle s_1,s_2\rangle_H=\langle d_A's_1, s_2\rangle_H
$,
for any pair of holomorphic sections $s_1$, $s_2$. Henceforth, I will mostly omit $H$ from the notation if there is no chance of confusion.

\begin{example} \label{ex:line-bundle-curvature}
Let $\Lcal$ be   a holomorphic line bundle with hermitian metric $H$. For a local holomorphic frame $s$ write $H_s=|s|^2$.
Then $F_{(\dbar_L,H)}=\dbar\partial\log H_s$, and the right hand side is independent of the choice of frame.
\end{example}

The transition functions of a collection of local trivializations of a holomorphic line bundle   on the open sets of a covering of $X$ give a $1$-cocycle
with values in the sheaf $\Ocal^\ast$ of germs of nowhere vanishing holomorphic functions.  The set of isomorphism classes of line bundles is then $H^1(X,\Ocal^\ast)$.
Recall that on a compact Riemann surface every holomorphic 
line bundle has a meromorphic section.  This 
gives an equivalence between the categories of holomorphic 
line bundles under tensor products and linear equivalence classes 
of divisors $\Dcal=\sum_{x\in X}m_x x$
 with their additive structure (here $m_x\in \ZBbb$ is zero for all but finitely many $x\in X$).
We shall denote by $\Ocal (\Dcal)$ the line 
bundle thus associated to $\Dcal$.
 Furthermore,  a divisor has a  {\bf degree}, $\deg\Dcal=\sum_{x\in X}m_x$.
We define this to be the degree of $\Ocal(\Dcal)$.  Alternatively, from the exponential sequence
$$
0\lra \underline \ZBbb\lra \Ocal \stackrel{f\mapsto e^{2\pi i f}}{\xrightarrow{\hspace*{1.5cm}}}\Ocal ^\ast \lra 0\ ,$$
we have the long exact sequence in cohomology:
$$
0\lra H^1(X,\ZBbb)\lra H^1(X,\Ocal)\lra H^1(X,\Ocal^\ast)
\stackrel{c_1}{\xrightarrow{\hspace*{.75cm}}}
H^2(X,\ZBbb)\lra 0\ .
$$
The fundamental class of $X$ identifies $H^2(X,\ZBbb)\cong\ZBbb$, and it is a standard exercise to show that under this identification: $\deg(\Dcal)=c_1(\Ocal(\Dcal))$. 
For a holomorphic vector bundle $\Ecal$, we declare the degree $\deg\Ecal :=\deg\det\Ecal$.
 Notice that the degree is topological, i.e.\ it does not depend on the holomorphic structure, just on the underlying complex bundle $E$.  By the Chern-Weil theory, for any hermitian metric $H$ on $\Ecal$ we have
 \begin{equation} \label{eqn:chern-weil}
 c_1(E)=\left[ \frac{\sqrt{-1}}{2\pi}\tr F_{(\dbar_E,H)}\right] =\left[ \frac{\sqrt{-1}}{2\pi}F_{(\dbar_{\det E},\det H)}\right] \ .
 \end{equation}
 Complex vector bundles on Riemann surfaces are classified topologically by their rank and degree.
 We will also make use of the {\bf slope} (or normalized degree) of a bundle, which is defined by the ratio $\mu(E)=\deg E/\rank E$.
 
If a line bundle $\Lcal=\Ocal(\Dcal)$ has a nonzero holomorphic section, then  since $\Dcal$ is linearly equivalent to an effective divisor (i.e.\ one with $m_x\geq 0$ for all $x$), $\deg\Lcal\geq 0$.  It follows that if $\Ecal$ is a holomorphic vector bundle with a subsheaf $\Scal\subset\Ecal$ and $\rank\Scal=\rank\Ecal$, then $\deg\Scal\leq \deg\Ecal$.  Indeed, the assumption implies $\det\Ecal\otimes (\det\Scal)^\ast$ has a nonzero holomorphic section.  We will use this fact later on.  
Notice that in the case above, $\Qcal=\Ecal/\Scal$ is a torsion sheaf.  In general, for any subsheaf $\Scal\subset \Ecal$ of a holomorphic vector bundle, $\Scal$ is contained in a uniquely defined holomorphic subbundle $\Scal'$ of $\Ecal$ called the {\bf saturation} of $\Scal$.  It is obtained by taking the kernel of the induced map $\Ecal\to \Qcal/\Tor(\Qcal)\to 0$.  From this discussion we conclude that $\deg\Scal$ is no greater than the degree $\deg\Scal'$ of its saturation.

 Let  $\omega$ be the K\"ahler form associated to a choice of conformal metric on $X$. This will be fixed throughout, and  for convenience we normalize
 so that 
 $$\int_X\omega=2\pi\ .$$
 The contraction: 
 $\Lambda: \Omega^2(X)\to \Omega^0(X)$, is defined  by setting $\Lambda(f\omega)=f$ for any function $f$.
For a holomorphic subbundle $\Scal$ of a hermitian holomorphic bundle $\Ecal$ with
projection operator $\pi$ we have the following useful formula, which follows easily from direct calculation using \eqref{eqn:chern-weil}.
 
 \begin{equation} \label{eqn:degree}
 \deg\Scal=\frac{1}{2\pi}\int_X\tr(\pi \sqrt{-1}\Lambda F_{(\dbar_E,H)})\, \omega
 -\frac{1}{2\pi}\int_X |\beta|^2\, \omega\ .
 \end{equation}

\begin{definition} \label{def:stability}
We say that $\Ecal$ is  {\bf stable}
 (resp.\ {\bf semistable}) if for all
 holomorphic subbundles $\Scal\subset \Ecal$, 
$0<\rank\Scal<\rank\Ecal$, 
we have $\mu(\Scal)<\mu(\Ecal)$ (resp.\ $\mu(\Scal)\leq\mu(\Ecal)$).  We call $\Ecal$ {\bf polystable} if it is a direct sum of stable bundles of the same slope.
\end{definition}

\begin{remark} \label{rem:tensor}
Line bundles are trivially stable.  If $\Ecal$ is (semi)stable and $\Lcal$ is a line bundle, then $\Ecal\otimes\Lcal$ is also (semi)stable.
\end{remark}

Before giving an example, recall the notion of an extension
\begin{equation} \label{eqn:extension}
0\lra \Scal\lra\Ecal\lra\Qcal\lra 0\ .
\end{equation}
The {\bf extension class}  is 
the image of the identity endomorphism under the coboundary map of the long exact sequence associated to \eqref{eqn:extension}
$$
H^0(X,\Qcal\otimes\Qcal^\ast)\lra H^1(X, \Scal\otimes \Qcal^\ast) \ .
$$
Notice that the isomorphism class of the bundle 
$\Ecal$ is unchanged under scaling, so the extension class (if not zero) should 
be regarded as an element of the projective space $\PBbb(H^1(X, \Scal\otimes \Qcal^\ast)
)$.
It is then an exercise to see that in terms of the second
fundamental form $\beta$, the extension class coincides
(projectively) with
the corresponding Dolbeault cohomology class $[\beta]\in
H^{0,1}_{\dbar}(X,S\otimes Q^\ast)$.
  We 
say that \eqref{eqn:extension} is {\bf split} if 
the extension class is zero.  Clearly, this occurs if and only there is an injection $\Qcal\to \Ecal$ lifting the projection. 

\begin{example}
Suppose $g\geq 1$.
Consider extensions of the type
$$
0\lra\Ocal \lra\Ecal\lra\Ocal (p)\lra 0\ .
$$
These are parametrized by $H^1(X, \Ocal(-p))\cong H^0(X, \Kcal(p))^\ast \cong H^0(X,\Kcal )^\ast$, which has dimension $g$.  Any non-split extension of this type is stable. Indeed, if $\Lcal\hookrightarrow\Ecal$ is a destabilizing line subbundle, then $\deg\Lcal\geq 1$.  
The induced map $\Lcal\to \Ocal(p)$ cannot be zero, since then by the inclusion $\Lcal\hookrightarrow \Ecal$ it would lift to a nonzero map $\Lcal\to\Ocal $, which is impossible.
Hence,  $\Lcal\to\Ocal (p)$ must be an isomorphism. Such an $\Lcal$ would therefore split the extension.
\end{example}

A connection is {\bf flat} if its curvature vanishes.  We say
that $\nabla$ is {\bf projectively flat} if 
$\sqrt{-1}\Lambda F_\nabla=\mu$,
where $\mu$ is a constant (multiple of the identity).
Note that by our normalization of the area, $\mu=\mu(E)$.
In Section 4, we will prove Weil's criterion for when a
holomorphic bundle $\Ecal$ admits a flat connection (i.e. $\nabla''=\dbar_E$, $F_\nabla=0$). Demanding that the connection be unitary imposes stronger conditions. This is the famous result of
Narasimhan-Seshadri.

\begin{theorem}[Narasimhan-Seshadri {\cite{NarasimhanSeshadri65}}] \label{thm:narasimhan-seshadri}
A holomorphic bundle $\Ecal\to X$ admits a projectively 
flat unitary connection if and only if $\Ecal$  is polystable.
\end{theorem}

\noindent
In Section \ref{sec:hitchin} we will prove Theorem \ref{thm:narasimhan-seshadri} as a special case of the more general result on Higgs bundles (see Theorem \ref{thm:hitchin}).

\subsubsection{Higgs fields} \label{sec:higgs-fields}
A {\bf Higgs bundle} is a pair $(\Ecal,\Phi)$ where $\Ecal$ is
a holomorphic bundle and $\Phi$ is a holomorphic section of
$\Kcal\otimes\End\Ecal$. We will sometimes regard $\Phi$ as a
section of $\Omega^{1,0}(X,\gfrak_E^\CBbb)$ satisfying $\dbar_{E}\Phi=0$.

\begin{definition} \label{def:higgs-stability}
We say that a pair $(\Ecal,\Phi)$ is {\bf stable} 
(resp.\ {\bf semistable}) 
if for all $\Phi$-invariant holomorphic subbundles 
$\Scal\subset \Ecal$, $0<\rank\Scal<\rank\Ecal$, we 
have $\mu(\Scal)<\mu(\Ecal)$ (resp.\ $\mu(\Scal)\leq\mu(\Ecal)$).
It is {\bf polystable} if it is a direct sum of Higgs bundles
of the same slope.
\end{definition}

The following is a simple but useful consequence of the definition and the additive properties of the slope on exact sequences.
\begin{lemma} \label{lem:semistable}
Let $f:(\Ecal_1,\Phi_1)\to (\Ecal_2,\Phi_2)$ be a holomorphic homomorphism of Higgs bundles, $\Phi_2 f=f\Phi_1$.  Suppose $(\Ecal_i,\Phi_i)$ is semistable, $i=1,2$, and $\mu(\Ecal_1)>\mu(\Ecal_2)$.  Then $f\equiv 0$. If $\mu(\Ecal_1)=\mu(\Ecal_2)$ and one of the two is stable, then either $f\equiv 0$ or $f$ is an isomorphism.
\end{lemma}

\begin{proof}
Consider the first statement.  Then if $f\not\equiv 0$,  the assumption $\Phi_2 f=f\Phi_1$ implies that the image of $f$ is $\Phi_2$-invariant, so by the condition on slopes
$f$ must have a kernel.  But then $\ker f$ is $\Phi_1$-invariant.  So $\mu(\ker f)\leq \mu(\Ecal_1)\leq \mu(\coker f)\leq \mu(\Ecal_2)$; contradiction. 
 The second statement follows similarly.
\end{proof}

A {\bf Higgs subbundle}  of $(\Ecal,\Phi)$ is by definition a holomorphic subbundle $\Scal\subset \Ecal$ that is $\Phi$-invariant. The restriction $\Phi_\Scal$ of $\Phi$ to $\Scal$ then makes  $(\Scal,\Phi_\Scal)$ a Higgs bundle, where now the inclusion $\Scal\hookrightarrow\Ecal$ gives a map of Higgs bundles.  Similarly, $\Phi$ induces a Higgs bundle structure on the quotient $\Qcal=\Ecal/\Scal$.
Given an arbitrary Higgs bundle, 
the  {\bf Harder-Narasimhan filtration of $(\Ecal,\Phi)$} 
  is a filtration by Higgs subbundles
$$
0=(\Ecal_0,\Phi_0)\subset (\Ecal_1,\Phi_1)\subset\cdots\subset 
(\Ecal_\ell,\Phi_\ell)=(\Ecal,\Phi)\ ,
$$
such that the quotients $(\Qcal_i, \Phi_{Q_i})=
(\Ecal_i,\Phi_i)/(\Ecal_{i-1}, \Phi_{i-1})$ are  semistable (cf.\ \cite{HarderNarasimhan74}). 
 The filtration is also required to satisfy $\mu(\Qcal_i)>\mu(\Qcal_{i+1})$, 
and one can show that the associated graded object
$
\Gr_{HN}(\Ecal,\Phi)=\oplus_{i=1}^\ell (\Qcal_i, \Phi_{Q_i})
$
is uniquely determined by the isomorphism class of
$(\Ecal,\Phi)$.
 The collection of slopes $\mu_i=\mu(\Qcal_i)$
 is an important invariant of the isomorphism class of the
Higgs bundle.

\begin{remark} \label{rem:maximal}
By construction, \emph{$\mu_i$ is the maximal slope of a Higgs subbundle of $\Ecal/\Ecal_{i-1}$} with its induced Higgs field.  We can also interpret \emph{$\mu_i$ as the minimal slope of a Higgs quotient of $(\Ecal_i,\Phi_i)$}.  Indeed, $(\Ecal_1,\Phi_1)$ is semistable, so this is trivially true if $i=1$.  Suppose $\Ecal_i\to \Qcal\to 0$ is a Higgs quotient for $i\geq 2$ and $\mu(\Qcal)\leq \mu_i$. If $\Qcal$ is the minimal such quotient, then it is semistable with respect to the induced Higgs field.
It follows from Lemma \ref{lem:semistable} that the induced map $\Ecal_1\to \Qcal$ must vanish.  Hence, the quotient passes to $\Ecal/\Ecal_1\to \Qcal\to 0$.  Now by the same argument, $\Ecal_2/\Ecal_1\to \Qcal$ vanishes if $i\geq 3$.
Continuing in this way, we obtain a quotient $\Qcal_i\to \Qcal\to 0$.  Now since $(\Qcal_i,\Phi_{\Qcal_i})$ is semistable and the quotient is nonzero, applying Lemma \ref{lem:semistable} once again, we conclude that $\mu_i\leq \mu(\Qcal)$.
\end{remark}

 Consider  the  $n$-tuple of numbers $\vec\mu(\Ecal,\Phi)=(\mu_1,\ldots, \mu_n)$ obtained
 from the Harder-Narasimhan filtration by repeating each of the 
$\mu_i$'s according to the ranks of the $\Qcal_i$'s.  
%Note that since $\Ecal$ is topologically trivial, $\sum_{i=1}^n\mu_i=0$.
  We then get a vector $\vec\mu(\Ecal, \Phi)$, called the 
{\bf  Harder-Narasimhan type} of $(\Ecal,\Phi)$.
There is a natural partial
ordering on vectors of this type that is key to the stratification we desire.  
 For a pair $\vec\mu$, $\vec\lambda$ of $n$-tuple's satisfying 
$\mu_1\geq\cdots\geq\mu_n$,  $\lambda_1\geq
\cdots\geq\lambda_n$,  and $\sum_{i=1}^n\mu_i=\sum_{i=1}^n\lambda_i$, we define
$$
\vec\lambda\leq\vec\mu \quad \iff\quad  \sum_{j\leq k}\lambda_j\leq\sum_{j\leq k}\mu_j\quad\text{for all}\ k=1,\ldots, n\ .
$$
The importance of this ordering is that it defines a stratification 
of the space of Higgs bundles. In particular, the Harder-Narasimhan type is upper semicontinuous.
This is the direct analog of the Atiyah-Bott stratification
for holomorphic bundles \cite[\S 7]{AtiyahBott82}.

There is a similar filtration of a semistable Higgs bundle
$(\Ecal,\Phi)$,
where the successive quotients are stable, all with 
slope $=\mu(E)$.  This is called the {\bf Seshadri filtration} \cite{Seshadri67}
and its associated graded $\Gr_S(\Ecal,\Phi)$ is therefore
polystable.  When $\Phi\equiv 0$, we recover the usual Harder-Narasimhan and Seshadri filtrations of holomorphic bundles $\Ecal$.  We will denote these by $\Gr_{HN}(\Ecal)$ and $\Gr_{S}(\Ecal)$.

\begin{example}
Consider an extension \eqref{eqn:extension} 
where $\rank\Scal=\rank\Qcal=1$ and $\deg\Scal>\deg\Qcal$.
 Then the Harder-Narasimhan filtration of $\Ecal$ is given by $0\subset \Scal\subset\Ecal$.
\end{example}

%
%
%
%\begin{lemma}
%Let $(\Ecal,\Phi_E)$ be a stable Higgs bundle and $(\Fcal,\Phi_F)$ arbitrary.  Let $f:\Ecal\to \Fcal$ be $\Phi$-invariant holomorphic bundle map. Then the following holds:
%\begin{itemize}
%\item If $\mu(\Ecal)>\mu_{max}(\Fcal,\Phi_F)$, then $f\equiv 0$.
%\item If $(\Fcal,\Phi_F)$ is stable with $\mu(\Fcal)=\mu(\Ecal)$, then either $f\equiv 0$ or $f$ is an isomorphism.
%\end{itemize}
%\end{lemma}

\subsection{The moduli space}

\subsubsection{Gauge transformations} \label{sec:gauge}
Let $\Acal_E$ denote the space of unitary connections on a rank $n$ hermitian vector bundle $E$.  If $\gfrak_E$ denotes the associated bundle of skew-hermitian endomorphisms of $E$, then one observes from the Leibniz rule that $\Acal_E$ is an infinite dimensional affine space modeled on $\Omega^1(X,\gfrak_E)$.   By the construction of the Chern connection discussed in Section \ref{sec:stability}, we also see that $\Acal_E$ can be identified with the space of holomorphic structures on $E$.  We will most often be interested in the case of \emph{fixed determinant}, i.e. where the induced holomorphic structure on $\det E$ is fixed.

The {\bf gauge group} is defined by 
$$
\Gcal_E=\{ g\in \Omega^0(X,\End E) : gg^\ast=I\}\ .
$$
In the fixed determinant case we also impose the condition that $\det g=1$ (see Remark \ref{rem:traceless}).
The gauge group acts on $\Acal_E$ by pulling back connections: $d_{g(A)}=g\circ d_A \circ g^{-1}$.  On the other hand, because of the identification with holomorphic structures we see that the complexification $\Gcal_E^\CBbb$, the {\bf complex gauge group}, also acts on $\Acal_E$.  Explicitly, if $\dbar_E=d''_A$, then $g(A)$ is the Chern connection of $g\circ \dbar_E \circ g^{-1}$.

The space of Higgs bundles is
$$
\Bcal_E=\{ (A,\Phi)\in \Acal_E \times \Omega^0(X, K\otimes \gfrak_E^\CBbb) : d''_A\Phi=0\}\ .
$$
Let $\Bcal_E^{ss}\subset \Bcal_E$ denote the subset of semistable Higgs bundles.

\begin{definition} \label{def:dolbeault}
The moduli space of rank $n$ semistable Higgs bundles (with fixed determinant) on $X$ is
$
\Mfrak_D^{(n)}=\Bcal^{ss}_E\doubleslash\Gcal_E^\CBbb
$,
where the double slash means that the orbits of $(\Ecal,\Phi)$ and $\Gr_S(\Ecal,\Phi)$ are identified.
\end{definition}

We have not been careful about topologies.  In fact, $\Mfrak_D^{(n)}$ can be given the structure of a (possibly nonreduced) complex analytic space using the Kuranishi map (cf.\ \cite{Kobayashi87}).  An algebraic construction using geometric invariant theory is given in \cite{Simpson94a}.  

A second comment is that $\Gcal_E^\CBbb/\Gcal_E$ may be identified with the space of hermitian metrics on $E$.  This leads to an important interpretation when studying the behavior of functionals along $\Gcal_E^\CBbb$ orbits in $\Acal_E/\Gcal_E$: we may either think of varying the complex structure $g(\dbar_E)$ with a fixed hermitian metric, or we may keep $\dbar_E$ fixed and vary the metric $H$ by
$\langle s_1, s_2\rangle_{g(H)}=\langle gs_1, gs_2\rangle_H$.

\subsubsection{Deformations of Higgs bundles}

Let $D''=d''_A+\Phi$, $D'=d'_A+\Phi^\ast$. The metric $\omega$ on $X$ and the hermitian metric on $E$  define $L^2$-inner products on forms with values in $E$ and $\End E$.
We  have the K\"ahler identities
\begin{align}
\begin{split} \label{eqn:kahler}
(D'')^\ast&=-\sqrt{-1}[\Lambda, D'] \ ;\\
(D')^\ast&=\sqrt{-1}[\Lambda, D'']\ ,
\end{split}
\end{align}
(see \cite[p.\ 111]{GriffithsHarris78} for the case $\Phi= 0$; the case $\Phi\neq 0$ follows by direct computation).

The infinitesimal structure of the moduli space is governed by a deformation complex $C(A,\Phi)$, which is obtained by differentiating the condition $d''_A\Phi=0$ and the action of the gauge group.

\begin{equation} \label{eqn:deformation-complex}
C(A,\Phi)\ :\ 
0\lra \Omega^0(X,\gfrak_E^\CBbb)
\stackrel{D''}{\xrightarrow{\hspace*{.5cm}}}
\Omega^{1,0}(X,\gfrak_E^\CBbb)\oplus
\Omega^{0,1}(X,\gfrak_E^\CBbb) 
\stackrel{D''}{\xrightarrow{\hspace*{.5cm}}}
\Omega^{1,1}(X,\gfrak_E^\CBbb)\to 0\ .
\end{equation}
Note that the holomorphicity condition on $\Phi$  guarantees that $(D'')^2=0$.  
Serre duality gives an isomorphism $H^0(C(A,\Phi))\simeq H^2(C(A,\Phi))$.  We call a Higgs bundle {\bf simple} if $H^0(C(A,\Phi))\simeq \CBbb$ (or $\{0\}$ in the fixed determinant case).

\begin{remark} \label{rem:simple}
A stable Higgs bundle is necessarily simple.
Indeed, if $\phi\in \ker D''$, then $\phi$ is a holomorphic
endomorphism of $\Ecal$ commuting with $\Phi$.  In particular,
$\det\phi$ is a holomorphic function and is therefore constant. Also,
$\ker\phi$ is $\Phi$-invariant.
 If $\phi$ is nonzero but not
an isomorphism 
$$
0\lra \ker\phi\lra \Ecal\lra \Ecal/\ker\phi\lra 0\ .
$$
Since $\Ecal/\ker\phi$ is also a subsheaf of $\Ecal$,
stability implies both $\mu(\ker\phi)$ and
$\mu(\Ecal/\ker\phi)$ are both less than $\mu(E)$, which is a contradiction.
Hence, $\phi$ is either zero or an isomorphism.  But applying the same
argument to $\phi-\lambda$ for any scalar $\lambda$, we
conclude that $\phi$ is a multiple of the identity.
\end{remark}

\begin{proposition} \label{prop:tangent-space}
At a simple Higgs bundle $[A, \Phi]$,
 $\Mfrak_D^{(n)}$ is smooth of complex dimension \break $(n^2-1)(2g-2)$, and the tangent space  may be identified with 
\begin{equation} \label{eqn:h1}
H^1(C(A,\Phi))\simeq\left\{ (\varphi,\beta) : d''_A\varphi=-[\Phi,\beta]\ ,\ (d''_A)^\ast\beta=\sqrt{-1}\Lambda[\Phi^\ast,\varphi]\right\}\ .
\end{equation}
\end{proposition}

\begin{example} (cf.\ \cite{Hitchin87a, Hitchin92})
We now give important examples of  stable Higgs bundles; namely, the \emph{Fuchsian} ones. First for rank 2. Fix a choice of square root $\Kcal^{1/2}$ of the canonical bundle, and let $\Ecal=\Kcal^{1/2}\oplus \Kcal^{-1/2}$.  Then the part of the endomorphism bundle that sends $\Kcal^{1/2}\to \Kcal^{-1/2}$ is isomorphic to $\Kcal^{-1}$.  Tensoring by $\Kcal$, it becomes trivial.  Hence, the 
$$
\Phi=\left(\begin{matrix} 0&0\\ 1&0\end{matrix}\right)\ ,
$$
makes sense as a Higgs field, and it is clearly holomorphic.  While $\Ecal$ is unstable as a holomorphic vector bundle
the Higgs bundle $(\Ecal,\Phi)$ is stable, since the only $\Phi$-invariant sub-line bundle is $\Kcal^{-1/2}$ which has negative degree.
Let us remark in passing that if we consider a different holomorphic structure $\Vcal$
 on $E$ given by the $\dbar$-operator 
 $$\dbar_E+\Phi^\ast=\left(\begin{matrix} \dbar_{\Kcal^{1/2}} & \omega\\0& \dbar_{\Kcal^{-1/2}}\end{matrix}\right)\ ,
 $$
 then $\Vcal$ is the unique (up to isomorphism) non-split extension
 $$
 0\lra\Kcal^{1/2}\lra\Vcal\lra\Kcal^{-1/2}\lra 0\ .
 $$
We now compute the tangent space $\Mfrak_D^{(2)}$ at $[(\Ecal,\Phi)]$.  Write 
$$
\beta=\left(\begin{matrix} b&b_1\\ b_2&-b\end{matrix}\right)\qquad ,\qquad
\varphi=\left(\begin{matrix} \phi&\phi_1\\ \phi_2&-\phi\end{matrix}\right)\ ,
$$
and compute
$$
[\Phi,\beta]=\left(\begin{matrix} -b_1&0\\ 2b&b_1\end{matrix}\right) \qquad ,\qquad
\sqrt{-1}\Lambda[\Phi^\ast,\varphi]=\left(\begin{matrix} \phi_2&-2\phi\\ 0&-\phi_2\end{matrix}\right)\ .
$$
Then the conditions \eqref{eqn:h1} that $(\beta,\varphi)$ define a tangent vector are
$$
\dbar_E\varphi=\left(\begin{matrix} b_1&0\\ -2b&-b_1\end{matrix}\right) \qquad ,\qquad
\dbar_E^\ast\beta=\left(\begin{matrix} \phi_2&-2\phi\\ 0&-\phi_2\end{matrix}\right)\ .
$$
In particular, $\phi_1\in H^0(X,\Kcal^2)$ and $b_2\in H_{\dbar}^{0,1}(X,K^\ast)\simeq H^0(X,\Kcal^2)^\ast$.  I claim that the other entries vanish.  Indeed, the equations for $\phi$ and $b_1$ are
$
\dbar\phi=b_1
$, and  $\dbar^\ast b_1=-2\phi$.  But this implies $(\dbar^\ast\dbar+2)\phi=0$. Hence, $\phi$, and therefore also $b_1$, must vanish. The same argument works for $\phi_2$ and $b$.  We therefore have an isomorphism
$$
T_{[\Ecal_F,\Phi_F]}\Mfrak_D^{(2)}\simeq H^0(X,\Kcal^2)\oplus (H^0(X,\Kcal^2))^\ast\ .
$$
For $n\geq 2$, there is a similar argument.  Here we take
$$
\Ecal_F=\Kcal^{(n-1)/2}\oplus \Kcal^{(n-3)/2}\oplus \cdots\oplus \Kcal^{-(n-1)/2}\ ,
$$
and 
$$
\Phi_F=\left(\begin{matrix}  0 & 0 & 0 & \cdots & 0 \\
1 & 0 & 0 & \cdots & \vdots \\
0&1&0&\cdots&\vdots \\
\vdots && \ddots &\ddots & \vdots \\
0&\cdots &0&1&0
   \end{matrix}\right)\ .
$$
Notice that with respect to this splitting the $(ij)$ entry of $\varphi$ is a section of $\Kcal^{j-i+1}$,
 and the $(ij)$ entry of $\beta$ is in $\Omega^{0,1}(X,K^{j-i})$. 
We obtain the following equations on the entries of a tangent vector $(\beta,\varphi)$,
\begin{align}
\begin{split} \label{eqn:phi-beta}
\dbar_E\varphi_{ij} & = \beta_{i-1,j}-\beta_{i,j+1}\ ; \\
\dbar_E^\ast\beta_{ij} & = \varphi_{i,j-1}-\varphi_{i+1,j} \ ,
\end{split}
\end{align}
where it is understood that terms with indices $\leq 0$ or $\geq n+1$ are set to zero.
Upon further differentiation as in the $n=2$ case, we find
\begin{align}
\begin{split} \label{eqn:L}
(L-\delta_{i1}-\delta_{jn})\varphi_{ij} & =\varphi_{i+1,j+1}+\varphi_{i-1,j-1} \ ; \\
(\widetilde L-\delta_{in}-\delta_{j1})\beta_{ij} & =  \beta_{i+1,j+1}+\beta_{i-1,j-1}\ ,
\end{split}
\end{align}
where $L=\dbar_E^\ast\dbar_E+2$ and $\widetilde L=\dbar_E\dbar_E^\ast+2$.
I claim that $\varphi_{ij}=0$ (resp.\ $\beta_{ij}=0$) for $i\geq j$ (resp.\ $i\leq j$). 
For example, by  \eqref{eqn:L}, $L\varphi_{n1}=0$, and
since $L$ is a positive operator,  $\varphi_{n1}$ vanishes.
 More generally, fix $ 0\leq p\leq n-2$. Then for $0\leq \ell\leq n-p-1$, there are polynomials $P_\ell$ such that
 \begin{equation} \label{eqn:kl}
 \varphi_{p+\ell+1, \ell+1}=P_\ell(L)\varphi_{p+1,1}\ .
 \end{equation}
 Indeed, let $P_0(L)=1$, $P_1(L)=L$ if $p\neq 0$ and  $P_1(L)=L-1$ if $p=0$. Suppose $P_k(L)$ has been defined for $0\leq k\leq \ell$,  where $0<\ell< n-p-1$. Use \eqref{eqn:L} and \eqref{eqn:kl} to find:
 \begin{align*}
 L\varphi_{p+\ell+1,\ell+1}&= \varphi_{p+\ell+2, \ell+2}+ \varphi_{p+\ell, \ell} \\
 LP_\ell(L) \varphi_{p+1, 1}&= \varphi_{p+\ell+2, \ell+2}+P_{\ell-1}(L) \varphi_{p+1,1}\ .
 \end{align*}
 Hence, we let $P_{\ell+1}(L)= LP_\ell(L)-P_{\ell-1}(L)$.    
  Since $L\geq 2$, we see from the  recursive definition that $P_{\ell+1}(L)\geq P_{\ell}(L)$, and
  hence for  all $\ell\geq 1$, $P_{\ell}(L)\geq P_1(L)\geq1$,  and $\geq 2$ if $p\neq 0$. 
 Taking $\ell=n-p-1$ in \eqref{eqn:kl}, we have
 \begin{equation} \label{eqn:kn}
 \varphi_{n,n-p}=P_{n-p-1}(L)\varphi_{p+1,1}\ .
 \end{equation}
 On the other, a similar argument implies
 $
 \varphi_{p+1,1}=P_{n-p-1}(L)\varphi_{n,n-p}
 $,
 from which we obtain
 $$
0= (P^2_{n-p-1}(L)-1)\varphi_{p+1, 1}=(P_{n-p-1}(L)+1)(P_{n-p-1}(L)-1)\varphi_{p+1, 1}\ .
$$
Hence, $\varphi_{p+1,1}$ is in the kernel of $P_{n-p-1}(L)-1$.  But then by the remark above, for $p\geq 1$, $\varphi_{p+1,1}$ must vanish.
Since $p\geq 1$ is arbitrary, this implies by  \eqref{eqn:kl} that $\varphi_{ij}=0$ for all $i>j$.  
 In the case $p=0$, notice that for all $\ell\geq 1$, $P_\ell(L)$ is a polynomial of positive degree in $\dbar_E^\ast\dbar_E$ with nonnegative coefficients and  constant term $=1$. 
 Indeed, by the definition
 $$
 P_{\ell+1}(L)-P_\ell(L)=(\dbar_E^\ast\dbar_E)P_\ell(L)+ P_{\ell}(L)-P_{\ell-1}(L)\ ,
 $$
 and so by induction $P_{\ell+1}(L)-P_\ell(L)$ has nonnegative coefficients and zero constant term.
 In this case, $(P_{n-1}(L)-1)\varphi_{1,1}=0$ implies that $\varphi_{1,1}$ is holomorphic.  Using \eqref{eqn:kl} again,
 $$
  \varphi_{\ell+1, \ell+1}=P_\ell(L)\varphi_{1,1}=(P_\ell(L)-1)\varphi_{1,1}+ \varphi_{1,1}=\varphi_{1,1}\ ,
 $$
 for all $\ell=0,\ldots, n-1$. But since $(\varphi_{ij})$ is traceless, it follows that in fact $\varphi_{ii}=0$ for all $i$.
The proof for $\beta_{ij}$ is exactly similar.

Going back to \eqref{eqn:phi-beta}, we see that $\varphi_{ij}$ (resp.\ $\beta_{ji}$) is holomorphic  (resp.\ harmonic) if $i<j$. Moreover, for $p\geq 1$,  \eqref{eqn:L} becomes
\begin{equation} \label{eqn:P}
(2-\delta_{i1}-\delta_{i n-p})\varphi_{i, i+p}=\varphi_{i+1, i+1+p}+\varphi_{i-1, i-1+p}\ .
\end{equation}
If $i=1$ this implies $\varphi_{1, p+1}=\varphi_{2, p+2}$. Suppose by induction that $\varphi_{k, k+p}=\varphi_{1,p+1}$ for all $k\leq i$.  Then if $i+p\neq n$, 
\eqref{eqn:P} implies
$$
2\varphi_{i,i+p}=\varphi_{i+1, i+1+p}+\varphi_{i-1,i-1+p} \quad \Longrightarrow\quad \varphi_{1,p+1}=\varphi_{i+1, i+1+p}\ .
$$
If $i+p=n$, we immediately get $\varphi_{in}=\varphi_{i-1, n-1}=\varphi_{1,p+1}$.  Hence, all differentials $\varphi_{ij}$, $j-i=p$, are equal. The same argument applies to $\beta_{ij}$.
From this
we conclude that the map $(\varphi,\beta)\mapsto (\varphi_{12},\ldots, \varphi_{1n}, \beta_{21}, \ldots, \beta_{n1})$
gives an isomorphism
\begin{equation} \label{eqn:higgs-fuchs}
T_{[\Ecal_F,\Phi_F]}\Mfrak_D^{(n)}\simeq \bigoplus_{j=2}^n H^0(X,\Kcal^j)\oplus (H^0(X,\Kcal^j))^\ast\ .
\end{equation}

The rank $n$ holomorphic vector bundle 
$\Vcal$ whose $\dbar$-operator is
$\dbar_E+\Phi^\ast_{F}$ is unstable and 
has a Harder-Narasimhan  filtration $0=\Vcal_0\subset\Vcal_1\subset\cdots\subset \Vcal_n=\Vcal$, $\Vcal_{j+1}/\Vcal_j=\Kcal^{-j+(n-1)/2}$, such that 
$$
 0\lra\Vcal_j\lra\Vcal_{j+1}\lra\Kcal^{-j+(n-1)/2}\lra 0\ .
 $$
 is the  (unique) non-split extension.
This is an example of an \emph{oper}. Opers will be discussed
in Section \ref{sec:opers}.
\end{example}

\subsubsection{The Hitchin map}
Given a Higgs bundle $(\Ecal,\Phi)$, the coefficient of
$\lambda^{n-i}$ in the expansion of $\det(\lambda+\Phi)$ is a
holomorphic section of $\Kcal^i$, $i=1,\ldots, n$.  In the case of fixed
determinant that we will mostly be considering, $\tr \Phi=0$,
so the sections start with $i=2$.
These pluricanonical sections are clearly invariant under the
action (by conjugation) of $\Gcal_E^\CBbb$, so we have a
well-defined map, called the {\bf Hitchin map},
\begin{equation} \label{eqn:hitchin-map}
h: \Mfrak_D^{(n)}\lra \bigoplus_{i=2}^n H^0(X,\Kcal^i) \ .
\end{equation}
The structure of this map and its fibers turns out be
extremely rich (cf.\ \cite{Hitchin87b}).  
In these notes, however, I will only discuss 
the following important fact which will be proven in the next 
section using Uhlenbeck compactness  
(for algebraic proofs, see \cite{Nitsure91, Simpson92}).
\begin{theorem} \label{thm:hitchin-proper}
The Hitchin map is proper.
\end{theorem}

\subsection{The Hitchin-Kobayashi correspondence} \label{sec:hitchin}

\subsubsection{Stability and critical metrics}

{\bf Hitchin's equations} for Higgs bundles on a trivial bundle are 
\begin{equation} \label{eqn:hitchin}
F_A+[\Phi,\Phi^\ast]=0\ .
\end{equation}
Here, $\Phi$ is regarded as an endomorphism valued $(1,0)$-form.
It will also be convenient to consider the case of bundles of nonzero degree.  In this case the equations become
\begin{equation} \label{eqn:hitchin2}
f_{(A,\Phi)}:=
\sqrt{-1}\Lambda(F_A+[\Phi,\Phi^\ast])=\mu\ .
\end{equation}
Here we recall the  normalization  $\vol(X)=2\pi$, and
 then on right hand side the scalar multiple of the identity 
endomorphism necessarily satisfies $\mu=\mu(E)$.

There are two ways of thinking of \eqref{eqn:hitchin2}: for a
Higgs bundle $(\Ecal, \Phi)$ a choice of hermitian
metric gives a Chern connection $A=(\dbar_E, H)$.  
Hence, we may either view \eqref{eqn:hitchin2} as an equation
for a hermitian metric $H$, or alternatively (and
equivalently) we may fix $H$ and consider 
$f_{(A,\Phi)}$ for all $(A,\Phi)$ in a
complex gauge orbit.  We will often go
back and forth between these equivalent points of view.

The solutions to the  equations \eqref{eqn:hitchin2} may be regarded 
as the absolute minimum for the {\bf Yang-Mills-Higgs functional}
 on the space of holomorphic pairs, defined as
\begin{equation} \label{eqn:ymh}
\YMH(A,\Phi)=\int_X \left|F_A+[\Phi,\Phi^\ast]\right|^2 \,\omega\ .
\end{equation}
The Euler-Lagrange equations for $\YMH$ are 
\begin{equation} \label{eqn:critical}
d_A f_{(A,\Phi)}=0\ ,\ [\Phi, f_{(A,\Phi)}]=0\ .
\end{equation}
 We call a metric {\bf critical} if  \eqref{eqn:critical} is
satisfied.
In this case, it is easy to see the bundle $(\Ecal,\Phi)$
splits holomorphically and isometrically as a direct 
sum of  Higgs bundles that are 
solutions to \eqref{eqn:hitchin2} with possibly
different slopes.

\begin{proposition} \label{prop:easy}
If a Higgs bundle $(\Ecal, \Phi)$ admits a  metric satisfying
\eqref{eqn:hitchin2},  then $(\Ecal,\Phi)$ is polystable.
\end{proposition}

\begin{proof}
Let $\Scal\subset\Ecal$ be a proper $\Phi$-invariant subbundle. Let $\pi$ denote the orthogonal projection to $S$ and $\beta=-\dbar_E \pi$ the second fundamental form.
Then since $\Scal$ is invariant, $(I-\pi)\Phi\pi=0$, or
$$
\Phi\pi = \pi\Phi\pi \ ,\
\pi\Phi^\ast = \pi\Phi^\ast\pi\ .
$$
In particular, this implies
\begin{align}
\tr (\pi[\Phi, \Phi^\ast] )&= \tr (\pi\Phi\Phi^\ast )+ \tr(\pi \Phi^\ast\Phi) \notag \\
&= \tr( \pi\Phi\Phi^\ast )- \tr (\Phi\pi\Phi^\ast )\notag \\
&= \tr (\pi\Phi\Phi^\ast\pi) - \tr (\Phi\pi\Phi^\ast\pi) \notag \\
&= \tr (\pi\Phi\Phi^\ast\pi) - \tr(\pi \Phi\pi\Phi^\ast\pi)\notag \\
&= \tr (\pi\Phi(I-\pi)\Phi^\ast \pi)= \tr (\pi\Phi(I-\pi)(I-\pi)\Phi^\ast \pi) \notag \\
&= \tr \left\{ (\pi\Phi(I-\pi))(\pi\Phi(I-\pi))^\ast\right\} \ ;\notag \\
\tr(\pi\sqrt{-1}\Lambda [\Phi, \Phi^\ast]) &=|\pi\Phi(I-\pi)|^2 \label{eqn:pi-phi}\ .
\end{align}
Plugging \eqref{eqn:hitchin2} into \eqref{eqn:degree}, and using \eqref{eqn:pi-phi}, we have
 $$
\deg\Scal=\rank(\Scal)\mu(\Ecal)-\frac{1}{2\pi}( \Vert \pi\Phi(I-\pi)\Vert^2 + \Vert \beta\Vert^2)\ .
$$
This proves $\mu(\Scal)\leq \mu(\Ecal)$.  Moreover, equality holds if and only if the two terms on the right hand side above vanish; i.e.\ the holomorphic structure and Higgs field split.
\end{proof}

The main result we prove in this section is the converse to Proposition \ref{prop:easy}.
\begin{theorem}[Hitchin \cite{Hitchin87a}, Simpson
\cite{Simpson88}] \label{thm:hitchin}
If $(\Ecal,\Phi)$ is polystable, then it admits a  metric
satisfying \eqref{eqn:hitchin2}.
\end{theorem}

\begin{remark} \label{rem:rank-one}
The result is straightforward in the case of line bundles $\Lcal$. Indeed, in rank $1$ the term $[\Phi,\Phi^\ast]$ vanishes, so \eqref{eqn:hitchin2} amounts to finding a constant curvature metric on $L$.  If $H$ is any metric, let $H_\varphi=e^\varphi H$ for a function $\varphi$.  Then
$
F_{(\dbar_L, H_\varphi)}=F_{(\dbar_L, H)}+\dbar\partial \varphi
$, and the problem is solved if we can find $\varphi$ such that
$$
\Delta\varphi=2\sqrt{-1}\Lambda(F_{(\dbar_L, H)})- 2\deg(L)\ .
$$
By the Hodge theorem the only condition to finding a solution to this equation is that the integral of the right hand side vanish (cf.\ \cite[p.\ 84]{GriffithsHarris78}), which it does by \eqref{eqn:chern-weil}.
\end{remark}

In order to prove Theorem \ref{thm:hitchin} in higher rank, it will be important to construct approximate critical metrics.
Let $0\subset (\Ecal_1,\Phi_1)\subset \cdots\subset
 (\Ecal_\ell, \Phi_\ell)=(\Ecal, \Phi)$
 be the Harder-Narasimhan filtration of the Higgs bundle 
$(\Ecal,\Phi)$.  We let $\Qcal_i=\Ecal_i/\Ecal_{i-1}$ and $\mu_i=\mu(\Qcal_i)$.  Then there is a smooth splitting $E=\bigoplus_i Q_i$, and given a hermitian metric $H$  we can make this splitting orthogonal.  Hence, there is a well-defined endomorphism
\begin{equation} \label{eqn:mu-gr}
\mu_{(\Gr(\Ecal,\Phi), H)}=\left(\begin{matrix} \mu_1 && \\ & \ddots & \\ && \mu_\ell \end{matrix}\right)\ .
\end{equation}
where the blocks $\mu_i$ have dimensions $\rank Q_i$.

\begin{definition} \label{def:approximate-critical}
We say that a metric on $(\Ecal,\Phi)$ is $\varepsilon$-approximate critical if
$$
\sup \left| f_{((\dbar_E,H),\Phi)}-\mu_{(\Gr(\Ecal,\Phi), H)}\right| < \varepsilon\ .
$$
\end{definition}

Note that the $\dbar$-operator for $\Ecal$ may be written in an upper triangular form with respect to this splitting, and the strictly upper triangular piece is determined by the extension classes.  By acting with a complex gauge transformation that is block diagonal, the extension classes  may be made arbitrarily small.  If moreover the bundles $\Qcal_i$ with their induced Higgs fields admit Hermitian-Yang-Mills-Higgs connections, then we can sum these up and obtain the following (for more details, see \cite{DaskalWentworth04}).

\begin{lemma} \label{lem:approximate-critical}
Let $(\Ecal,\Phi)$ be an unstable Higgs bundle of rank $n$,
 and suppose that Theorem \ref{thm:hitchin} has been proven 
for Higgs bundles of rank less than $n$.  
Then for any $\varepsilon>0$ there is an 
$\varepsilon$-approximate critical metric on $(\Ecal,\Phi)$.
\end{lemma}

\subsubsection{Preliminary estimates}

Recall the map \eqref{eqn:hitchin-map}.
A crucial point is the following a priori estimate.
\begin{proposition} \label{prop:higgs-bound}
Let $(\Ecal, \Phi)$ be a Higgs bundle.
There are  constants $C_1, C_2 >0$  depending only  on the metrics on $X$ and $E$, and on
$\Vert h[\Ecal,\Phi]\Vert$, such that
$$\sup |\Phi|^2\leq
 C_1 +
 C_2\, \sup\left| \sqrt{-1}\Lambda(F_A+[\Phi,\Phi^\ast])\right|\ .$$ 
\end{proposition}

We need the  following
\begin{lemma}[{cf.\ \cite[p.\ 27]{Simpson92}}] \label{lem:simpson}
For a matrix $P$ 
there are constants $C_1, C_2>0$ 
depending only on the eigenvalues of $P$ such that
$$
|[P,P^\ast]|^2\geq C_1|P|^4-C_2(1+|P|^2)\ .
$$
\end{lemma}

\begin{proof}
Choose a unitary basis such that $P=S+N$, where $S$ is diagonal and $N$ is strictly upper triangular. By assumption, $|S|$ is bounded.  It is easy to see that it then suffices to show there is $C>0$ such  that for all strictly upper triangular $N$,  $|[N,N^\ast]|\geq C|N|^2$.  Suppose not.  Then by scaling we can find a sequence $N_j$, $|N_j|=1$, and $[N_j, N_j^\ast]\to 0$.  After passing to a subsequence, we may assume $N_j\to N$, with $[N,N^\ast]=0$, $|N|=1$. But this is a contradiction.  Indeed, if $a_1,\ldots, a_n$ and $b_1, \ldots, b_n$ are the rows and columns of $N$, then reading off the diagonal of $NN^\ast=N^\ast N$ implies $|a_i|^2=|b_i|^2$ for $i=1,\ldots, n$.  But $b_1=0$, which from this equality implies $a_1=0$.  This in turn implies $b_2=0$, and hence $a_2=0$. Continuing in this way, we conclude $N=0$; contradiction.
\end{proof}

We will also need the following computation.
\begin{align}
[[P,P^\ast],P] &= (PP^\ast-P^\ast P)P-P(PP^\ast-P^\ast P)\notag \\
&=2PP^\ast P-P^\ast P^2 -P^2 P^\ast \notag\\
\langle [[P,P^\ast],P], P\rangle =\tr ([[P,P^\ast],P]P^\ast) &=
\tr((2PP^\ast P-P^\ast P^2 -P^2 P^\ast)P^\ast)\notag \\
&=2\tr(PP^\ast)^2-2\tr( P^2 (P^\ast)^2 ) \notag\\
\langle \ad([P,P^\ast])P, P\rangle&=|[P,P^\ast]|^2 \label{eqn:ad}\ .
\end{align}

\begin{proof}[Proof of Proposition \ref{prop:higgs-bound}]
Regard $\Phi$ as a holomorphic section of $\Kcal\otimes\End \Ecal$.  We also make use of three easy facts.  First, if $H$ is a hermitian metric on $E$ and $\widehat H$ is the induced metric on $\End E$, then $F_{(\End \Ecal,\widehat H)}= \ad F_{(\Ecal, H)}$, where the adjoint indicates that the curvature endomorphism acts by commutation.
Second, if $\widehat H, h$ are hermitian metrics on $\End E$ and $K$, respectively, then 
\begin{equation} \label{eqn:end-curvature}
F_{(\Kcal\otimes\End \Ecal, h\otimes \widehat H)}= F_{(\End \Ecal,\widehat H)}+ F_{(\Kcal, h)}\cdot I\ .
\end{equation}
Third, if $s$ is a holomorphic section of a vector bundle with unitary connection $A$ and curvature $F_A$, then we have the following Weitzenb\"ock formula:
\begin{equation} \label{eqn:weitzenbock}
\Delta|s|^2= 2|d_A s|^2 -2\langle \sqrt{-1}\Lambda F_A s,s\rangle\ .
\end{equation}
Indeed (cf.\ \eqref{eqn:kahler}),
\begin{align*}
\Delta|s|^2&=-2\dbar^\ast\dbar|s|^2=
2\sqrt{-1}\Lambda\partial\bar\partial |s|^2=2\sqrt{-1}\Lambda\partial\langle s, d'_A s\rangle \\
&=2\sqrt{-1}\Lambda\langle d'_A s, d'_A s\rangle +2\sqrt{-1}\Lambda\langle s, d''_Ad'_A s\rangle \\
&=2|d'_A s|^2 +2\sqrt{-1}\Lambda\langle s, F_A s\rangle \\
&= 2|d_A s|^2 -2\langle \sqrt{-1}\Lambda F_A s,s\rangle \ .
\end{align*}
Now using eqs. \eqref{eqn:ad}, \eqref{eqn:end-curvature}, and \eqref{eqn:weitzenbock}, 
along with Lemma \ref{lem:simpson},
we have
\begin{align*}
\Delta |\Phi|^2&\geq - 2\langle \sqrt{-1}\Lambda F_{(\Kcal\otimes\End \Ecal, h\otimes\widehat H)}\Phi,\Phi\rangle \\
&\geq -2\langle \sqrt{-1}\Lambda F_{(\End \Ecal, \widehat H)}\Phi,\Phi\rangle- C_3|\Phi|^2 \\
&=-2\langle \ad (\sqrt{-1}\Lambda F_{(\Ecal, H)})\Phi,\Phi\rangle- C_3|\Phi|^2 \\
&=2\langle \ad (\sqrt{-1}\Lambda [\Phi,\Phi^\ast])\Phi,\Phi\rangle 
%&\qquad \qquad
-2\langle \ad (\sqrt{-1}\Lambda (F_{(\Ecal, H)}+[\Phi,\Phi^\ast])\Phi,\Phi\rangle- C_3|\Phi|^2 \\
&\geq C_1|\Phi|^4-C_2(1+|\Phi|^2)-C_4\sup\left|\sqrt{-1}\Lambda (F_{(\Ecal, H)}+[\Phi,\Phi^\ast]\right||\Phi|^2\ .
\end{align*}
Now at a maximum of $|\Phi|^2$ the left hand side is nonpositive.  Since $C_1>0$, the proposition follows immediately.
\end{proof}

\begin{remark}
Notice that the sign in \eqref{eqn:hitchin} is decisive for this argument (cf.\ \cite{Hitchin90}).
\end{remark}

Finally, the existence proof will be based on Donaldson's elegant argument in \cite{Donaldson83}. This requires the introduction of the functional $J=J(A,\Phi)$, defined as follows.
For a hermitian endomorphism $\phi$ of $E$, let 
$$
\nu(\phi)=\sum_{i=1}^n|\lambda_i| \quad ,\quad
N^2(\phi)=\int_X \nu^2(\phi)\, \frac{\omega}{2\pi}\ ,
$$
where the $\lambda_i$ are the (pointwise) eigenvalues of $\phi$. Then we define
\begin{equation} \label{eqn:J}
J(A,\Phi)=N(f_{(A,\Phi)}-\mu(E))\ .
\end{equation}

We next prove the following two results of Donaldson (see \cite[Lemmas 2 $\&$ 3]{Donaldson83}), adapted here to the case of Higgs bundles.

\begin{lemma} \label{lem:donaldson-upper-bound}
Let $(A,\Phi)$ be a Higgs bundle with underlying bundle $\Ecal$.
Suppose it fits into an extension of Higgs bundles  $0\to\Mcal\to\Ecal\to\Ncal\to 0$, and that
$\mu(\Ncal)\leq \mu(\Ecal)\leq \mu(\Mcal)$.  Then
$$
(\rank \Mcal)(\mu(\Mcal)-\mu(\Ecal))+(\rank\Ncal )(\mu(\Ecal)-\mu(\Ncal))\leq J(A,\Phi)\ .
$$
\end{lemma}

\begin{proof}
With respect to the orthogonal splitting $E=M\oplus N$, and letting $F_E$, $F_M$, and $F_N$ denote the curvature and induced curvatures of the Chern connection for $(\Ecal, H)$,
 we have
$$
\sqrt{-1}\Lambda F_E=\left(
\begin{matrix}
\sqrt{-1}\Lambda F_M+b_M & -(d''_A)^\ast\beta \\
-((d''_A)^\ast\beta)^\ast &\sqrt{-1}\Lambda F_N+b_N
\end{matrix}
\right)\ ,
$$
where $\beta$ is the second fundamental form, and 
$$
b_M=-\sqrt{-1}\Lambda(\beta\wedge\beta^\ast)\ ,\ b_N=-\sqrt{-1}\Lambda(\beta^\ast\wedge\beta)\ .
$$
Notice that $\tr b_M=-\tr b_N=|\beta|^2$. Similarly, if we write $\displaystyle \Phi=\left(\begin{matrix} \Phi_M &\varphi\\ 0 &\Phi_N\end{matrix}\right)$, then
$$
[\Phi,\Phi^\ast]=\left(\begin{matrix} [\Phi_M,\Phi_M^\ast]+\varphi\wedge\varphi^\ast & \varphi\wedge\Phi_N^\ast+\Phi_M^\ast\wedge\varphi \\ \Phi_N\wedge\varphi^\ast+\varphi^\ast\wedge\Phi_M
&  [\Phi_N,\Phi_N^\ast]+\varphi^\ast\wedge\varphi
\end{matrix}
\right)\ .
$$
It follows that
$$
f_{(A,\Phi)}=\left(\begin{matrix} f_M+b_M+\sqrt{-1}\Lambda\varphi\wedge\varphi^\ast & \cdots \\
\dots & f_N+b_N+\sqrt{-1}\Lambda\varphi^\ast\wedge\varphi
\end{matrix}
\right)\ .
$$
Hence, (cf.\ \cite[p.\ 271]{Donaldson83}),
\begin{align*}
\nu(f_{(A,\Phi)}-\mu(E)) &\geq  \left| \tr(\sqrt{-1}\Lambda F_M)-(\rank \Mcal)\mu(\Ecal)+|\beta|^2+|\varphi|^2 \right| \\
&\qquad +\left| \tr(\sqrt{-1}\Lambda F_N)-(\rank \Ncal)\mu(\Ecal)-|\beta|^2-|\varphi|^2\right| \ ,
\end{align*}
and therefore
\begin{align*}
J(A,\Phi)&\geq \int_X\nu(f_{(A,\Phi)}-\mu(\Ecal)) \frac{\omega}{2\pi} \\
&\geq  \left|\int_X\left( \tr(\sqrt{-1}\Lambda F_M)-(\rank \Mcal)\mu(\Ecal)+|\beta|^2+|\varphi|^2\right) \frac{\omega}{2\pi} \right| \\
&\qquad +\left| \int_X \left(\tr(\sqrt{-1}\Lambda F_N)-(\rank \Ncal)\mu(\Ecal)-|\beta|^2-|\varphi|^2\right)\frac{\omega}{2\pi}\right| \\
&\geq (\rank \Mcal)(\mu(\Mcal)-\mu(\Ecal)) + (\rank \Ncal)(\mu(\Ecal)-\mu(\Ncal))\ .
\end{align*}
\end{proof}

\begin{lemma} \label{lem:donaldson-lower-bound}
Let $(A_0,\Phi_0)$ be a stable Higgs bundle of rank $n$ that 
 fits into an extension of Higgs bundles 
 $0\to\Scal\to\Ecal\to\Qcal\to 0$. 
Assume Theorem \ref{thm:hitchin} has been proven for 
Higgs bundles of rank less than $n$.
 Then we can choose a point $(A,\Phi)$ in the complex gauge orbit of $(A_0,\Phi_0)$ such that
 $$
 J(A,\Phi)< (\rank \Scal)(\mu(\Ecal)-\mu(\Scal))+(\rank\Qcal)(\mu(\Qcal)-\mu(\Ecal))\ .
 $$
\end{lemma}

\begin{proof}
First, consider the Harder-Narasimhan filtrations of  $(\Scal,\Phi_S)$ and $(\Qcal, \Phi_Q)$.  By 
applying Lemma \ref{lem:approximate-critical} we may 
assume for any $\varepsilon>0$ that there is a  metric on $S$  such that
$$
\sup \left| f_{((\dbar_S,H_S),\Phi_S)}-\mu_{(\Gr(\Scal,\Phi_S), H_S)}\right| < \varepsilon\ ,
$$
and similarly for $Q$.  
We endow $E=S\oplus Q$ with the sum of these two metrics.
This is equivalent to a pair $(A,\Phi)$ in the orbit of
$(A_0,\Phi_0)$.
Next, since $(A_0, \Phi_0)$ (and hence also $(A,\Phi)$) 
is simple we may further assume that 
$$-\dbar_{A_0}^\ast\beta +\sqrt{-1}\Lambda \left( \varphi\wedge
\Phi_Q^\ast+\Phi_S^\ast\wedge\varphi\right)=0
$$
(see \eqref{eqn:h1}).
This is accomplished via a complex gauge transformation of the 
form $\displaystyle g=\left(\begin{matrix} 1&\phi\\0&1\end{matrix}\right)$. In particular, the $\dbar$-operators on $S$ and 
$Q$ remain unchanged, and so the approximate critical 
structure still holds.
With this understood, we perform a further gauge
transformation so that $(A,\Phi)$ coincides with
$(A_0,\Phi_0)$ but with $\beta$ and $\varphi$ scaled by $t$.
Then $f_{(A,\Phi)}-\mu(E)$ is block diagonal with entries
\begin{align}
\begin{split} \label{eqn:donaldson}
& f_S-\mu_{(\Gr(\Scal,\Phi_S), H_S)}
+\mu_{(\Gr(\Scal,\Phi_S), H_S)}-\mu(E)
+t^2\left(b_S+\sqrt{-1}\Lambda\varphi\wedge\varphi^\ast \right) \ ;\\
 & f_Q-\mu_{(\Gr(\Qcal,\Phi_Q), H_Q)}
+\mu_{(\Gr(\Qcal,\Phi_Q), H_Q)}-\mu(E)
+t^2\left(b_Q+\sqrt{-1}\Lambda\varphi^\ast\wedge\varphi\right)\ .
\end{split}
\end{align}
Since $(\Ecal,\Phi)$ is stable, $\mu(\Ecal)$ is strictly bigger than the maximal slope of a subsheaf of $\Scal$, and strictly smaller than the minimal slope of a quotient of $\Qcal$.
This says that for $t$ and $\varepsilon$ 
chosen sufficiently small, the first line in \eqref{eqn:donaldson} is negative definite and the second is positive definite.  It follows that
$$
\nu(f_{(A,\Phi)}-\mu(E))
\leq (\rank\Scal)(\mu(E) -\mu(S))+(\rank\Qcal)(\mu(\Qcal)-\mu(\Ecal))
-2t^2\left( |\beta|^2+|\varphi|^2\right)+O(\varepsilon)\ .
$$
Without loss of generality, assume that
$\Vert\beta\Vert^2+\Vert\varphi\Vert^2=1$. By the argument in
\cite{Donaldson83} we may also assume $|\beta|$, $|\varphi|$ are
bounded uniformly in $\varepsilon$.
The result now follows by fixing $t$ and choosing
$\varepsilon$ sufficiently small.
\end{proof}

\subsubsection{The existence theorem}

We will prove the following in the next section where the Yang-Mills-Higgs flow will be introduced.
\begin{lemma} \label{lem:minimizing}
In any complex gauge orbit 
there exists a  sequence $(A_i, \Phi_i)$ satisfying the following conditions:
\begin{enumerate}
\item $(A_i, \Phi_i)$  is minimizing for $J$ ;
\item if
 $f_{(A_j,\Phi_j)}= 
\sqrt{-1}\Lambda(F_{A_j}+[\Phi_j,\Phi_j^\ast])$, then $\sup | f_{(A_j,\Phi_j)}|$ is   bounded  uniformly in $j$;
\item $\Vert d_{A_j} f_{(A_j,\Phi_j)}\Vert_{L^2}\to 0$ and $\Vert [f_{(A_j,\Phi_j)}, \Phi_i]\Vert_{L^2}\to 0$.
\end{enumerate}
\end{lemma}

Next, we will need one of the most fundamental results of gauge theory, stated here for the case of Riemann surfaces.

\begin{proposition}[Uhlenbeck {\cite{Uhlenbeck82}}] \label{prop:uhlenbeck}
Fix $p\geq 2$.
Let $\{A_j\}$ be a sequence of $L^p_1$-connections with 
$\Vert F_{A_j}\Vert_{L^p}$ uniformly bounded. 
Then there exists a sequence of unitary gauge 
transformations $g_j\in L^p_2$ and a 
smooth unitary connection $A_\infty$ such that 
(after passing to a subsequence) 
$g_j(A_j)\to A_\infty$ weakly in $L^p_1$ and strongly in $L^p$.
\end{proposition}

Assuming these results, we now prove  the existence theorem.

\begin{proof}[Proof of Theorem \ref{thm:hitchin}]
It clearly suffices to assume $(\Ecal,\Phi)$ is stable.  Furthermore,
by Remark \ref{rem:rank-one}, we may proceed by induction.  Assume that the result has been proven for all bundles of rank $<n=\rank E$.

\medskip\noindent
{\bf Step 1.}  \emph{The limiting bundle $(\Ecal_\infty, \Phi_\infty)$}.
Choose a minimizing sequence for $J$ as in Lemma \ref{lem:minimizing}.
Since the sequence lies in a single complex gauge orbit, the
image of the Hitchin map $h[A_i, \Phi_i]$ is unchanged.
Hence, by Proposition \ref{prop:higgs-bound}  the $\Phi_i$ are uniformly
bounded.  By Lemma \ref{lem:minimizing} (ii), this in turn implies that $\Vert F_{A_j}\Vert_{L^p}$
is bounded for any $p$.  We therefore may assume by
Proposition \ref{prop:uhlenbeck} that there is a smooth
connection $A_\infty$ so that if we write
$\dbar_{A_j}=\dbar_{A_\infty}+a_j$, then $a_j\to 0$ weakly in
$L^p_1$.
By the Sobolev embedding theorem, we may assume in particular that
the $a_j\to 0$ in some $C^\alpha$.  Notice that it follows that $F_{A_j}\to F_{A_\infty}$ weakly in $L^p$.
 From the
holomorphicity condition
$$
0=\dbar_{A_j}\Phi_j=\dbar_{A_\infty}\Phi_j+[a_j,\Phi_j]\ .
$$
  Elliptic regularity for $\dbar_{A_\infty}$ implies a bound
$\Vert \Phi_j\Vert_{L^2_1}\leq C\Vert \Phi_j\Vert_{L^2}$, say. Differentiating the previous equation gives
\begin{equation} \label{eqn:phi1}
\dbar_{A_\infty}^\ast\dbar_{A_\infty}\Phi_j+\dbar_{A_\infty}^\ast[a_j,\Phi_j]=0
\end{equation}
By the Cauchy-Schwarz inequality and the previous estimate we have
\begin{equation} \label{eqn:phi2}
\Vert \dbar_{A_\infty}^\ast[a_j,\Phi_j]\Vert_{L^2}\leq C_1\Vert a_j\Vert_{L^4_1}\Vert \Phi_j\Vert_{L^4}
+C_2\Vert \Phi_j\Vert_{L^2}\ .
\end{equation}
Now we may assume $\{a_j\}$ is bounded in $L^4_1$, and using elliptic regularity for the Laplacian $\dbar_{A_\infty}^\ast\dbar_{A_\infty}$ along with
 the inclusions $L^2_1\hookrightarrow L^4$, $L^2_2\hookrightarrow C^\alpha$,
 by \eqref{eqn:phi1} and \eqref{eqn:phi2} we have an estimate
 $\Vert \Phi_j\Vert_{C^\alpha}\leq C\Vert \Phi_j\Vert_{L^2}$. 
Since the $\Phi_j$ are uniformly bounded their $L^2$ norms are bounded, so
 we may assume that $\Phi_j$ converges
in $C^\alpha$ to some $\Phi_\infty$.
Moreover, 
 by holomorphicity of the $\Phi_j$ we can write
 $$
 \dbar_{A_\infty}\Phi_\infty=\dbar_{A_\infty}(\Phi_\infty-\Phi_j)-[a_j, \Phi_j]\ ,
 $$
 and since $[a_j, \Phi_j]\to 0$ in $C^\alpha$ we see that
 $\dbar_{A_\infty} \Phi_\infty=0$ weakly.
 Hence, by Weyl's lemma $\Phi_\infty$ is actually  holomorphic, and
 thus $(\Ecal_\infty, \Phi_\infty)$ is a Higgs bundle.

\medskip\noindent
{\bf Step 2.} \emph{Construction of a nonzero map $\Ecal\to \Ecal_\infty$}.  Let $g_j$ be complex gauge transformations such that $g_j(A)=A_j$.  Holomorphicity of $g_j$ implies $\dbar_{A_\infty}g_j+[a_j,g_j]=0$.  By the exact same argument as in Step 1, 
%Since the $a_j$ are bounded, elliptic regularity for $\dbar_{A_\infty}$ implies a bound
%$\Vert g_j\Vert_{L^2_1}\leq C\Vert g_j\Vert_{L^2}$, say. Differentiating the previous equation gives
%$$
%\dbar_{A_\infty}^\ast\dbar_{A_\infty}g_j+\dbar_{A_\infty}^\ast[a_j,g_j]=0
%$$
%By the Cauchy-Schwarz inequality and the previous estimate we have
%$$
%\Vert \dbar_{A_\infty}^\ast[a_j,g_j]\Vert_{L^2}\leq C_1\Vert a_j\Vert_{L^4_1}\Vert g_j\Vert_{L^4}
%+C_2\Vert g_j\Vert_{L^2}
%$$
%Now we may assume $\{a_j\}$ is bounded in $L^4_1$, and using elliptic regularity for the Laplacian $\dbar_{A_\infty}^\ast\dbar_{A_\infty}$ along with
 %the inclusions $L^2_1\hookrightarrow L^4$, $L^2_2\hookrightarrow C^\alpha$, 
 we have an estimate
 $\Vert g_j\Vert_{C^\alpha}\leq C\Vert g_j\Vert_{L^2}$.  Now rescale $g_j$ so that $\Vert g_j\Vert_{L^2}=1$.  The $C^\alpha$-estimate above still holds for the rescaled map, so by compactness we may assume there is a continuous $g_\infty: \Ecal\to \Ecal_\infty$ such that $g_j\to g_\infty$ in $C^\alpha$.  Because of the normalization, we know that $g_\infty\not\equiv 0$.  
 Moreover, it follows as in Step 1 that $g_\infty$ is holomorphic.
% by holomorphicity of the $g_j$ we can write
% $$
% \dbar_{A_\infty}g_\infty=\dbar_{A_\infty}(g_\infty-g_j)-[a_j, g_j]
% $$
% and since $[a_j, g_j]\to 0$ in $C^\alpha$ we see that
% $\dbar_{A_\infty} g_\infty=0$ weakly.
% Hence, by Weyl's lemma $g_\infty$ is actually a holomorphic map.  
 Finally, by the $C^\alpha$ convergence of $g_j$ and $\Phi_j$ and the fact that $g_j\Phi=\Phi_j g_j$, we have
 $g_\infty\Phi=\Phi_\infty g_\infty$.
 
 \medskip\noindent
{\bf Step 3.}  \emph{The map $g_\infty$ is an isomorphism}.  Suppose  the contrary.  Let $\Scal=\ker g_\infty$ and $\Qcal=\Ecal/\Scal$.  Then $\Qcal$ is a subsheaf of $\Ecal_\infty$.  Let $\Mcal$ denote its saturation and $\Ncal=\Ecal_\infty/\Mcal$.  
Since $\Phi_\infty g_\infty=g_\infty\Phi$, the subbundle $\Scal$ is $\Phi$-invariant.  Similarly, $\Mcal$ is $\Phi_\infty$-invariant.  Also,
from the discussion in Section \ref{sec:stability}, we have
\begin{align} 
\begin{split} \label{eqn:slopes}
\mu(\Qcal)-\mu(\Ecal) &\leq \mu(\Mcal)-\mu(\Ecal) \ ;\\
\mu(\Ecal)-\mu(\Scal) &\leq \mu(\Ecal)-\mu(\Ncal)\ .
\end{split}
\end{align}
Then we have the following extensions of Higgs bundles (see \cite{Donaldson83}):
\begin{equation}
\begin{split}  \label{eqn:donaldson-diagram}
\xymatrix{
0\ar[r] & \Scal \ar[r] & \Ecal \ar[r] \ar[d]^{g_\infty} & \Qcal \ar[r]\ar[d] & 0 \\
0 & \Ncal \ar[l] & \Ecal_\infty \ar[l]  & \Mcal \ar[l]& 0 \ar[l]
}
\end{split}
\end{equation}
Applying Lemma \ref{lem:donaldson-upper-bound} to the bottom row of \eqref{eqn:donaldson-diagram}  and Lemma \ref{lem:donaldson-lower-bound} to the top row implies
\begin{align*}
(\rank\Mcal)( \mu(\Mcal)-\mu(\Ecal))+&(\rank\Ncal)( \mu(\Ecal)-\mu(\Ncal)
) \leq J(A_\infty,\Phi_\infty) \\
&\leq\lim_{j\to\infty} J(A_j,\Phi_j)= \inf J(A,\Phi) \\
&< (\rank\Scal)( \mu(\Ecal)-\mu(\Scal))+(\rank\Qcal)( \mu(\Qcal)-\mu(\Ecal))\ ,
\end{align*}
where for the second line we 
can use either the the lower semicontinuity of $J$ (see \cite{Donaldson83}) or the argument in \cite[Corollary 2.12 and Lemma 2.17]{DaskalWentworth04}.
But since $\rank\Mcal=\rank\Qcal$ and $\rank\Scal=\rank\Ncal$, this contradicts \eqref{eqn:slopes}.

\medskip\noindent
{\bf Step 4.}  \emph{Solution to Hitchin's equations}.
Finally, I  claim that the Higgs bundle $(A_\infty,\Phi_\infty)$
is a solution to \eqref{eqn:hitchin}.  Indeed, by the remark following eq.\ 
\eqref{eqn:critical} this follows if we can show
$d_{A_\infty}f_{(A_\infty,\Phi_\infty)}=0$ and
 $[f_{(A_\infty,\Phi_\infty)},\Phi_\infty]=0$.  The second fact holds, since $[f_{(A_j,\Phi_j)},\Phi_j]\to 0$ in $L^2$ by assumption, and $f_{(A_j,\Phi_j)}$ (resp.\ $\Phi_j$) converges weakly in $L^p$ (resp.\ $C^\alpha$).
 For the first claim, let $B$ be a test form. Then
  \begin{align*}
 \langle d_{A_\infty}f_{(A_\infty,\Phi_\infty)}, B\rangle_{L^2}&=
 \langle f_{(A_\infty,\Phi_\infty)}, d_{A_\infty}^\ast B\rangle_{L^2} \\
 &=
\lim_{j\to\infty}
 \langle f_{(A_j,\Phi_j)}, d_{A_j}^\ast B\rangle_{L^2} 
+
\lim_{j\to\infty}
 \int_X\tr\left\{ f_{(A_j,\Phi_j)}[a_j, B^\ast ] \right\}\\
 &=
\lim_{j\to\infty}
 \langle d_{A_j}f_{(A_j,\Phi_j)},  B\rangle_{L^2} 
+
\lim_{j\to\infty}
 \int_X\tr\left\{ f_{(A_j,\Phi_j)}[a_j, B^\ast ] \right\}\ .
 \end{align*}
 The first term vanishes since $\Vert d_{A_j}f_{(A_j,\Phi_j)}\Vert_{L^2}\to 0$, and the second term vanishes since $f_j$ is bounded and $a_j\to 0$ in $C^\alpha$.  Since $B$ is arbitrary, $d_{A_\infty}f_{(A_\infty,\Phi_\infty)}=0$, and 
this completes the proof.
\end{proof}

The same type of argument leads to the
\begin{proof}[Proof of Theorem \ref{thm:hitchin-proper}]
Let $[A_j,\Phi_j]$ be a sequence of polystable Higgs bundles with $h[A_j,\Phi_j]$ bounded. By Theorem \ref{thm:hitchin} we may assume $(A_j,\Phi_j)$ satisfies \eqref{eqn:hitchin}.  Since 
$h[A_j,\Phi_j]$ is  bounded, the pointwise spectrum of $\Phi_j$ is uniformly bounded.  Therefore, Proposition \ref{prop:higgs-bound} provides uniform sup bounds on $|\Phi_j|$.  Again using \eqref{eqn:hitchin} we have uniform bounds on $|F_{A_j}|$.  Now Uhlenbeck compactness can be used to extract a convergent subsequence which also satisfies \eqref{eqn:hitchin} as in the proof of the existence theorem above.
\end{proof}

\subsubsection{The Yang-Mills-Higgs flow}
We define the {\bf Yang-Mills-Higgs flow } for a pair $(A,\Phi)$ by the equations
\begin{align}
\begin{split}\label{eqn:ymh-flow}
\frac{\partial A}{\partial t}&= -d_A^\ast(F_A+[\Phi, \Phi^\ast]) \ ;\\
\frac{\partial \Phi}{\partial t} &= [\Phi, \sqrt{-1}\Lambda(F_A+[\Phi, \Phi^\ast])]\ .
\end{split}
\end{align}
In the above, we only consider initial conditions where $\Phi$ is $d''_A$-holomorphic.  Notice then that this holomorphicity condition is preserved along a solution to \eqref{eqn:ymh-flow}.  Indeed, as in Donaldson \cite{Donaldson85}, the flow is tangent to the complex gauge orbit and exists for all $0\leq t<+\infty$.  
The flow equations may be regarded as the $L^2$-gradient flow of the $\YMH$ functional.
They generalize the Yang-Mills flow equations. For more on this we refer to \cite{Hong01, Wilkin08} and the references therein.  Here we limit ourselves to a discussion of a few key properties.
In particular, we justify the assumptions in the previous section.

As in \eqref{eqn:hitchin2}, set $f_{(A,\Phi)}= \sqrt{-1}\Lambda(F_A+[\Phi, \Phi^\ast])$.

\begin{lemma} \label{lem:ymh-flow1}
For all $t\geq 0$,
$$
\frac{d}{dt}\YMH(A,\Phi)=-2\Vert d_A f_{(A,\Phi)}\Vert_{L^2}^2-4\Vert[\Phi,f_{(A,\Phi)}\Vert^2_{L^2}\ .
$$
\end{lemma}

\begin{proof}
We have
$$
\frac{d}{dt}\YMH(A,\Phi)=2\int_X\tr(f_{(A,\Phi)}\dot f_{(A,\Phi)}) \omega\ .
$$
Now using dots to denote time derivatives, 
\begin{align*}
\dot f_{(A,\Phi)}&=
\sqrt{-1}\Lambda\left(d_A\dot A+[\dot\Phi,\Phi^\ast]+[\Phi,\dot\Phi^\ast] \right)\\
&=
\sqrt{-1}\Lambda\left( -d_Ad_A^\ast(F_A+[\Phi,\Phi^\ast])+[[\Phi,f],\Phi^\ast]+[\Phi, [\Phi,f]^\ast]\right) \\
&=
 -d_A^\ast d_A  f_{(A,\Phi)}+
 \sqrt{-1}\Lambda\left(
 [\Phi, f_{(A,\Phi)}]\Phi^\ast+\Phi^\ast[\Phi, f_{(A,\Phi)}]+\Phi[\Phi, f_{(A,\Phi)}]^\ast+[\Phi, f_{(A,\Phi)}]^\ast\Phi\right)\ .
\end{align*}
Taking traces we get
\begin{equation} \label{eqn:fdot}
\tr(f_{(A,\Phi)}\dot f_{(A,\Phi)}) 
=-\tr(f_{(A,\Phi)} d_A^\ast d_A  f_{(A,\Phi)})
-2\sqrt{-1}\Lambda\tr\left([\Phi, f_{(A,\Phi)}][\Phi, f_{(A,\Phi)}]^\ast\right)\ ,
\end{equation}
and the result follows by integration by parts.
\end{proof}

As a consequence of Lemma \ref{lem:ymh-flow1}, $\YMH$ decreases along the flow.  Moreover, we have the following inequality
$$
\int_0^\infty dt\left\{2\Vert d_A f_{(A,\Phi)}\Vert_{L^2}^2+4\Vert[\Phi,f_{(A,\Phi)}\Vert^2_{L^2} \right\}
\leq \YMH(A_0,\Phi_0)\ .
$$
It follows that if $(A_j,\Phi_j)$ is a sequence with $\YMH(A_j,\Phi_j)$ uniformly bounded, then we may replace it with another sequence $(\widetilde A_j,\widetilde \Phi_j)$ with 
$\YMH(\widetilde A_j,\widetilde \Phi_j)$
also uniformly bounded but such that 
$d_{A_j}f_{(\widetilde A_j,\widetilde \Phi_j)}$
and $[\Phi_j,f_{(\widetilde A_j,\widetilde \Phi_j)}]$ converge to $0$ in $L^2$.

Now let's compute
\begin{align*}
\Delta \left|f_{( A,\Phi)}\right|^2&=-d^\ast d \left|f_{( A,\Phi)}\right|^2 =\ast d\ast d\tr f_{( A,\Phi)}^2 \\
&=2\ast d\ast \tr (f_{( A,\Phi)} d_A f_{( A,\Phi)}) \\
&=2\ast  \tr( df_{( A,\Phi)} \wedge \ast d_A f_{( A,\Phi)})
-2\tr(f_{( A,\Phi)} d_A^\ast d_A f_{( A,\Phi)}) \\
&=2\left| df_{( A,\Phi)}\right|^2+4\left| [\Phi,f_{( A,\Phi)}]\right|^2+\frac{\partial}{\partial t}\left|f_{( A,\Phi)}\right|^2\ ,
\end{align*}
from \eqref{eqn:fdot}.
We have shown
\begin{lemma}
For all $t\geq 0$,
$$
\frac{\partial}{\partial t}\left|f_{( A,\Phi)}\right|^2
-\Delta \left|f_{( A, \Phi)}\right|^2 
=
-2\left| d_A f_{( A, \Phi)}\right|^2-4\left| [\Phi, f_{( A, \Phi)}]\right|^2\ .
$$
\end{lemma}
In particular, $ \left|f_{( A,\Phi)}\right|$ is a subsolution of the heat equation, and so
 $\sup \left|f_{( A,\Phi)}\right|$ is nonincreasing.  In fact,  one can use
an explicit argument with the heat kernel
 to show that for $t\geq 1$, say,  the 
 $\sup \left|f_{( A_t,\Phi_t)}\right|\leq C\YMH(A_0,\Phi_0)$
 for a fixed constant $C$.
In particular,
if $(A_j,\Phi_j)$ is a sequence with $\YMH(A_j,\Phi_j)$ uniformly bounded, then we may replace it with another sequence $(\widetilde A_j,\widetilde \Phi_j)$ with 
$f_{(\widetilde A_j,\widetilde \Phi_j)}$ uniformly bounded.

\begin{proof}[Proof of Lemma \ref{lem:minimizing}]
Choose $(A_j,\Phi_j)$ a minimizing sequence for $J$ in the complex gauge orbit of $(A,\Phi)$.
Note that $\YMH(A_j,\Phi_j)$ is then uniformly bounded.  In addition,
by an argument similar to the one above (see \cite{DaskalWentworth04}), $J$ is also decreasing along the $\YMH$-flow. 
Hence, replacing each $(A_j,\Phi_j)$ with a point along the $\YMH$-flow with initial condition $(A_j,\Phi_j)$ also gives a $J$-minimizing sequence.  On the other hand, by the discussion in this section, we can choose points along the flow where items (ii) and (iii) are also satisfied.  This completes the proof.
\end{proof}

Let $\Bcal_E^{min}$ be the set of all Higgs bundles satisfying the Hitchin equations \eqref{eqn:hitchin2}.
The $\YMH$-flow sets up an infinite dimensional, singular Morse theory problem where $\Bcal_E^{min}$ is the minimum of the functional, and Higgs bundles not in  $\Bcal_E^{min}$ but 
satisfying \eqref{eqn:critical} play the role of higher critical points. 
This Morse theory picture can actually be shown to be more than just an analogy.  In particular, we have the following

\begin{theorem}[Wilkin {\cite{Wilkin08}}] \label{thm:wilkin}
The $\YMH$-flow gives a $\Gcal_E$-equivariant deformation retraction
of $\Bcal_E^{ss}$ onto $\Bcal_E^{min}$.
\end{theorem}

\section{The Betti Moduli Space} \label{sec:betti}

\subsection{Representation varieties}
\subsubsection{Definition}
Fix a base point $p\in X$ and
 set $\pi=\pi_1(X,p)$. 
Let  $\Hom(\pi,\SL_n(\CBbb))$ denote the set of homomorphisms from $\pi$ to $\SL_n(\CBbb)$. This has the structure of an affine algebraic variety.
Let
\begin{equation*} 
\Mfrak_B^{(n)} =\Hom(\pi,\SL_n(\CBbb))\doubleslash \SL_n(\CBbb)\ ,
\end{equation*}
denote the representation variety,
where the double slash indicates the invariant theoretic quotient by   overall conjugation of $\SL_n(\CBbb)$.  Then $\Mfrak_B^{(n)}$ is an irreducible affine variety of complex dimension $(n^2-1)(2g-2)$.
There is a surjective algebraic quotient map $\Hom(\pi,\SL_n(\CBbb)) \to \Mfrak_B^{(n)}$,  and this is a geometric quotient on the open set of irreducible (or simple) representations. Points of $\Mfrak_B^{(n)}$ are in 1-1 correspondence with conjugacy classes of semisimple (or reductive) representations, and every $\SL_n(\CBbb)$ orbit in $\Hom(\pi,\SL_n(\CBbb))$ contains a semisimple representation in its closure (for these results, see \cite{LubotzkyMagid85}).  Following Simpson \cite{Simpson94a, Simpson94b} I  will refer to $\Mfrak_B^{(n)}$ as the {\bf Betti moduli space} of rank $n$.

 Let $E\to X$ be a trivial rank $n$ complex
vector bundle.  A flat connection $\nabla$ on $E$ gives rise
to a representation of $\pi$ as follows.  Recall that we have
fixed a base point $p\in X$.  We also fix a frame
$\{\ebold_i\}$ of $E_p$.  For each loop $\gamma$ based at $p$, parallel
translation of the frame $\{\ebold_i\}$ defines an element of
$\GL_n(\CBbb)$.  Since the connection is flat this is independent of
the choice of path in the homotopy class.  In this way we have
defined an element $\hol(\nabla)\in \Hom(\pi, \GL_n(\CBbb))$.  If $\nabla$
induces the trivial connection on $\det E$, the holonomy lies
in $\SL_n(\CBbb)$, and we will assume this from now on.
%The holonomy of a  flat connection 
%gives a representation $\rho: \pi_1(X,p)\to \GL_n(\CBbb)$.
%A projectively flat connection gives a representation 
%$\rho: \pi_1(X,p)\to \PGL_n(\CBbb)$.
%If the connection is in addition unitary for a hermitian
%metric, then the corresponding holonomy representations take values in
%the unitary groups $\U_n$ and $\PU_n$.
%
Conversely, given a representation $\rho: \pi_1(X,p)\to
\SL_n(\CBbb)$,  we obtain a holomorphic bundle $\Vcal_\rho$
with a flat connection $\nabla$ by the quotient 
$
\Vcal_\rho=\widetilde X\times \CBbb^n/\pi
$,
where $\widetilde X$ is the universal cover of $X$, and the quotient identifies $(x,v)\sim (x\gamma,
v\rho(\gamma))$.
Let $\Ccal_E$ denote the space of connections on $E$, and $\Ccal^{flat}_E\subset\Ccal_E$ the flat connections.
 Let $\Gcal_E^\CBbb(p)$ denote the space of complex gauge
transformations that are the identity at $p$,  acting on
$\Ccal_E$ by conjugation ({\bf warning:} this is a different
action of $\Gcal^\CBbb_E$ from the one on the space of \emph{unitary} connections in Section \ref{sec:gauge}). 

\begin{proposition} \label{prop:monodromy}
The holonomy map gives an $\SL_n(\CBbb)$-equivariant homeomorphism
$$
\hol: 
\Ccal_E^{flat}/\Gcal_E^\CBbb(p) \simrightarrow \Hom(\pi,\SL_n(\CBbb))\ .
$$
In particular, $\Ccal_E^{flat}\doubleslash\Gcal_E^\CBbb\simeq \Mfrak_B^{(n)}$.
\end{proposition}

\subsection{Local systems and holomorphic connections}

\subsubsection{Definitions}
\begin{definition}
A {\bf complex $n$-dimensional  local system} on $X$ is a  sheaf of abelian groups that is locally isomorphic to the constant sheaf $\underline\CBbb^n$.
\end{definition}

\noindent
Here $\underline\CBbb$ denotes the locally constant sheaf modeled on $\CBbb$. Clearly a local system $\Vbold$ is a sheaf of modules over $\underline\CBbb$.

\begin{definition}
Let $\Vcal\to X$ be a holomorphic bundle.  A {\bf holomorphic connection} on $\Vcal$ is a $\CBbb$-linear operator $\nabla: \Vcal\to \Kcal \otimes\Vcal$ satisfying the Leibniz rule
\begin{equation} \label{eqn:leibniz}
\nabla(fs)=df\otimes s + f\nabla s\ ,
\end{equation}
for local sections $f\in \Ocal $, $s\in \Vcal$.
\end{definition}

For a local system $\Vbold$ 
let $\Vcal$ be the holomorphic bundle  $\Vcal=\Ocal\otimes_{\underline\CBbb}\Vbold$. 
Then $\Vcal$ inherits a holomorphic connection as follows: 
choose a local parallel frame $\{\vbold_i\}$ for $\Vbold$.
Any local section of $\Vcal$ may be written uniquely as $s=\sum_{i=1}^n f_i\otimes\vbold_i$, with $f_i\in \Ocal$.  Then \emph{define} 
$
\nabla s=\sum_{i=1}^n df_i\otimes \vbold_i
$. Since the transition functions for $\Vbold$ are constant this is well-defined independent of the choice of frame, and $\nabla$ also immediately  satisfies the Leibniz rule.
Conversely, a holomorphic connection defines a \emph{flat} connection on the underlying complex vector bundle, since
in a local holomorphic frame the curvature $F_\nabla$ is necessarily of type $(2,0)$, and
 on a Riemann surface there are no $(2,0)$-forms. In particular, the $\underline\CBbb$-subsheaf $\Vbold\subset \Vcal$ of locally parallel sections $\nabla s=0$ defines a local system. This gives a categorical equivalence between local systems and holomorphic bundles with a holomorphic connection (see
\cite[Th\'eor\`eme 2.17]{Deligne70}).

%\begin{remark}
%In higher dimensions the same construction holds provided one imposes the integrability condition $\nabla^2=0$.
%\end{remark}

A local system has a {\bf monodromy
 representation} $\rho:\pi\to \GL_n(\CBbb)$, obtained by developing local parallel frames.
 Conversely, given $\rho$ we construct a local system as  in the previous section. 
 We will sometimes denote these $\Vbold_\rho$ and $\Vcal_\rho$.
 For simplicity, in these notes I will almost 
 always assume the monodromy lies in $\SL_n(\CBbb)$, or in other words, $\det\Vcal_\rho\simeq\Ocal$ and the induced connection on $\det\Vcal_\rho$ is trivial.

Not every holomorphic bundle $\Vcal$ admits a holomorphic connection.  In particular, such a connection is flat, and so by \eqref{eqn:chern-weil} a necessary condition is that $\deg\Vcal=0$.  
In fact, one can say more about the Harder-Narasimhan type of a bundle with a holomorphic connection.

\begin{proposition}[cf.\ \cite{EsnaultViehweg99, Bolibrukh02}] \label{prop:bound}
Suppose $\Vcal$ is an unstable bundle with 
an irreducible holomorphic connection,  and let $\mu_1>\mu_2>\cdots>\mu_\ell$ be the Harder-Narasimhan type.  Then for each $i=1,\ldots, \ell-1$, $\mu_i-\mu_{i+1}\leq 2g-2$.
\end{proposition}

\begin{proof}
Let $0\subset \Vcal_1\subset\cdots\subset \Vcal_\ell=\Vcal$ be the Harder-Narasimhan filtration of $\Vcal$.  Then since the connection is irreducible the $\Ocal$-linear map 
$\Vcal_i\stackrel{\nabla}{\xrightarrow{\hspace*{.5cm}}} \Vcal/\Vcal_i\otimes \Kcal $
 is nonzero for each $i=1,\ldots, \ell-1$. 
 Let $j\leq i$ be the smallest integer such that
 $\Vcal_j\to \Vcal/\Vcal_i\otimes \Kcal $
is nonzero. Then it follows from the sequence$$
0\lra \Vcal_{j-1}\lra \Vcal_j \lra \Qcal_j\lra 0
$$
that there is a nonzero map $\Qcal_j\to \Vcal/\Vcal_i\otimes \Kcal $.
With this fixed $j$,  let $k\geq i$ be the largest integer such that $\Qcal_j\to \Vcal/\Vcal_k\otimes \Kcal $ is nonzero.
It follows from
$$
0\lra \Qcal_{k+1}\lra \Vcal/\Vcal_k\lra \Vcal/\Vcal_{k+1}\lra 0
$$
 that  $\Qcal_j\to \Qcal_{k+1}\otimes \Kcal $ is nonzero.
  Since the $\Qcal_i$ are all semistable, we have by Lemma \ref{lem:semistable} that
$$
\mu_j=\mu(\Qcal_j)\leq \mu(\Qcal_{k+1}\otimes\Kcal )=\mu_{k+1}+2g-2\ ,
$$
and the result follows, since $\mu_i-\mu_{i+1}\leq \mu_j-\mu_{k+1}$.
\end{proof}

\subsubsection{The Weil-Atiyah theorem} The goal of this section is to prove the following

\begin{theorem}[Weil {\cite{Weil38}}, Atiyah {\cite{Atiyah57}}] \label{thm:weil}
A holomorphic bundle $\Vcal \to X$ admits a holomorphic connection if and only if each indecomposable factor of $\Vcal$ has degree zero.
\end{theorem}

The proof I give here follows Atiyah.
The following construction will be useful (see {\cite[p.\ 193]{Atiyah57}}).
Any holomorphic bundle $\Vcal\to X$ gives rise to a counterpart
 $D(\Vcal)$ as follows.  First, as a smooth bundle $D(\Vcal)=(V\otimes K)\oplus V$.  With respect to this splitting define the $\Ocal$-module structure by 
$$
f(\varphi,s)=(f\varphi+s\otimes df, fs)\ ,\qquad \ f\in \Ocal\ ,\ \varphi\in \Vcal\otimes\Kcal\ ,\ s\in \Vcal\ .
$$
One checks that this gives $D(\Vcal)$ the structure of a locally free sheaf over $\Ocal$.
Then we have a compatible inclusion $\varphi\mapsto(\varphi,0)$ and projection $(\varphi,s)\mapsto s$ making $D(\Vcal)$ into an extension
\begin{equation} \label{eqn:de}
0\lra \Vcal\otimes\Kcal\lra D(\Vcal)\lra \Vcal\lra 0\ .
\end{equation}
Observe that \eqref{eqn:de} \emph{splits if and only if $\Vcal$ admits a holomorphic connection}. Indeed, such a $\nabla$ gives a splitting by $s\mapsto (\nabla s,s)$, and if \eqref{eqn:de} splits then there is a $\underline\CBbb$-linear map $\Vcal\to \Vcal\otimes\Kcal$ satisfying \eqref{eqn:leibniz}.

\begin{remark} \label{rem:de}
The construction is functorial with respect to subbundles.  If $0=\Vcal_0\subset\Vcal_1\subset\cdots\subset\Vcal_\ell=\Vcal$ is a filtration of $\Vcal$ by holomorphic subbundles, then
there is a filtration 
$$0=D(\Vcal_0)\subset D(\Vcal_1)\subset\cdots\subset D(\Vcal_\ell)=D(\Vcal)\ .$$
\end{remark}

\begin{lemma} \label{lem:extension-class}
Given a holomorphic bundle $\Vcal\to X$, let 
$$[\beta]\in H^1(X,(\Vcal\otimes\Kcal)\otimes\Vcal^\ast)\simeq H^{1,1}_{\dbar}(X,\End V)\ ,$$
denote the extension class.  Then $[\tr\beta]=-2\pi\sqrt{-1}\, c_1(V)$.
\end{lemma}

\begin{proof}
Choose $s^{(i)}$ local holomorphic frames for $\Vcal$ on $U_i$, and let  $\psi_{ij}$ denote the transition functions: $s^{(i)}=s^{(j)}\psi_{ij}$.
  We can define local splittings of \eqref{eqn:de} by $s^{(i)}f^{(i)}\mapsto s^{(i)}\otimes df^{(i)}$, for $f^{(i)}$ a vector of holomorphic functions on $U_i$. 
  In particular,
  $$
  f^{(j)}=\psi_{ij}f^{(i)}\ ,\ 
  \partial f^{(j)}=\psi_{ij}(\psi_{ij}^{-1}\partial\psi_{ij}f^{(i)}+\partial f^{(i)})\ .
  $$
   Since the extension class is given by the image of $I$ under the map 
$$
H^0(X,\End \Vcal)\to H^1(X, \End\Vcal\otimes \Kcal)\ ,
$$ 
it follows from the local splitting above that $[\beta]$ is represented by the cocycle $[\psi_{ij}^{-1}d\psi_{ij}]$.  Hence, $[\tr\beta]=[d\log\det\psi]$.  On the other hand, if $h$ is a hermitian metric on $\det\Vcal$, then
$$
h_i|s^{(i)}_1\wedge \cdots\wedge s^{(i)}_n|^2=h_j|s^{(j)}_1\wedge \cdots\wedge s^{(j)}_n|^2\ ,
$$
so $h_i|\det\psi_{ij}|^2=h_j$. This implies $d\log\det\psi_{ij}=\partial\log h_j-\partial\log h_i$. By the Dolbeault isomorphism $[\beta]$ is represented by $[\dbar\partial\log h_i]=[F_{(\dbar_{\det\Vcal},h)}]=-2\pi\sqrt{-1}\, c_1(V)$ (see Example \ref{ex:line-bundle-curvature} and \eqref{eqn:chern-weil}).
\end{proof}

\begin{lemma} \label{lem:nilpotent}
If $\Vcal\to X$ is an indecomposable holomorphic bundle and $\phi\in H^0(X,\End \Vcal)$, Then there is $\lambda\in \CBbb$ such that $\phi-\lambda I$ is nilpotent.
\end{lemma}

\begin{proof}
Since $\det(\phi-\lambda I)$ is holomorphic and $X$ is closed, the eigenvalues of $\phi$ must be constant.  So without loss of generality assume $\ker\phi\neq \{0\}, \Vcal$, and consider the sequence
\begin{equation} 
\begin{matrix}
0 \lra & \ker\phi &\lra & \Vcal & \lra & \coker \phi & \lra 0   \\
& \begin{sideways}$=$\end{sideways} &&&& \begin{sideways}$=$\end{sideways} & \\
& \Scal &&&& \Qcal &
\end{matrix}
\label{eqn:kernel}
\end{equation}
Write:
$$
\dbar_E=\left(\begin{matrix} \dbar_S&\beta \\ 0&\dbar_Q\end{matrix}\right)\qquad ,\qquad
\phi=\left(\begin{matrix} 0&\phi_1\\ 0&\phi_2\end{matrix}\right)\ .
$$
We wish to show $\phi_2=0$.  First note that
$$
0=\dbar_E\phi=\left(\begin{matrix} 0&\dbar_E\phi_1+\beta\phi_2\\ 0&\dbar_Q\phi_2\end{matrix}\right)\ .
$$
So $\phi_2$ is holomorphic as an endomorphism of $\Qcal$.  If $\phi_2\neq 0$, then it is an isomorphism.  This is so because again the eigenvalues of $\phi_2$ are constant, and by assumption $0$ is not an eigenvalue.  Hence, we can rewrite the upper right entry in the matrix equation above as: $\dbar_E(\phi_1\phi_2^{-1})+\beta=0$.  But then the Dolbeault class of $\beta$ vanishes and \eqref{eqn:kernel} splits, contradicting the assumption that $\Vcal$ be indecomposable.
\end{proof}

\begin{proof}[Proof of Theorem \ref{thm:weil}]
Suppose $\Vcal$ has a holomorphic connection. Then by Remark \ref{rem:de}, $D(\Vcal)$ splits.  Moreover, since $D(\Vcal)$ is natural with respect to subbundles, $D(\Vcal_i)$ splits for each indecomposable factor of $\Vcal$.  But then by Lemma \ref{lem:extension-class}, $\deg(\Vcal_i)=0$ for all $i$.  Conversely, suppose $\Vcal$ is indecomposable and $\deg(\Vcal)=0$. It suffices to show $D(\Vcal)$ splits. Now by Serre duality the extension class 
$$
[\beta]\in H^1(X, \End(\Vcal)\otimes \Kcal)\simeq \left(H^0(X, \End(\Vcal)) \right)^\ast\ ,
$$
and the perfect pairing is $\displaystyle (\beta,\phi)=\int_X \tr(\beta\phi)$. By Lemma \ref{lem:nilpotent} we may express $\phi=\lambda I+\phi_0$, where $\phi_0$ is nilpotent.  Then by Lemma \ref{lem:extension-class},
\begin{equation}
(\beta,\phi)=(\beta,\phi_0)+\lambda(\beta, I)
=(\beta,\phi_0)+\lambda\int_X \tr \beta\\
=(\beta,\phi_0) -2\pi\sqrt{-1}\lambda\deg(E)
=(\beta,\phi_0)\ . \label{eqn:phi0}
\end{equation}
Set $\Vcal_\ell=\Vcal$, and 
recursively define $\Vcal_{i-1}$ to be the saturation of $\phi_0(\Vcal_i)$.  Note that $\Vcal_{i-1}$ is a proper subbundle of $\Vcal_i$, since otherwise the restriction of  $\phi_0$ would be almost everywhere an isomorphism. Eventually the process terminates.
Adjust $\ell$ so that $\Vcal_0=\{0\}$, $\Vcal_1\neq \{0\}$. By Remark \ref{rem:de},  $\beta$ preserves the  filtration 
$0=\Vcal_0\subset\Vcal_1\subset\cdots\subset\Vcal_\ell=\Vcal$. Choose a hermitian metric on $V$ and let $\pi_i$ be orthogonal projection to $V_i$.  Note that
$$
I=\sum_{i=1}^\ell(\pi_i-\pi_{i-1})=\sum_{i=1}^\ell(\pi_i-\pi_i\pi_{i-1})=\sum_{i=1}^\ell\pi_i(I-\pi_{i-1})\ ,
$$
and $
(I-\pi_i)\beta\pi_i=(I-\pi_{i-1})\phi\pi_i=0
$.
Then 
\begin{align*}
\tr(\beta\phi_0)=
\tr(\phi_0\beta)&=\sum_{i=1}^\ell \tr(\phi_0\beta\pi_i(I-\pi_{i-1}))\\
&=\sum_{i=1}^\ell \tr((I-\pi_{i-1})\phi_0\beta\pi_i)\\
&=
\sum_{i=1}^\ell \tr((I-\pi_{i-1})\phi_0\pi_i\beta\pi_i)\\
&=0\ .
\end{align*}
So $(\beta,\phi_0)=0$, and by \eqref{eqn:phi0} we conclude $[\beta]=0$. The proof is complete.
\end{proof}

\subsection{The Corlette-Donaldson theorem} \label{sec:corlette}

\subsubsection{Hermitian metrics and equivariant maps}

Let $D=\SU_n\backslash\SL_n(\CBbb)$ and $\rho: \pi\to \SL_n(\CBbb)$. Then $\pi$ acts on the right on $D$ via the representation $\rho$. Following Donaldson, we give a concrete description of $D$ with its $\SL_n(\CBbb)$-action.  Set
$$
D=\{\text{positive hermitian $n\times n$ matrices $M$ with $\det M=1$}\}\ .
$$
Then the right $\SL_n$ action is given by $(M,g)\mapsto g^{-1}M(g^{-1})^\ast$.  Note that the space $D$ may be interpreted as the space of hermitian inner products on $\CBbb^n$ which induce a fixed one on $\det \CBbb^n$. The invariant metric on $D$ is given by $|M^{-1}dM|^2=\tr(M^{-1}dM)^2$.

\begin{definition} \label{def:equivariant}
A map $u:\widetilde X\to D$ is $\rho$-equivariant if $u( x\gamma)=u(x)\rho(\gamma)$ for all $x\in X$, $\gamma\in \pi$. 
\end{definition}

Let $E=\widetilde X\times \CBbb^n /\pi$.  We now claim that a hermitian metric on the bundle $E$ is equivalent to a choice of $\rho$-equivariant map, up to the choice of basepoints.  Indeed, suppose $u:\widetilde X\to D$ is $\rho$-equivariant.  By definition, a  section of $E$ is a map $\sigma: \widetilde X\to \CBbb^n$ such that $\sigma(x\gamma)=\sigma(x)\rho(\gamma)$.  Hence, if we define 
$\Vert\sigma\Vert^2(x)=\langle\sigma(x), \sigma(x)u(x)\rangle_{\CBbb^n}$, then
$$
\Vert\sigma\Vert^2(x\gamma)=\langle\sigma(x)\rho(\gamma), \sigma(x)u(x)(\rho(\gamma)^{-1})^\ast\rangle_{\CBbb^n}=\Vert\sigma\Vert^2(x)\ ,
$$
and so this is a well-defined metric on $E$.  In the other direction, given a metric $H$,  if $\sigma_i$ are sections, then
write $\langle \sigma_i,\sigma_j\rangle_H(x)=\langle \sigma_i(x),\sigma_j(x) u(x)\rangle_{\CBbb^n}$, for a hermitian matrix valued function $u(x)$.  Then
\begin{align*}
\langle \sigma_i(x),\sigma_j(x) u(x)\rangle_{\CBbb^n}&=
\langle \sigma_i,\sigma_j\rangle_H(x)=\langle \sigma_i,\sigma_j\rangle_H(x\gamma)  \\
&=\langle \sigma_i(x)\rho(\gamma),\sigma_j(x)\rho(\gamma)u(x\gamma)\rangle_{\CBbb^n} \\
&=\langle \sigma_i(x),\sigma_j(x)\rho(\gamma)u(x\gamma)\rho(\gamma)^\ast\rangle_{\CBbb^n} 
\end{align*}
for all sections. Hence, $\rho(\gamma)u(x\gamma)\rho(\gamma)^\ast=u(x)$, and $u$ is $\rho$-equivariant.

\subsubsection{Harmonic metrics}

If $u:\widetilde X\to D$ is a continuously
  differentiable $\rho$-equivariant map, 
we define its energy as follows.  
The derivative $du$ is a section of $T^\ast\widetilde X\otimes
u^\ast(TD)$.  
We have fixed an invariant metric on $D$, so the norm 
$e_u(x)=|du|^2(x)$.  In fact, by equivariance, $e_u(x)$ is 
invariant under $\pi$, so it gives a well-defined
 function on $X$ which is called the {\bf energy density}.  
The {\bf energy} of $u$ is then by definition 
\begin{equation} \label{eqn:energy}
E_\rho(u)=\int_X e_u(x)\, \omega\ .
\end{equation}
Note that the energy only depends on the conformal structure on 
$X$ and not the full metric.

The Euler-Lagrange equations for $E_\rho$ are easy to write
down. Define
\begin{equation} \label{eqn:tension}
\tau(u)=d_\nabla^\ast du\ .
\end{equation}
In the above we note that the bundle
$u^\ast(TD)$ has a connection $\nabla$: the pull-back of the
Levi-Civita connection on $D$. It is with respect to this
connection that $d_\nabla$ is defined.  The tensor $\tau(u)$
is called the {\bf tension field}. It is a section of 
$u^\ast(TD)$.
\begin{definition}\label{def:harmonic}
A $C^2$ $\rho$-equivariant map
$u$ is called {\bf harmonic} 
if it satisfies
\begin{equation} \label{eqn:harmonic-map}
\tau(u)=0\ .
\end{equation}
\end{definition}

 Eq.\ \eqref{eqn:harmonic-map} is a second order
elliptic nonlinear partial differential equation in $u$.
This statement is a slightly misleading because $u$ is a
mapping and not a collection of functions. 
This annoying fact makes defining weak solutions a little tricky.
In the case of maps between compact manifolds (the non-equivariant problem)
 one way to circumvent this issue  is to use a Nash isometric embedding of the target into
a euclidean space and rewrite the equations in terms of coordinate functions (cf.\ \cite{Schoen83}).   A more sophisticated technique, better suited to the equivariant problem, is to define the Sobolev
space theory intrinsically (cf.\ \cite{KorevaarSchoen93, KorevaarSchoen97,  Jost94}).
On the other hand, if we \emph{assume} $u$ is Lipschitz continuous, then we can introduce local
coordinates $\{y^a\}$ on $D$ and write
\eqref{eqn:harmonic-map} locally.  By Rademacher's theorem the pull-backs
$s_a=u^\ast(\partial/\partial y^a)$ give a local frame for
$u^\ast(TD)$ almost everywhere, and the connection forms for $\nabla$ 
in this frame are $\Gamma_{ab}^c(u) du^a\otimes s_c$, where
$\Gamma_{ab}^c(u)$ are the Christoffel symbols on $D$
evaluated along $u$. Writing $u=(u^1, \ldots, u^N)$ in terms
of the coordinates on $\{y^a\}$, it is easy to see that
the local expression of
\eqref{eqn:harmonic-map} becomes
\begin{equation}  \label{eqn:harmonic-map-local}
-\tau(u)^a= \Delta u^a+\Gamma_{bc}^a(u) \nabla u^b\cdot \nabla
u^c=0\ .
\end{equation}
To be clear, the dot product in the second term refers to the
metric on $X$, and $\Delta$ is the Laplace operator on $X$.
Notice that this equation is conformally invariant with
respect to the metric on $X$, a
manifestation of the fact that the energy functional itself is
conformally invariant.

In light of the previous section,
 $\rho$-equivariant maps are equivalent to choices of hermitian metrics.  
Given a flat connection $\nabla$ and
 hermitian metric on $E$ we can construct the equivariant map 
in a more intrinsic way.  First, lift $\nabla$ and $E$ to obtain 
a flat connection on a trivial bundle on the universal cover 
$\widetilde X$.  We will use the same notation to denote this lifted 
bundle and connection. If 
we choose a base point $\hat p$ covering the base 
point $p$ for $\pi_1(X,p)$, and we choose a unitary 
frame $\{\ebold_i(\hat p)\}$ for the fiber $E_{\hat p}$, let 
$\{\ebold_i(x)\}$ be given by parallel transport with respect to 
$\nabla$.  Then the map $u: \widetilde X\to D$ is given by 
$x\mapsto \langle \ebold_i, \ebold_j\rangle(x)$. 
 It is $\rho$-equivariant and uniquely determined up
 to the choice of $\hat p$ and the base point in $D$. 

Conversely, if $u:\widetilde X\to D$ is any $\rho$-equivariant 
map such that $u(\hat p)=I$, then $u$ defines a hermitian metric 
for which it is the equivariant map constructed above. 
 Notice that there is an equivalence of the type we saw for
 Higgs bundles.  If $g\in \Gcal_E^\CBbb(p)$ then the
 corresponding $\rho$-equivariant map obtained 
from the pair $(g(\nabla),H)$ is the same as that for
 $(\nabla, Hg)$.  Finally, if we act by a constant 
$g\in\SL_n(\CBbb)$, the same is true,
 but now the map is $(\rho\cdot g)$-equivariant.
The moral of the story is that finding a  harmonic
 metric is equivalent to finding a harmonic equivariant map in the $\Gcal_E^\CBbb$ orbit of $\nabla$. 

Given the data $(\nabla, H)$, we may uniquely 
write $\nabla=d_A+\Psi$ where, 
$d_A$ is a unitary connection on $(E,H)$, 
and $\Psi$ is a $1$-form with values in the bundle $\sqrt{-1}\gfrak_E$ of 
hermitian endomorphisms. 
 We can explicitly define $\Psi$ with respect to a local frame $\{s_i\}$ by 
\begin{equation}\label{eqn:psi}
\langle \Psi s_i, s_j\rangle=\frac{1}{2}\left\{
\langle \nabla s_i, s_j\rangle+\langle s_i,
 \nabla s_j\rangle -d\langle s_i, s_j\rangle
\right\}\ .
\end{equation}

\begin{lemma}[{cf.\ \cite{Donaldson87}}] \label{lem:energy-psi}
The energy of the map defined above is given by $E_\rho(u)
=4\Vert \Psi\Vert^2$. 
\end{lemma}

\begin{proof}
From the definition above and the fact that $d_A$ is unitary,
$$
du_{ij}=\langle d_A \ebold_i, \ebold_j\rangle+\langle \ebold_i, d_A\ebold_j\rangle\ .
$$
On the other hand, the $\ebold_i$ are parallel with respect to $\nabla$, so
$d_A\ebold_j=-\Psi\ebold_j$.  Hence, $u^{-1}du=-2\Psi$.
\end{proof}

\begin{definition} \label{def:harmonic-metric}
We say that $H$ is a {\bf harmonic metric} if the map $u$ defined above is a harmonic map.
\end{definition}

\begin{proposition}[Corlette {\cite{Corlette88}}] \label{prop:semisimple}
If $\rho$ admits a harmonic metric then $\rho$ is semisimple.
\end{proposition}

\begin{proof}
Suppose  that $H$ is a critical metric but that $\nabla$ is reducible.  Let $V_1\subset V$ be a subbundle invariant with respect to the connection $\nabla$.   Let $V_2$ be the orthogonal complement of $V_1$, and $H_1$, $H_2$ the induced metrics.  We can express
$$
\nabla=\left( \begin{matrix} \nabla_1 & \beta \\ 0 &\nabla_2\end{matrix}\right)=
\left( \begin{matrix} d_{A_1}+ \Psi_1 & \beta \\ 0 &d_{A_2}+\Psi_2\end{matrix}\right)\ ,
$$
where $\beta\in \Omega^1(X, \Hom(V_2,V_1))$.  It suffices to show that the connection splits, or in other words that $\beta\equiv 0$.  The proposition then follows by induction.
Now using \eqref{eqn:psi} it follows that if $s_1, s_2$ are local sections of $V_1$, then 
$\langle\Psi s_1, s_2\rangle=\langle\Psi_1 s_1, s_2\rangle$.  Similarly, $\langle\Psi s_1, s_2\rangle=\langle\Psi_1 s_1, s_2\rangle$ for local sections of $V_2$. On the other hand, if $s_i\in V_i$, then
$\langle\Psi s_1, s_2\rangle=\frac{1}{2}\langle s_1, \beta s_2\rangle$.
It follows that
$$
\Psi=\left( \begin{matrix} \Psi_1 &\frac{1}{2} \beta \\ \frac{1}{2}\beta^\ast &\Psi_2\end{matrix}\right)\ .
$$
We now deform the metric $H$ to a family $H_t$ as follows: scale $H_1\mapsto e^{-(\rank V_2)t}H_1$, and $H_2\mapsto e^{+(\rank V_1)t}H_2$.  This, of course, preserves the orthogonal splitting and the condition $\det H_t=1$.
But $H_t$ is a geodesic homotopy of $\rho$-equivariant maps, and so
by a result of Hartman the energy $E_\rho(u_t)$ is convex \cite{Hartman67}.  On the other hand, by Lemma \ref{lem:energy-psi},
$$
\frac{1}{4}E_\rho(u_t)=\Vert\Psi_1\Vert^2_{H_1}+\Vert\Psi_2\Vert^2_{H_2}+\Vert\beta\Vert^2_{H}\, e^{-(\rank V)t/2}\ .
$$
In particular, $E_\rho(u_t)$ is bounded as $t\to \infty$.  The only way $E_\rho(u_t)$ could have a critical point at $t=0$ is if $E_\rho(u_t)$ is constant, which implies $\beta\equiv 0$. 
 This completes the proof.
\end{proof}

\subsubsection{The Corlette-Donaldson Theorem}

In this section we prove the following
\begin{theorem}[Corlette {\cite{Corlette88}}, Donaldson {\cite{Donaldson87}}, Jost-Yau {\cite{JostYau91}}, Labourie {\cite{Labourie91}}] \label{thm:corlette}
Let $\rho:\pi\to \SL_n(\CBbb)$ be semisimple. 
 Then there exists a 
$\rho$-equivariant harmonic map $u:\widetilde X\to D$.
\end{theorem}

The following result can be compared to Lemma \ref{lem:minimizing}. 
It will be proven when we discuss the harmonic map flow in the next section.
\begin{lemma} \label{lem:lipschitz}
For any $\rho:\pi\to \SL_n(\CBbb)$ there is a sequence $u_j$ of 
 $\rho$-equivariant maps $u_j:\widetilde X\to D$ satisfying
the conditions:
\begin{enumerate}
\item $u_j$ is energy minimizing.
\item 
 The $u_j$ have a  uniformly bounded Lipschitz constant.
\item $\tau(u_j)\to 0$ in $L^2$.
\end{enumerate}
\end{lemma}

\begin{lemma} \label{lem:reductive}
Let $\rho:\pi\to \SL_n(\CBbb)$ be irreducible, and let $u_j: 
\widetilde X\to D$ be a sequence of $\rho$-equivariant maps with a 
uniform Lipschitz constant.  Then $u_j(\hat p)$ is bounded.
\end{lemma}

\begin{proof}
Suppose not.
Set $h_j=u_j(\hat p)$ and choose $\varepsilon_j\to 0$ such that
 (perhaps after passing to a subsequence) $\varepsilon_j h_j\to 
h_\infty\neq 0$.  Notice that $\det h_\infty=0$, so 
$V=\ker h_\infty$ is a proper subspace of $\CBbb^n$.  I 
claim $\rho(\pi)$ fixes $V$.  
Indeed, if $\rho(\gamma)=g^{-1}$ and $v\in V$, then since $d(u_j(\hat p), u_j(\hat p)\cdot g^{-1})$ is uniformly bounded we have
$$
| \langle  w, v h_j\rangle_{\CBbb^n}-\langle w, v g h_j g^\ast\rangle_{\CBbb^n}|\leq B\ ,
$$
for a constant $B$ independent of $j$, and all $w\in \CBbb^n$.  It follows that
$$
| \langle w, v\varepsilon_j h_j \rangle_{\CBbb^n}-\langle wg, vg  \varepsilon_j h_j \rangle_{\CBbb^n}|\lra 0\ ,
$$
and since $vh_\infty=0$ we conclude that $\langle wg, vg h_\infty \rangle_{\CBbb^n}=0$.  Since $w$ was arbitrary, $v g\in V$.
\end{proof}

\begin{proof}[Theorem \ref{thm:corlette}]
By induction it suffices to prove the result for irreducible representations.
Let $u_j$ be a minimizing sequence as in Lemma \ref{lem:reductive},
the existence of which is  guaranteed by Lemma
\ref{lem:lipschitz}.  It follows from Ascoli's theorem
that  there is a uniformly convergent subsequence, also
denoted $u_j$, with the limit $u_j\to u_\infty$ a Lipschitz
$\rho$-equivariant map. 
I claim that we may arrange for $u_\infty$ to be a harmonic
map. Indeed,
  since the convergence is uniform, we
may choose local coordinates and write $u^a$. Then since
$|du^a|$ is uniformly bounded, we may assume further 
 that $u_j\to
u_\infty$ weakly in $L^2_{1,loc.}$.
By the condition in Lemma \ref{lem:lipschitz} (iii), 
the coordinates  $u_\infty^a$ are in $L^2_{1,loc.}$ and form 
 a weak solution of \eqref{eqn:harmonic-map-local}.  
Since $u_\infty$ is Lipschitz, elliptic regularity of the
Laplace operator implies $u_\infty\in L^2_{2,loc.}$.
By the remark following \eqref{eqn:harmonic-map-local}, we may
assume that the local metric on $X$ is euclidean.  Now
differentiate to obtain:
\begin{align*}
\Delta(\nabla u^a_\infty)+\nabla(\Gamma^a_{bc}(u_\infty)\nabla
u^b_\infty\cdot\nabla u^c_\infty)&= 0\ ;\\
\Delta(\nabla^2 u^a_\infty)+\nabla^2(\Gamma^a_{bc}(u_\infty)\nabla
u^b_\infty\cdot\nabla u^c_\infty)&= 0\ .
\end{align*}
Notice that since $u_\infty$ is Lipschitz
 the second term in the first equation is in $L^2$.
It then follows that $u^a_\infty\in L^2_{3,loc.}$.  Because of the
inclusion $L^2_3\hookrightarrow L^4_2$, the second term of
the second equation above is then in $L^2$.  This in turn implies
$u^a_\infty\in L^2_{4,loc.}$.  Finally, $L^2_4 \subset C^{2,\alpha}$,
and so $u_\infty$ is a strong solution to the harmonic map equations
\eqref{eqn:harmonic-map}. 
This completes the proof.
\end{proof}

\subsubsection{The harmonic map flow}
The harmonic map flow is defined by 
\begin{equation} \label{eqn:harmonic-map-flow}
\dot u= -\tau(u)\ .
\end{equation}
Here $u_t$ is a family of $\rho$-equivariant maps.
Since $D$ has non-positive curvature, the flow is
very well-behaved.
Long time existence  is proven in \cite{EellsSampson64, Hamilton75}.

The variation of the energy along the flow is given by 
$$
\frac{d}{dt}E(u_t)=2\int_X \langle du, d\dot u\rangle
=2\int_X \langle d_\nabla^\ast du, \dot u\rangle\omega
=-2\int_X |\tau(u)|^2\omega \ .
$$
In particular, \emph{energy decreases along the flow}.
Moreover,
\begin{equation} \label{eqn:tau-bound}
2\int_0^\infty dt\int_X |\tau(u_t)|^2\omega\leq E(u_0)\ .
\end{equation}

We are now ready for the 

\begin{proof}[Proof of Lemma \ref{lem:lipschitz}]
The proof
is based on the famous 
Eells-Sampson-Bochner 
formula for the change of the energy
 density along the harmonic map flow \cite{EellsSampson64}.  Let $u=u(t,x)$ be a solution to  \eqref{eqn:harmonic-map-flow}, and $e=e_u(t,x)$. Then
$$
-\frac{\partial e}{\partial t}+\Delta e=
|\nabla du |^2+\Ric_X(du, du)-\Riem_D(du,du,du,du)
$$ 
Now since $\Riem_D\leq 0$ and $\Ric_X$ is
 bounded below a negative constant, we have
$$
\frac{\partial e}{\partial t}-\Delta e \leq C\cdot e\ .
$$
Using an explicit argument with the heat kernel, 
this inequality along with the fact that energy is decreasing imply an estimate of the following type
\begin{equation} \label{eqn:energy-bound}
\sup e_{u_t}\leq C\cdot E_{u_0}\ ,
\end{equation}
for $t\geq 1$, say, 
where $C$ is depends only on the geometry of $X$ and $D$.

Now
let $u^{(j)}$ be an energy minimizing sequence 
of $\rho$-equivariant maps.  Let $u^{(j)}_t$ 
be the corresponding maps after the time $t$ flow of
 \eqref{eqn:harmonic-map-flow}.   
 Then since  energy is decreasing along the flow,
$u^{(j)}_{t_j}$ is 
also energy minimizing for any choice of sequence $t_j$.  
On the other hand,
 the right hand of \eqref{eqn:energy-bound} is uniformly bounded,
 so if we choose each $t_j\geq 1$, say, then  $u^{(j)}_{t_j}$ 
is also uniformly Lipschitz. Finally, for each fixed initial
condition $u_0$, \eqref{eqn:tau-bound} implies $\tau(u_{t_j})\to
0$ in $L^2$ along some sequence.
By a diagonalization argument we can arrange for
$u^{(j)}_{t_j}$ to satisfy this property as well.
\end{proof}

\subsection{Hyperk\"ahler reduction}

\subsubsection{The moduli spaces are real isomorphic}
Using \eqref{eqn:psi}, given a hermitian metric we may identify the space of all connections
$$
\Ccal_E=\left\{ (A,\Psi) \in \Acal_E\times  \Omega^1(M,\sqrt{-1}\gfrak_E) \right\}\ .
$$
Then  $\Ccal_E$
 is a \emph{hyperk\"ahler manifold}, and the action of the gauge group $\Gcal$ has associated moment maps
\begin{equation} \label{E:momentmaps}
\mu_1(A,\Psi) = F_A+\tfrac{1}{2}[\Psi,\Psi] \ ,\
\mu_2(A,\Psi)= d_A\Psi \ ,\
\mu_3(A,\Psi)=  d_A(\ast\Psi)\ .
\end{equation} 
 Let  ${\bf m}=(\mu_1,\mu_2,\mu_3)$.
 The hyperk\"ahler quotient is by definition
 $$
{\bf m}^{-1}(0) \bigr/ \Gcal
= \mu_1^{-1}(0)\cap \mu_2^{-1}(0)\cap \mu_3^{-1}(0) \bigr/ \Gcal_E\ .
$$ 
  
The two pictures we have been discussing above are equivalent to  a reduction of $\Ccal_E$ in steps, but  in two different ways.  The first is the  point of view of Hitchin and Simpson described in Section \ref{sec:hitchin}.  Namely, the space of Higgs bundles is given by 
$$
\Bcal_E=\mu_2^{-1}(0)\cap \mu_3^{-1}(0)\subset \Ccal_E\ ,
$$
where the relationship between $\Psi$ is obtained from  $\Phi$ by $\Psi=\Phi+\Phi^\ast$,
and conversely $\Phi$ is the $(1,0)$ part of $\Psi$.
 Just like for functions on surfaces, $\Psi$ harmonic if and only if $\Phi$ is holomorphic.
Now Theorem \ref{thm:hitchin} guarantees that the orbit of every polystable Higgs bundle intersects
locus $\mu_1^{-1}(0)$ in $\Bcal^{ss}$.  Hence, we have
$$
\Mfrak_D^{(n)}=\Bcal_E^{ss}\doubleslash \Gcal_E^\CBbb
={\bf m}^{-1}(0) \bigr/ \Gcal_E
= \mu_1^{-1}(0)\cap \mu_2^{-1}(0)\cap \mu_3^{-1}(0) \bigr/ \Gcal_E\ .
$$ 

The second point of view (e.g.\ Corlette and Donaldson, Section \ref{sec:corlette}) comes from the observation that the space of flat connections is
$$
\Ccal_E^{flat}=\mu_1^{-1}(0)\cap \mu_2^{-1}(0) \subset \Ccal_E\ .
$$
Given $\nabla\in \Ccal_E^{flat}$,
the condition that the associated $\hol(\nabla)$-equivariant map be harmonic is precisely that $\nabla\in \mu_3^{-1}(0)$.  Indeed, suppose  $\delta\nabla$ is a variation of $\nabla$.
It follows from  \eqref{eqn:psi} that $\delta\Psi=\delta\nabla+(\delta\nabla)^\ast$. In the case of a complex gauge transformation with $g^{-1}\delta g=\phi$, $\delta\nabla=\nabla\phi$, and
$$
\delta\Psi=d_A(\phi+\phi^\ast)+[\Psi, \phi-\phi^\ast]\ .
$$
It is easy to see that  the second term will not contribute in the variation
$
\tr(\delta\Psi\wedge \ast\Psi)+\tr(\Psi\wedge\ast\delta\Psi)
$ (by direct computation, and also from the fact that unitary gauge transformations do not vary the associated equivariant map).
So from Lemma \ref{lem:energy-psi} we have 
\begin{align*}
\delta E(u)&=4\int_X \tr(\delta\Psi\wedge \ast\Psi)+\tr(\Psi\wedge\ast\delta\Psi) \\
&=4\int_X \tr(d_A(\phi+\phi^\ast)\wedge \ast\Psi)+\tr(\Psi\wedge\ast d_A(\phi+\phi^\ast)) \\
&=-8\int_X\tr((\phi+\phi^\ast)d_A(\ast\Psi))\ .
\end{align*}
Since $\Psi$ is hermitian and $\phi$ is arbitrary, $\Psi$ is a critical point for the energy  if and only if $d_A(\ast\Psi)=0$.

Now Theorem \ref{thm:corlette} guarantees that the orbit of every semisimple representation contains a harmonic map.  It therefore follows that the holonomy map gives a homeomorphism
$$
\Mfrak_B^{(n)}\simeq \Ccal_E^{flat}\doubleslash\Gcal_E^\CBbb\simeq \mu_1^{-1}(0)\cap \mu_2^{-1}(0)\cap \mu_3^{-1}(0) \bigr/ \Gcal_E\ .
$$
So the Dolbeault and Betti moduli spaces coincide!
\begin{theorem}[\cite{Simpson94a, Simpson94b}]
The identification above gives a homeomorphism $\Mfrak^{(n)}_D\simeq \Mfrak^{(n)}_B$.
\end{theorem}

\subsubsection{Equivariant cohomology}

As in the case of the $\YMH$-flow, the harmonic map flow actually has continuity properties as $t\to \infty$.  To describe this, 
let $\Gcal_E(p)\subset \Gcal_E$ denote the subgroup of gauge transformations that are the identity at the point $p$.
Now  the holonomy map gives a proper embedding
\begin{equation} \label{eqn:hol-embedding}
\hol : {\bf m}^{-1}(0)/\Gcal_E(p)\hookrightarrow \Hom(\pi, \SL_n(\CBbb))\ ,
\end{equation}
which is  $\SU_n$-equivariant. 
\begin{theorem}[{cf.\ \cite{DWW}}] \label{thm:representation-retract}
The inclusion \eqref{eqn:hol-embedding}
is  an $\SU_n$-equivariant deformation retract.
\end{theorem}

An explicit retraction is defined using  the harmonic map flow to define a flow on the space of representations.  
   Fix a lift $\tilde p\in \widetilde X$ of $p$.
Given $\rho\in \Hom(\pi, \SL_n(\CBbb))$, choose  $\nabla\in \Ccal_E^{flat}$ with $\hol(\nabla)=\rho$. 
The hermitian metric  gives a unique $\rho$-equivariant lift $u:\widetilde X\to D$ with $u(\tilde p)=I$.  Let $u_t$, $t\geq 0$, denote the solution to \eqref{eqn:harmonic-map-flow} with initial condition $u$.  There is a unique continuous family $h_t\in \SL_n(\CBbb)$, $h_t^\ast=h_t$, such that $h_0=I$, and $ h_tu_t(\tilde p)=z$.  Notice that a different choice of flat connection $\widetilde \nabla$ with $\hol(\widetilde \nabla)=\rho$ will be related to $\nabla$ by a based gauge transformation $g$.  The  flow corresponding to $\widetilde \nabla$ is  $\tilde u_t=g\cdot u_t$, and since $g(\tilde p)=I$,  $\tilde h_t=h_t$.  Hence, $h_t$ is well-defined by $\rho$.  The flow on $\Hom(\pi, \SL_n(\CBbb))$
is then defined by $\rho_t=h_t\rho h_t^{-1}$.
The result states that this flow defines a continuous retraction to $\hol\left({\bf m}^{-1}(0)/\Gcal_E(p) \right) $.  When $\rho$ is not semisimple, the flow converges  to a semisimplification.

This result has consequences for computing the equivariant cohomology of moduli space \cite{AtiyahBott82, Kirwan84, Daskalopoulos92}.
 In particular, Theorem \ref{thm:representation-retract} implies
$$H^\ast_{\SU_n}({\bf m}^{-1}(0)/\Gcal_E(p))\simeq H^\ast_{\SU_n}(\Hom(\pi, \SL_n(\CBbb))\ .$$
Note that since $\SL_n(\CBbb)/\SU_n$ is contractible,  on the right hand side we may take equivariant cohomology with respect to $\SL_n(\CBbb)$.
On the other hand, Theorem \ref{thm:wilkin} implies
$$H^\ast_{\SU_n}({\bf m}^{-1}(0)/\Gcal_E(p))= H^\ast_{\Gcal_E}(\Bcal^{min}_E) \simeq
H^\ast_{\Gcal_E}(\Bcal^{ss}_E)\ .$$
It follows that the equivariant cohomology of the space of representations may be computed by studying the equivariant Morse theory of $\YMH$ on $\Bcal_E$ in the spirit of \cite{AtiyahBott82}.  This is complicated, since $\Bcal_E$ is singular. Some progress has been made using this approach (see \cite{DWWW, WentworthWilkin11}.

Figure 1 gives a  cartoon of $\Ccal_E$ with the subspaces $\Ccal_E^{flat}$ and $\Bcal_E$, and the flows that have been defined.

\bigskip
\setlength{\unitlength}{1cm}
\begin{picture}(14,8)
\put(4, 0){
{\scalebox{.65}{\includegraphics{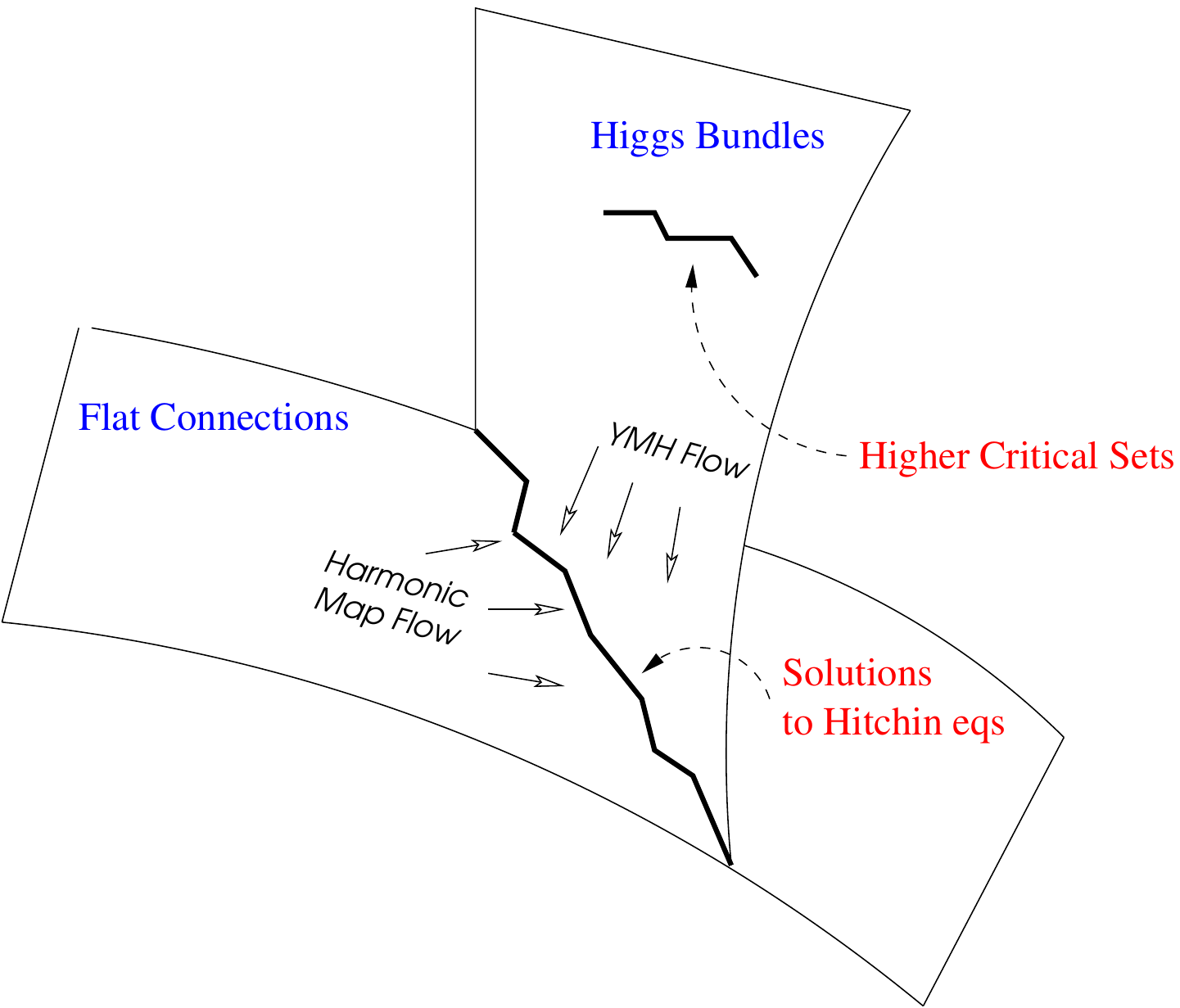}}}}
\put(7,0){Figure 1.}
\end{picture}

\section{Differential Equations}

\subsection{Uniformization}
For more on  the discussion in this  section I refer to the classic text of Gunning \cite{Gunning66}.

\begin{definition}
The {\bf Schwarzian derivative} of a univalent
 holomorphic function  $f(z)$ defined on a domain in $\CBbb$ is given by 
$$
S(f)=\{f,z\}= \frac{f'''}{f'}-\frac{3}{2}\left(\frac{f''}{f'}\right)^2\ .
$$
\end{definition}

\noindent
By straightforward calculation one shows the following:
\begin{enumerate}
\item $S(f\circ g)=(S(f)\circ g) (g')^2+ S(g)\ ;$
\item $S(f)=0 \iff f$ is the restriction of a M\"obius transformation.
\end{enumerate}
A particular consequence of (i) and (ii) is then
$$
\hspace{-10cm} {\rm (iii)}\ 
  S(f)=S(g) \Longrightarrow f=\phi\circ g\ ,
  $$
   where $\phi$ is  a M\"obius 
transformation.

The Schwarzian derivative gives a link between uniformization
and the monodromy of differential equations, as I briefly
explain here. Let $Q(z)$, $y(z)$ be locally defined
holomorphic functions, and consider the ODE
\begin{equation} \label{eqn:2nd-order-ode}
y''(z)+Q(z)y(z)=0\ .
\end{equation}
If $y_1, y_2$ are independent solutions of
\eqref{eqn:2nd-order-ode} and $y_2\neq 0$, then a calculation
shows that $f=y_1/y_2$ satisfies $S(f)=2Q$.

Note that for a univalent function $f$,  $S=S(f)$ is not quite a tensor: rather, by (i) it transforms
with respect to local coordinate changes as
\begin{equation} \label{eqn:projective-connection}
S(w)(w')^2=S(z)-\{w,z\}\ ,
\end{equation}
so $S$ nearly transforms as a quadratic differential.
A collection $\{S(z)\}$ of local holomorphic functions  on
$X$
transforming as in \eqref{eqn:projective-connection} is called
a  {\bf projective connection}.
The space of projective connections on $X$ is an affine space 
modeled on the space $H^0(X,\Kcal^2)$ of holomorphic quadratic
differentials.

Next, consider the transformation properties of the solutions
$y$ to \eqref{eqn:2nd-order-ode}, where $2Q=S$ is an arbitrary
projective connection  (cf.\ \cite{HawleySchiffer66}). If we assume $y$ is a local holomorphic
section of $\Kcal^{-1/2}$, then we have
\begin{align*}
y(z)&= y(w) (w')^{-1/2}\ ; \\
y''(z) &= y''(w)(w')^{3/2}-\tfrac{1}{2}y(z)\{w,z\}\ ,
\end{align*}
and so 
$$
y''(z)+\tfrac{1}{2}S(z)y(z)=(y''(w)+\tfrac{1}{2}S(w)y(w))(w')^{3/2}\ .
$$
We deduce that $Dy=y''+\tfrac{1}{2}Sy$ gives a well-defined
map of $\underline\CBbb$-modules $D:\Kcal^{-1/2}\to
\Kcal^{3/2}$.
Therefore, given a projective connection
$S$
we have a rank $2$ local system $\Vbold$, defined by the solution space to \eqref{eqn:2nd-order-ode}, $2Q=S$. Moreover, there is an exact sequence
of $\underline \CBbb$-modules
$$
0\lra \Vbold \lra \Kcal^{-1/2}\lra \Kcal^{3/2}\lra 0\ .
$$

Now assume $X$ has a  uniformization as a hyperbolic surface.
So 
$\rho_F:\pi\to
\PSL_2(\RBbb)$ is a discrete and faithful representation such that $X$ is biholomorphic to $\HBbb/\rho_F(\pi)$.
Let $u$ be a (multi-valued) inverse of the quotient map
$\HBbb\to X$. In other words, $u$ is a univalent function $u:\widetilde X\to \HBbb$
that is equivariant with respect to $\rho_F$.
 Set $S_F(z)=S(u)(z)$.  Then by items (i) and
(ii) above, for any $\gamma\in\pi$,
$$
S_F(\gamma z)=S(u)(\gamma z)=S(\rho_F(\gamma)\circ
u)(z)=S(u)(z)=S_F(z)\ .
$$
So $S_F$ is a well-defined projective connection on $X$.

Now the key point is the following: if $y_1, y_2$ are linearly independent solutions to
\eqref{eqn:2nd-order-ode} where $2Q=S_F$, 
then $S(y_1/y_2)=S(u)$ and so
by (iii) above there is a M\"obius transformation $\phi$ such that $y_1/y_2=\phi\circ u$. 
It follows that the (projective) monodromy of the local system associated 
 to \eqref{eqn:2nd-order-ode} in the case $2Q=S_F$
is conjugate  to $\rho_F$. 
If $S$ is any  fixed choice of projective connection, one may ask for the holomorphic quadratic
differential $Q$ such that $S_F=S+Q$. This is the famous problem of
\emph{accessory parameters} (cf.\ \cite{Uniformization}).

\begin{remark} \label{rem:projective}
I want to clarify the following issue:
the bundle $\Kcal^{1/2}$ involves a choice of square root 
 of
the canonical bundle (i.e.\ a {\bf spin structure}), of which there are $2^{2g}$ possibilities.  This choice is \emph{precisely}
equivalent to a lift of  the corresponding monodromy $\rho$ from 
$\PSL_2(\CBbb)$
to $\SL_2(\CBbb)$.  To see this, let
$\Vcal_\rho=\Ocal\otimes_{\underline \CBbb}\Vbold_\rho$,  and notice
that $\Vcal_\rho$ fits into an exact sequence (now of $\Ocal$-modules)
\begin{equation} \label{eqn:v-rho}
0\lra \Kcal^{1/2}_\rho\lra \Vcal_\rho\lra \Kcal^{-1/2}_\rho\lra 0\ ,
\end{equation}
where now we also label the choice of spin structure by $\rho$.
Since $\Vcal_\rho$ has a holomorphic connection, by Theorem \ref{thm:weil},
\eqref{eqn:v-rho} cannot split.
On the other hand, the extensions are parametrized by the projective space of $H^1(X,\Kcal)\simeq (H^0(X,\Ocal))^\ast=\CBbb$.  So all the bundles $\Vcal$ obtained in this way as $\rho$ varies are isomorphic, modulo the choice of $\Kcal^{1/2}$.
Eq.\ \eqref{eqn:v-rho} also implies that $\Vcal^\ast_\rho\otimes \Kcal^{-1/2}_\rho$ has a nonzero
holomorphic section.  Moreover,  if we have such an exact sequence
  for one spin structure, then \eqref{eqn:v-rho} cannot hold for any other
choice $\Kcal^{1/2}$.
Indeed,  the induced map $\Kcal^{1/2}\to\Kcal^{-1/2}_\rho$ would necessarily vanish, and so the inclusion $\Kcal^{1/2}\to\Vcal_\rho$ would lift to give an isomorphism
 $\Kcal^{1/2}\simeq \Kcal_\rho^{1/2}$. So $\Kcal^{-1/2}$ is determined by $\rho$.
Changing the lift of the projective  monodromy $\rho$ to $\SL_2(\CBbb)$ amounts to
$\rho\mapsto \rho\otimes \chi$ for some character $\chi:\pi\to \ZBbb/2$. This  
 corresponds to tensoring $\Vcal_\rho$ by a flat line
bundle $\Lcal_\chi$ whose square is trivial. It follows that 
from the condition that $H^0(X,\Vcal_{\rho\otimes \chi}^\ast\otimes
\Kcal^{-1/2}_{\rho\otimes \chi})\neq \{0\}$, and the argument given above, that  $\Kcal^{1/2}_{\rho\otimes \chi}=\Kcal^{1/2}_\rho\otimes\Lcal_\chi$. 
\end{remark}

\subsection{Higher order equations}

\subsubsection{Invariance properties}
The structure outlined in the previous section for equations of the type \eqref{eqn:2nd-order-ode} extends to higher order equations.  
We consider $n$-th order differential equations on $\HBbb$:
\begin{equation} \label{eqn:n-diffeq}
y^{(n)}+ Q_2 y^{(n-2)} +\cdots + Q_n y=0\ .
\end{equation}
We would like an appropriate invariance property under coordinate changes in order to have solutions that are intrinsic to $X$.  Motivated by the example of projective connections, we attempt to realize local solutions of \eqref{eqn:n-diffeq} in the sheaf $\Kcal^{1-q}$,
where $n=2q-1$ and we have chosen a spin structure if $q$ is a half-integer.
Solutions to \eqref{eqn:n-diffeq} are given by the kernel of an operator 
$\Kcal ^{1-q}
 \stackrel{D}{\xrightarrow{\hspace*{.75cm}}}  \Kcal ^q
$.

\begin{theorem} [cf.\ \cite{Itzykson91}, see also \cite{Wilczynski62, Hejhal75}]\label{thm:wk}
Let $D: \Kcal ^{1-q}\to \Kcal^q$ be $\underline\CBbb$-linear and locally of the form 
$$
Dy=y^{(n)}+ Q_2 y^{(n-2)} +\cdots + Q_n y\ .
$$
  Then 
$12Q_2/n(n^2-1)$
is a projective connection,
and for $k\geq 3$,
there exist $w_k$, linear combinations of $Q_j$, $j=2,\ldots, k$ and derivatives, with coefficients polynomials in $Q_2$,
 such that $w_k$ transform as a $k$-differentials. Conversely, given one such operator  and $k$ differentials $w_k$, $k=2,\ldots, n$, these conditions uniquely determine an operator $D$.
\end{theorem}

\noindent 
The expressions for the $w_k$ are quite complicated.  For example, we reproduce some of   \cite[Table 1]{Itzykson91}:
\begin{align}
\begin{split} \label{eqn:wk}
w_2&= Q_2 \\
w_3&=Q_3-\frac{n-2}{2}Q_2' \\
w_4&= Q_4-\frac{n-3}{2}Q_3' + \frac{(n-2)(n-3)}{10}Q''_2-\frac{(n-2)(n-3)(5n+7)}{10n(n^2-1)}Q^2_2\ .
\end{split}
\end{align}

 It follows from Theorem \ref{thm:wk} that  the space of all such $D$ is an affine space modeled on the Hitchin base
$\bigoplus_{j=2}^n H^0(X,\Kcal^j)$.
The map $D:  \Kcal ^{1-q}\to  \Kcal ^{q}$ is clearly locally surjective.  Moreover, the Wronskian of any fundamental set of solutions $Dy_i=0$ is constant.
 We therefore obtain a local system ${\bf V}$ and an exact sequence of sheaves over $\underline\CBbb$.

\begin{equation} \label{eqn:n-diffop}
0\lra \Vbold \stackrel{\varphi}{\xrightarrow{\hspace*{.75cm}}}  \Kcal ^{1-q}
 \stackrel{D}{\xrightarrow{\hspace*{.75cm}}}  \Kcal ^q \lra 0\ .
\end{equation}
In this situation, 
we say that the local system $\Vbold$ {\bf is  realized in} $\Kcal ^{1-q}$.  

\begin{remark}
If we tensor by a line bundle with a holomorphic connection and replace derivatives $y^{(j)}$ with derivatives in a local parallel frame of the line bundle, then we can consider local systems realized in $\Lcal$:
\begin{equation} \label{eqn:ell}
0\lra \Vbold \stackrel{\varphi}{\xrightarrow{\hspace*{.75cm}}}  \Lcal
 \stackrel{D}{\xrightarrow{\hspace*{.75cm}}}  \Lcal\otimes\Kcal ^n \lra 0\ ,
\end{equation}
where $\deg \Lcal=-(n-1)(g-1)$. 
% This will have a global polymorphic frame only if and only if the corresponding Wronskian takes values in the trivial sheaf, or
%\begin{equation} \label{eqn:diffop-condition}
%\Lcal^n=\Kcal ^{-n(n-1)/2}
%\end{equation}
\end{remark}

\subsubsection{The Riemann-Hilbert correspondence}

The goal of this section is to characterize which local systems can be realized as the monodromy of solutions to differential equations.  To motivate the following, if $\Vbold$ is a local system realized in $\Lcal$, and $\Vcal=\Ocal\otimes_{\underline\CBbb}\Vbold $,
 notice  that in \eqref{eqn:ell} there is a surjective sheaf map $\Vcal\to \Lcal$ given by 
$f\otimes \vbold\mapsto f \varphi(\vbold)$, for $f\in \Ocal $, $\vbold\in \Vbold$.  In particular, $\Vcal^\ast\otimes\Lcal$ has a nonzero holomorphic section.

\begin{theorem} \label{thm:ohtsuki}
A representation $\rho: \pi\to \SL_n(\CBbb)$ can be realized in $\Lcal$ if and only if $\rho$ is irreducible,
$H^0(X,\Vcal_\rho^\ast\otimes\Lcal)\neq\{0\}$, and $\Lcal^n=\Kcal ^{-n(n-1)/2}$.
\end{theorem}

\begin{proof}
According to Hejhal \cite[Theorem 3]{Hejhal75}, 
the monodromy representation arising from a differential 
operator $D$ is necessarily irreducible.  
 I shall give a proof of this fact below
 (see Proposition \ref{prop:irreducible}).  Accepting this point for the time being, 
 from the discussion above we also have a nonzero section of  $ \Vcal_\rho^\ast\otimes\Lcal$. 
 Moreover, if $y_1,\ldots, y_n$ is an independent set of solutions $Dy_i=0$ on $\HBbb$, then the Wronskian
 $$
 W(y_1,\ldots y_n)=\det\left( \begin{matrix}
 y_1 & \cdots & y_n \\
  y'_1 & \cdots & y'_n \\
  \vdots && \vdots \\
   y^{(n-1)}_1 & \cdots & y^{(n-1)}_n 
   \end{matrix}
 \right)\ ,
 $$
 is a well-defined nowhere vanishing global holomorphic section of $\Lcal^n\otimes\Kcal^{n(n-1)/2}$ on $X$.  The latter is therefore trivial.
 This proves the necessity part of the assertion.  
For the converse, we follow a classical argument using the Wronskian (cf.\
\cite{Ohtsuki82}).
 Assume we have a nonzero holomorphic section $\varphi$ of
$\Vcal_\rho^\ast\otimes\Lcal$. This induces a map (also
denoted by $\varphi$):
$\Vbold_\rho\to \Lcal$.  Because $\rho$ is irreducible,
$\varphi$ is injective.
Because $\Lcal^n=\Kcal^{-n(n-1)/2}$ we can write
$\Lcal=\Lcal_0\otimes \Kcal^{-(n-1)/2}$, where $\Lcal_0$ has a
flat connection.  If we express a section of $\Lcal$ as $\lbold\otimes w$, where $\lbold$ is a parallel section of
$\Lcal_0$, then we define $y'={\bf l}\otimes w'$.  With this
understood, choose a local frame $\{{\bf v}_i\}$ for
$\Vbold_\rho$, and  set
 $$
Dy
=\det\left( \begin{matrix}
 \varphi({\bf v}_1) & \cdots & \varphi({\bf v}_n) & y\\
 \varphi({\bf v}_1)' & \cdots & \varphi({\bf v}_n)'& y' \\
  \vdots && \vdots&\vdots \\
  \varphi({\bf v}_1)^{(n)} & \cdots &\varphi({\bf v}_n)^{(n)}  & y^{(n)}
   \end{matrix}
 \right)\ .
 $$
Then if $y$ is a local holomorphic section of $\Lcal$, $Dy$ is
a well-defined local holomorphic section of
$\Lcal^{n+1}\otimes\Kcal^{n(n+1)/2}=\Lcal\otimes\Kcal^n$.
Clearly, the kernel of $D$ is precisely $\Vbold_\rho$.  Moreover, since the monodromy of $\Vbold_\rho$ is in $\SL_n(\CBbb)$, it is easy to see that $Dy$ is actually globally defined on $X$.  Finally,
$\Lcal^n=\Kcal^{-n(n-1)/2}$, so
 $$
\det\left( \begin{matrix}
 \varphi({\bf v}_1) & \cdots & \varphi({\bf v}_n) \\
 \varphi({\bf v}_1)' & \cdots & \varphi({\bf v}_n)' \\
  \vdots && \vdots \\
  \varphi({\bf v}_1)^{(n-1)} & \cdots &\varphi({\bf v}_n)^{(n-1)}  
   \end{matrix}
 \right)\ ,
 $$
 is a nonzero holomorphic function on $X$, which may therefore be set equal to $1$.
Hence, $Dy$ has the form \eqref{eqn:n-diffop}.
This completes the proof.
\end{proof}

\begin{example} \label{ex:fuchsian}
The lift of the monodromy of a  projective connection defines a representation into $\SL_2(\CBbb)$ which, via the irreducible embedding $\SL_2\hookrightarrow\SL_n$, gives a representation into $\SL_n(\CBbb)$.  
%In the case of the Fuchsian uniformization, we shall also call this a {\bf Fuchsian representation}. 
It is straightforward, if somewhat tedious, to calculate the differential equations associated to  the local systems arising in this way.  Below are some examples where we let $2Q$ to be a  projective connection on $X$.

\begin{itemize}
\item $n=2$: $y''+Qy=0$;
\item $n=3$: $y'''+4Qy'+2Q'y=0$;
\item $n=4$: $y^{(4)}+10Qy''+10Q'y'+(9Q^2+3Q'')y=0$;
\item $n=5$: $y^{(5)}+20Qy'''+30Q'y''+(64Q^2+18Q'')y'+(64QQ'+4Q''')y=0$;
\item $n=6$: $y^{(6)}+35Qy^{(4)}+70Q'y'''+(63Q''+259Q^2)y''+
(28Q'''+518QQ')y'+$ \\
${}$\hskip3in $(130(Q')^2+155QQ''+5Q^{(4)}+225Q^3)y=0$.
\end{itemize}
Note that $w_3, w_4$ in \eqref{eqn:wk} vanish for these examples.
\end{example}

\subsection{Opers} \label{sec:opers}

\subsubsection{Oper structures}
In this section we introduce opers.  For more details  consult \cite{BeilinsonDrinfeld05, 
BenZviFrenkel01a, BenZviFrenkel01b, BenZviBiswas04, JoshiPauly09, Simpson10}.

\begin{definition}[Beilinson-Drinfeld {\cite{BeilinsonDrinfeld05}}] \label{def:oper}
An $\SL_n$-{\bf oper} is a holomorphic bundle $\Vcal \to X$, a holomorphic connection $\nabla$ inducing the trivial connection on $\det\Vcal$, and a filtration
$
0=\Vcal_0\subset \Vcal_1\subset\cdots\subset \Vcal_n=\Vcal
$
 satisfying 
\begin{enumerate}
\item $\nabla\Vcal_i\subset \Vcal_{i+1}\otimes \Kcal $;
\item the induced $\Ocal$-linear map $\Vcal_i/\Vcal_{i-1} \stackrel{\nabla}{\xrightarrow{\hspace*{.6cm}}} 
\Vcal_{i+1}/\Vcal_{i}\otimes \Kcal $ is an isomorphism for $1\leq i\leq n-1$.
\end{enumerate}
\end{definition}

There is an action of $\Gcal^\CBbb$ on the space of opers which pulls back connections and filtrations.
Let $\Op_n$  denote the space of gauge equivalence classes of 
$\SL_n$-opers on $X$.  Given a holomorphic connection on a bundle $\Vcal$, we shall call a filtration 
$
0=\Vcal_0\subset \Vcal_1\subset\cdots\subset \Vcal_n=\Vcal
$
satisfying (i) and (ii) an {\bf oper structure}.  Not every holomorphic connection admits an oper structure. For example, we have the following important

\begin{proposition} \label{prop:irreducible}
The holonomy representation of an oper is irreducible.
\end{proposition}

First we have

\begin{lemma} \label{lem:gr}
For any $\SL_n$-oper, $\det \Vcal_j\simeq\Lcal^j\otimes \Kcal ^{nj-j(j+1)/2}$,
where $\Lcal\simeq \Vcal/\Vcal_{n-1}$, and $\Lcal^n\simeq \Kcal ^{-n(n-1)/2}$.
\end{lemma}

\begin{proof}
To simplify notation, set $v_i=\det\Vcal_i$, $\kappa=\Kcal $, and 
use additive notation for line bundle tensor products.  
Then Definition \ref{def:oper} (ii) 
gives $v_i-v_{i-1}=v_{i+1}-v_i+\kappa$, and so
\begin{align*}
v_j&=\sum_{i=1}^j (v_i-v_{i-1})=\sum_{i=1}^j (v_{i+1}-v_i+k)=v_{j+1}-v_1+j\kappa \\
v_{j+1}-v_j &=  v_1-j\kappa\ .
\end{align*}
Now summing again
\begin{align*}
v_i-v_1&=\sum_{j=1}^{i-1} (v_{j+1}-v_{j})=(i-1)v_1-\frac{i(i-1)}{2}\kappa \\
v_{i} &=  iv_1-\frac{i(i-1)}{2}\kappa \\
0&=v_n= nv_1-\frac{n(n-1)}{2}\kappa\ .
\end{align*}
Set $\Lcal=v_1-(n-1)\kappa$, and this completes the proof.
\end{proof}

\begin{proof}[Proof of Proposition \ref{prop:irreducible}] (cf.\ \cite{JoshiPauly09})
Suppose that $(\Vcal,\nabla)$ has an oper structure and $0\neq
\Wcal\subset
\Vcal$ is $\nabla$-invariant. Let $\Wcal_i=\Wcal\cap \Vcal_i$. 
I claim that the induced map 
$$
\Wcal_i/\Wcal_{i-1}\lra \Wcal_{i+1}/\Wcal_i\otimes \Kcal \ ,
$$
is an inclusion of sheaves for all $i=1, \ldots, n-1$.
Indeed, consider the commutative diagram of $\Ocal$-modules:
$$
\xymatrix{
\Wcal_i/\Wcal_{i-1} \ar[d] \ar[r]&\Wcal_{i+1}/\Wcal_i\otimes
\Kcal     \ar[d] \\
\Vcal_i/\Vcal_{i-1}  \ar[r]&\Vcal_{i+1}/\Vcal_i\otimes
\Kcal     
}
$$
The vertical arrows are inclusions and the lower horizontal
arrow is an isomorphism.  This proves the claim.
Set $r_i=\rank(\Wcal_i/\Wcal_{i-1})$. By the claim,   if
$r_i=0$, then $r_j=0$ for $j\leq i$.  Let $1\leq \ell\leq n$ be the smallest
integer for which  $r_\ell\neq 0$.  It follows that $r_i=1$ if
and only if $\ell\leq i\leq n$.

Applying the  inclusions recursively and using Lemma \ref{lem:gr}, we find
$$
\Wcal_i/\Wcal_{i-1}\hookrightarrow  \Vcal/\Vcal_{n-1}\otimes
\Kcal ^{n-i}\cong \Kcal ^{(n-2i+1)/2}\ .
$$
In particular (see Section \ref{sec:stability}),
$$
\deg(\Wcal_i/\Wcal_{i-1})\leq (n-2i+1)(g-1)\ ,
$$
and so 
$$
\deg \Wcal=\sum_{i=\ell}^n\deg(\Wcal_i/\Wcal_{i-1})
\leq \sum_{i=\ell}^n(n-2i+1)(g-1)=-(n-\ell+1)(\ell-1)(g-1)\ . 
$$
The right hand side is strictly negative unless $\ell=1$.
But since $\Wcal$ has a holomorphic connection induced by
$\nabla$, $\deg \Wcal=0$. Hence, the only possibility is $\ell=1$, which implies 
$\Wcal=\Vcal$. This completes the proof.
\end{proof}

We now show that if a holomorphic connection admits an
 oper structure, then that structure is
 unique up to gauge equivalence. 
For the next part of the discussion, it will be  useful to have the following diagram in mind (cf.\ Lemma \ref{lem:gr}):

\begin{equation} 
\begin{split} \label{eqn:big-diagram}
\xymatrix{
&&& 0 \ar[d] & \\
 & 0 \ar[d] &  & \Lcal\otimes\Kcal ^{n-j}  \ar[d] & \\
0\ar[r] & \Vcal_{j-1} \ar[r] \ar[d] & \Vcal \ar[r] \ar[d]^{\begin{sideways}$\sim$ \end{sideways}} & \Rcal_{j-1} \ar[r] \ar[d] & 0 \\
0\ar[r] & \Vcal_{j} \ar[r] \ar[d] & \Vcal \ar[r]  & \Rcal_{j} \ar[r] \ar[d] & 0 \\
& \Lcal\otimes \Kcal ^{n-j} \ar[d] && 0 & \\
& 0 &&&
}
\end{split}
\end{equation}

\begin{lemma} \label{lem:ext1}
$\displaystyle H^1(X, (\Lcal\otimes \Kcal ^{n-j})\otimes \Rcal_i^\ast)=\begin{cases}
0 & i\geq j+1\ ; \\ H^1(X,\Kcal ) & i=j\ . \end{cases}
$
\end{lemma}

\begin{proof}
Fix $j$ and do induction on $i$.  If $i=n-1$, then $\Rcal_{n-1}=\Lcal$ and
$$
H^1(X, (\Lcal\otimes \Kcal ^{n-j})\otimes \Rcal_{n-1}^\ast)=
H^1(X, \Kcal ^{n-j})=\begin{cases} 0 & n-j-1>0\ ; \\
H^1(X,\Kcal) & n=j+1\ .
\end{cases}
$$
Now the exact sequence $0\to\Rcal_i^\ast\to \Rcal_{i-1}^\ast\to\Lcal^\ast\otimes\Kcal^{i-n}\to 0$ gives the following
$$
H^1(X, (\Lcal\otimes \Kcal ^{n-j})\otimes \Rcal_{i}^\ast)
\lra H^1(X, (\Lcal\otimes \Kcal ^{n-j})\otimes \Rcal_{i-1}^\ast)\lra
H^1(X, \Kcal ^{i-j})\lra 0\ .
$$
By induction the first term vanishes and the last two terms are isomorphic.  This proves the lemma.
\end{proof}

\begin{lemma} \label{lem:ext2}
$\displaystyle H^1(X, \Vcal_j\otimes \Rcal_i^\ast)=\begin{cases}
0 & i\geq j+1\ ; \\ H^1(X,\Kcal ) & i=j \ .\end{cases}
$
\end{lemma}

\begin{proof}
Fix $i$ and induct on $j$.  Now $\Vcal_1\simeq \Lcal\otimes\Kcal ^{n-1}$, so the result in this case follows from Lemma \ref{lem:ext1}.  Next consider the exact sequence
$$
H^1(X,\Vcal_{j-1}\otimes \Rcal_i^\ast)\lra
H^1(X,\Vcal_{j}\otimes \Rcal_i^\ast)\lra
H^1(X,\Lcal\otimes \Kcal ^{n-j}\otimes \Rcal_i^\ast)\lra 0\ .
$$
By induction, the first term vanishes and so the second and third terms are isomorphic.  Again, the result follows from Lemma \ref{lem:ext1}.
\end{proof}

\begin{corollary} \label{cor:ext}
$\displaystyle H^1(X, \Vcal_{j-1}\otimes (\Lcal\otimes \Kcal ^{n-j})^\ast)= H^1(X,\Kcal )
$
\end{corollary}

\begin{proof}
Consider the exact sequence
$$
H^1(X,\Vcal_{j-1}\otimes \Rcal_j^\ast)\lra
H^1(X,\Vcal_{j-1}\otimes \Rcal_{j-1}^\ast)\lra
H^1(X,  \Vcal_{j-1}\otimes (\Lcal_{j-1}\otimes \Kcal ^{n-j})^\ast)\lra 0\ .
$$
By Lemma \ref{lem:ext2} the first term vanishes and the second is $\simeq H^1(X,\Kcal )$.
\end{proof}

\begin{lemma} \label{lem:ext3}
The extension 
$0\to \Vcal_{j-1}\to \Vcal_j\to \Lcal\otimes \Kcal ^{n-j}\to 0$,
 is non-split.
\end{lemma}

\begin{proof}
Consider the diagram:
\begin{equation} 
\begin{split} \label{eqn:ext-diagram}
\xymatrix{
  & H^1(X,\Vcal_{j-1}\otimes \Rcal_j^\ast) \ar[d] \\
  H^0(X,\Rcal_{j-1}\otimes \Rcal_{j-1}^\ast) \ar[d] \ar[r]^{I\mapsto[\beta]}  
  & H^1(X,\Vcal_{j-1}\otimes \Rcal_{j-1}^\ast) \ar[d]^g \\
   H^0(X,\Rcal_{j-1}\otimes (\Lcal\otimes \Kcal ^{n-j})^\ast)  \ar[r]  
  & H^1(X,\Vcal_{j-1}\otimes  (\Lcal \otimes\Kcal ^{n-j})^\ast) \ar[d] \\
   H^0(X,(\Lcal\otimes \Kcal ^{n-j})\otimes (\Lcal \otimes\Kcal ^{n-j})^\ast)  \ar[u] \ar[ur]_{I\mapsto [\alpha]}
   & 0
}
\end{split}
\end{equation}
By the comment following \eqref{eqn:extension}, 
$[\alpha]$ is the extension class of 
$0\to \Vcal_{j-1}\to \Vcal_j\to \Lcal\otimes \Kcal ^{n-j}\to 0$,
and $[\beta]$  is the extension class of  $0\to \Vcal_{j-1}\to \Vcal\to \Rcal_{j-1}\to 0$.
By Lemma \ref{lem:ext2}, $g$ is injective.  By tracing through the definition of the coboundary one has $[\alpha]=g[\beta]$.  Finally, since $\Vcal$ has a holomorphic connection and $\deg\Vcal_{j-1}\neq 0$
by Lemma \ref{lem:gr}, it follows from Theorem \ref{thm:weil} that   $[\beta]\neq 0$.
\end{proof}

Finally, we can state the result on the uniqueness of the underlying holomorphic structures.

\begin{proposition} \label{prop:bd}
Let $(\Vcal, \nabla)$ be an $\SL_n$-oper. 
 Then the oper structure on $\Vcal$
 is uniquely determined by $\Lcal=\Vcal/\Vcal_{n-1}$.  
In particular, the isomorphism class of the  bundle $\Vcal$ is
fixed on every connected component of $\Op_n$. 
\end{proposition}

\begin{proof}
By Lemma \ref{lem:gr}, $\Vcal_1=\Lcal\otimes\Kcal^{n-1}$, and so  is determined.  
By Corollary \ref{cor:ext} and Lemma \ref{lem:ext3}, 
each $\Vcal_{j}$ is successively determined by 
$\Vcal_{j-1}$ as the unique nonsplit extension of the 
sequence appearing in Lemma \ref{lem:ext3}.  Continuing 
in this way until $j=n$, this proves the first statement. The
second statement follows as well, 
since by Lemma \ref{lem:gr}
 we also have $\Lcal^n\simeq \Kcal^{-n(n-1)/2}$,  
and therefore 
 the set of  possible $\Lcal$'s
is  discrete.
\end{proof}

\begin{corollary} \label{cor:oper-embedding}
The map sending an oper to its  monodromy representation
gives an embedding $\Op_n\to \Mfrak_B^{(n)}$. 
\end{corollary}

\begin{proof}
Fix a representation $\rho: \pi\to \SL_n(\CBbb)$, and suppose
that up to conjugation $\rho$ is the monodromy of opers
$(\Vcal_\rho,\nabla_1)$ and $(\Vcal_\rho,\nabla_2)$. 
 In light of Proposition \ref{prop:bd},
it suffices to show that 
the line bundle $\Lcal$ is uniquely determined by $\rho$.  
Let $\Lcal$ and $\Mcal$ be line bundles of degree
$-(n-1)(g-1)$ such that $H^0(X,\Vcal_\rho^\ast\otimes\Lcal)\neq \{0\}$ and
$H^0(X,\Vcal_\rho^\ast\otimes\Mcal)\neq \{0\}$.
Let $\{\Vcal_i\}$ be the oper structure for
$(\Vcal_\rho,\nabla_1)$, and assume $\Vcal_\rho/\Vcal_{n-1}=\Lcal$. If $\Lcal$
and $\Mcal$ are not isomorphic, it follows from 
$$
0\lra \Lcal^\ast\otimes\Mcal\lra \Vcal_\rho^\ast\otimes\Mcal\lra
\Vcal_{n-1}^\ast\otimes\Mcal\lra 0\ ,
$$
that 
$H^0(X,\Vcal_{n-1}^\ast\otimes\Mcal)\neq \{0\}$.
Now for $j\leq n-1$, $\deg \Rcal_j^\ast\otimes\Mcal<0$, so by
applying this argument successively we conclude 
that
$H^0(X,\Vcal_{1}^\ast\otimes\Mcal)\neq \{0\}$.
But
$\Vcal_1^\ast\otimes\Mcal=\Lcal^\ast\otimes\Mcal\otimes\Kcal^{1-n}$ 
also has negative degree, so we get a
contradiction.
\end{proof}

\begin{remark}
There are precisely $n^{2g}$ possibilities for the 
line bundle $\Lcal$ in Proposition \ref{prop:bd}.  
These choices label the components of $\Op_n$. As in Remark \ref{rem:projective}, these correspond precisely to the $n^{2g}$ ways of lifting a monodromy representation in $\PSL_n(\CBbb)$ to $\SL_n(\CBbb)$.
For simplicity, from
 now on we will always take $\Lcal=\Kcal^{-(n-1)/2}$ where if $n$ is even we assume a fixed choice of $\Kcal^{1/2}$.
\end{remark}

\subsubsection{Opers and differential equations}
We first show how to
 obtain an oper from a local system that is realized in 
$\Kcal ^{1-q}$, $n=2q-1$.  So assume we are given the exact sequence 
\eqref{eqn:n-diffop}, and set $\Vcal=\Vcal_n=\Ocal\otimes_{\underline\CBbb}\Vbold$. For $k=1,\ldots, n-1$,
define
$$
\Vcal_{n-k}=\bigl\{ \sum_{i=1}^n f_i\otimes \vbold_i : \sum_{i=1}^n f^{(j)}_i\varphi(\vbold_i)=0\ ,\ j=0,\ldots, k-1\bigr\}\ .
$$
Then $\Vcal_{n-k}\subset\Vcal$ is a coherent subsheaf and we have exact sequences 
\begin{equation} \label{eqn:v-sequence}
0\lra \Vcal_{n-k-1}\lra \Vcal_{n-k}\lra \Kcal ^{1-q+k}\lra 0\ .
\end{equation} 
Property (i) of Definition \ref{def:oper} is clearly satisfied. Furthermore, in view of \eqref{eqn:v-sequence}, the connection $\nabla$ induces an $\Ocal $-linear map $\Vcal_{n-k-1}\to \Vcal_{n-k}/\Vcal_{n-k-1}\otimes \Kcal\simeq \Kcal^{2-q+k} $, by
$$
\sum_{i=1}^n f_i\otimes \vbold_i \mapsto \sum_{i=1}^n f^{(k+1)}_i\varphi(\vbold_i)\ ,
$$
and this is an isomorphism of sheaves.  So property (ii) holds as well.

Conversely,
suppose that $\Vcal$ is a rank $n$ holomorphic bundle with 
holomorphic connection $\nabla$ that admits an oper structure. 
By  Lemma \ref{lem:gr} we have $\Vcal/\Vcal_{n-1}\simeq \Kcal ^{1-q}$.  
It follows that 
for any $\SL_n$-oper (we continue to assume $\Lcal=\Kcal^{-(n-1)/2}$), 
$H^0(X,\Vcal^\ast\otimes\Kcal ^{1-q})\neq\{0\}$. Since the monodromy of an oper is irreducible by Proposition \ref{prop:irreducible}, the hypotheses of Theorem \ref{thm:ohtsuki} are satisfied, and $(\Vcal, \nabla)$ is realized in $\Kcal ^{1-q}$.

\begin{theorem}[Beilinson-Drinfeld
{\cite{BeilinsonDrinfeld05}}] \label{thm:bd}
The embedding
 above gives an isomorphism between the connected components of
$
\Op_n
$
and the $($affine$)$ Hitchin base
$ \bigoplus_{j=2}^n H^0(X, \Kcal ^{j})
$.
\end{theorem}

\begin{corollary}[Teleman {\cite{Teleman60}}]
The monodromy of a differential equation \eqref{eqn:n-diffop}
$($or \eqref{eqn:ell}$)$ is never unitary.
\end{corollary}

\begin{proof}
If $\rho$ is the monodromy, then 
from the correspondence above and Lemma \ref{lem:gr} we see
that  $\Vcal_\rho$ is an unstable bundle. But then  from the easy
direction of Theorem \ref{thm:narasimhan-seshadri} (see Proposition \ref{prop:easy}),
$\Vcal_\rho$ cannot admit a flat unitary connection.
\end{proof}

\subsubsection{Opers and moduli space}

The main goal of this section is to prove the following

\begin{theorem}
The map $\Op_n\hookrightarrow \Mfrak_B^{(n)}$ is a
proper embedding.
\end{theorem}

\noindent 
By the upper semicontinuity of the Harder-Narasimhan type (see Section \ref{sec:higgs-fields}),
 this theorem is a direct consequence of the following

\begin{proposition} \label{prop:maximal-hn-type}
Among bundles with holomorphic connections,
opers have strictly maximal Harder-Narasimhan type.
\end{proposition}

We begin with 

\begin{lemma} \label{lem:hn-type}
The Harder-Narasimhan filtration of a bundle $\Vcal$ with an oper 
structure is given by the oper filtration itself.
\end{lemma}

\begin{proof}
It suffices to show that for each $j=0, \ldots, n-1$, $\Vcal_{j+1}/\Vcal_j$ is the maximally destabilizing subsheaf of $\Vcal/\Vcal_j$.  In order to do this, let $\mu_{max}(\Vcal/\Vcal_j)$ denote the maximal slope of a  subsheaf of $\Fcal\subset\Vcal/\Vcal_j$, $0<\rank\Fcal<\rank(\Vcal/\Vcal_j)$.  We make the inductive hypothesis that
$$
\mu_{max}(\Vcal/\Vcal_j)=\mu(\Vcal_{j+1}/\Vcal_j)=(n-1)(g-1)-j(2g-2)\ .
$$
Note that this is trivially satisfied for $j=n-1$.  Now suppose $j\leq n-2$ and let $\Fcal\to \Vcal/\Vcal_{j}$ be the maximally destabilizing subsheaf.  Then $\Fcal$ is semistable, and from the sequence
$$
0\lra \Vcal_{j+1}/\Vcal_j\lra \Vcal/\Vcal_j\lra \Vcal/\Vcal_{j+1}\lra 0\ ,
$$
and the inductive hypothesis, we have
$$
\mu(\Fcal)\geq \mu(\Vcal_{j+1}/\Vcal_j)>\mu(\Vcal_{j+2}/\Vcal_{j+1})=\mu_{max}(\Vcal/\Vcal_{j+1})\ .
$$
It follows that the induced  map $\Fcal\to \Vcal/\Vcal_{j+1}$ must vanish.  Therefore, $\Fcal\simeq\Vcal_{j+1}/\Vcal_j$, and moreover the inductive hypothesis is satisfied for $j$.
This concludes the proof.
\end{proof}

\begin{proof}[Proof of Proposition \ref{prop:maximal-hn-type}] (cf.\ \cite[Theorem 5.3.1]{JoshiPauly09})
Let $(\Vcal, \nabla)$ be an unstable bundle with holomorphic connection.
I claim that 
it suffices to assume that $\nabla$ is irreducible.   Indeed, in the case of rank $1$ there is nothing to prove. Suppose the result has been proven for rank $<n$ and suppose $(\Vcal, \nabla)$ is reducible.  Since the Harder-Narasimhan type is upper semicontinuous, we may assume there is a splitting $(\Vcal,\nabla)=(\Vcal_1,\nabla_1)\oplus (\Vcal_2,\nabla_2)$, with $n_i=\rank\Vcal_i\geq 1$.  Then by the induction hypothesis, it suffices to assume the $\Vcal_i$ have the Harder-Narasimhan types of rank $n_i$-opers.  Indeed, if not then we can change the Harder-Narasimhan types of $\Vcal_i$, without changing the ordering of the slopes for $\Vcal$, so that $\Vcal$ has a larger Harder-Narasimhan type.
Let
\begin{equation} \label{eqn:mu-oper}
\mu_i=\mu_i^{(n)}=\mu(\Kcal ^{q-i})=(n+1-2i)(g-1)\ ,
\end{equation}
be the Harder-Narasimhan type of a rank $n$-oper (see Lemmas \ref{lem:hn-type} and \ref{lem:gr}).  
If $\lambda_i$ is a reordering of the slopes $\{\mu_i^{(n_1)}, \mu_j^{(n_2)}\}$, we need to show 
\begin{equation} \label{eqn:lambda-inequality}
\sum_{i=1}^k\lambda_i\leq \sum_{i=1}^k \mu_i^{(n)}\ ,
\end{equation}
for all $k=1,\ldots, n$, with strict inequality for some $k$.  Assume $n_1\geq n_2$.  Without changing the ordering of the slopes we can sequentially subtract even integers from the leading entries $\mu_i^{(n)}$, $\lambda_i=\mu_i^{(n_1)}$ for $2i\leq n_1-n_2$, and add the integers to last entries where $n_1+n_2+2\leq 2i$.
 Notice that the multiplicities of the resulting first and last slopes in $\{\mu_i\}$ and $\{\lambda_i\}$ are equal and  will cancel in the sums, so it suffices to consider the intervening sums.  This reduces the problem to one of two cases:  
$n_1=n_2$ or $n_1=n_2+1$ (and  $n=n_1+n_2$), where it is straightforward to verify \eqref{eqn:lambda-inequality}.

With this understood, we may assume that $(\Vcal,\nabla)$ is irreducible.
The Harder-Narasimhan type of an oper is given by \eqref{eqn:mu-oper}.
Let $\Vcal_{i-1}\subset\Vcal_{i}$, $i=1,\ldots, \ell$, be the Harder-Narasimhan filtration of $\Vcal$,
and $\lambda_i=\mu(\Vcal_i/\Vcal_{i-1})$. Let $n_i=\rank(\Vcal_i/\Vcal_{i-1})$ and $d_i=n_i\lambda_i$.
 Then it suffices to show
\begin{equation} \label{eqn:hn}
\sum_{i=1}^j n_i\lambda_i\leq \sum_{i=1}^{\rank(\Vcal_j)}\mu_i\ ,
\end{equation}
for $j=1,\ldots, \ell$.  The left hand side is just $\deg\Vcal_j$
while the right hand side is
$$
\sum_{i=1}^{\rank(\Vcal_j)}(n+1-2i)(g-1)=(g-1)\rank(\Vcal_j)(n-\rank(\Vcal_j))\ .
$$
Hence, \eqref{eqn:hn} is equivalent to 
\begin{equation} \label{eqn:hn2}
\deg\Vcal_j\leq (g-1)\bigl( \sum_{i=1}^{j}n_i\bigr)\bigl(n-\sum_{i=1}^{j}n_i\bigr)\ .
\end{equation}
Repeatedly apply Proposition \ref{prop:bound} to find
\begin{align*}
\lambda_j&\leq \lambda_{j+1}+2g-2 \\
\lambda_j&\leq \lambda_{j+2}+2(2g-2) \\
\lambda_j&\leq \lambda_{j+i}+i(2g-2) \\
\lambda_j&\leq \lambda_{\ell}+(\ell-j)(2g-2) \ ,
\end{align*}
for any $i\leq \ell-j$.  This implies
\begin{align*}
\frac{n_{j+1}}{n_j}d_j&\leq d_{j+1}+(2g-2)n_{j+1} \\
\frac{n_{j+i}}{n_j}d_j&\leq d_{j+i}+i(2g-2)n_{j+i} \\
\frac{n_{\ell}}{n_j}d_j&\leq d_{\ell}+(\ell-j)(2g-2)n_{j+1} \ ,
\end{align*}
from which we have 
\begin{equation} \label{eqn:inequality}
\bigl( \sum_{i=1}^{\ell-j} n_{i+j}\bigr)\frac{d_j}{n_j} \leq \sum_{i=1}^{\ell-j} d_{i+j}+(2g-2)\sum_{i=1}^{\ell-j} in_{i+j}
\end{equation}
Consider first  the case $j=1$. Then \eqref{eqn:inequality} becomes
\begin{align}
\bigl( \sum_{i=2}^{\ell} n_{i}\bigr)\frac{d_1}{n_1} &\leq \sum_{i=2}^{\ell} d_{i}+(2g-2)\sum_{i=2}^{\ell} (i-1)n_{i} \notag \\
( n-n_{1})\frac{d_1}{n_1} &\leq -d_1+(2g-2)\sum_{i=2}^{\ell} (i-1)n_{i}\notag \\
d_1 &\leq \frac{n_1}{n}(2g-2)\sum_{i=2}^{\ell} (i-1)n_{i} \ . \label{eqn:hn3}
\end{align}
We claim that
\begin{equation} \label{eqn:claim1}
\frac{2}{n}\sum_{i=2}^{\ell} (i-1)n_{i}  \leq n-n_1=\sum_{i=2}^\ell n_i\ .
\end{equation}
Note that this combined with \eqref{eqn:hn3} proves \eqref{eqn:hn2} in the case $j=1$.
To prove the claim, let $r_i=n_i-1\geq 0$.  Then \eqref{eqn:claim1} becomes
\begin{align*}
2\sum_{i=2}^{\ell} (i-1)(r_{i}+1)  &\leq n\sum_{i=2}^\ell (r_i+1) \\
2\sum_{i=2}^{\ell} (i-1)r_{i} +\ell(\ell-1)  &\leq \left[\sum_{i=2}^\ell (r_i+1)\right]^2+n_1\sum_{i=2}^\ell (r_i+1)\ ,
\end{align*}
which holds if
$$
2\sum_{i=2}^{\ell} (i-1)r_{i} +\ell(\ell-1)  \leq \left[\sum_{i=2}^\ell r_i + (\ell-1)\right]^2+(\ell-1)\ ,
$$
which in turn, after canceling like terms from both sides,  holds if
$$
\sum_{i=2}^{\ell} (i-1)r_{i}   \leq \sum_{i=2}^\ell (\ell-1)r_i \ ,
$$
and the latter is manifestly true since $r_i\geq 0$. Hence, \eqref{eqn:claim1} holds.

We now proceed by induction.  So suppose that \eqref{eqn:hn2} holds for $j$. We show that it holds also for $j+1$.  Adding \eqref{eqn:hn2} (for $j$) and \eqref{eqn:inequality} (for $j+1$) we have
\begin{align*}
\deg\Vcal_{j+1}= \deg\Vcal_j+d_{j+1}&\leq
(g-1)\sum_{i=1}^j n_i\sum_{i=1}^{\ell-j} n_{i+j} \\
&\qquad -\frac{n_{j+1}}{\sum_{i=2}^{\ell-j} n_{i+j}}\sum_{i=1}^{j+1} d_i+
\frac{n_{j+1}}{\sum_{i=2}^{\ell-j} n_{i+j}}(2g-2)\sum_{i=1}^{\ell-j-1}in_{i+j+1} \\
\frac{\sum_{i=1}^{\ell-j} n_{i+j}}{\sum_{i=2}^{\ell-j} n_{i+j}}\, \deg\Vcal_{j+1}
&\leq 
(g-1)\sum_{i=1}^j n_i\sum_{i=1}^{\ell-j} n_{i+j}+
\frac{n_{j+1}}{\sum_{i=2}^{\ell-j} n_{i+j}}(2g-2)\sum_{i=1}^{\ell-j-1}in_{i+j+1} \\
\frac{\deg\Vcal_{j+1}}{n-\sum_{i=1}^{j+1} n_{i}}
&\leq 
(g-1)\sum_{i=1}^j n_i+
\frac{n_{j+1}}{\sum_{i=1}^{\ell-j} n_{i+j}\sum_{i=2}^{\ell-j} n_{i+j}}(2g-2)\sum_{i=1}^{\ell-j-1}in_{i+j+1}\ ,
\end{align*}
where in going from the first inequality to the second we have used the fact that $\deg\Vcal_{j+1}=\sum_{i=1}^{j+1}d_i$.
Hence, it suffices to show
$$
2\sum_{i=1}^{\ell-j-1}i n_{i+j+1}\leq \sum_{i=1}^{\ell-j}n_{i+j}\sum_{i=2}^{\ell-j} n_{i+j}\ .
$$
In terms of the $r_i$ defined above, this becomes
\begin{align*}
2\sum_{i=2}^{\ell-j}(i-1) (r_{i+j}+1)&\leq (r_{j+1}+1)\sum_{i=2}^{\ell-j} (r_{i+j}+1)+\left(\sum_{i=2}^{\ell-j} (r_{i+j}+1)\right)^2 \\
2\sum_{i=2}^{\ell-j}(i-1) r_{i+j}+(\ell-j)(\ell-j-1)&\leq (r_{j+1}+1)\sum_{i=2}^{\ell-j} r_{i+j}
+(r_{j+1}+1)(\ell-j-1)  \\
&\qquad\qquad\qquad\qquad
+\left(\sum_{i=2}^{\ell-j} r_{i+j}+(\ell-j-1)\right)^2 \ .
\end{align*}
But this is a consequence of
$$
2\sum_{i=2}^{\ell-j}(i-1) r_{i+j}\leq 2\sum_{i=2}^{\ell-j}(\ell-j-1) r_{i+j}\ ,
$$
which obviously holds.  This completes the proof of the maximality of the Harder-Narasimhan type.  We now show that if the Harder-Narasimhan type of $(\Vcal,\nabla)$ is maximal then the filtration $\{\Vcal_i\}$ is an oper structure.  Indeed, consider the $\Ocal$-linear map $\nabla: \Vcal_i\to \Vcal/\Vcal_{i+1}\otimes \Kcal$.  By Remark \ref{rem:maximal}, the minimal slope of a quotient of $\Vcal_i$ is $\mu_i=\mu(\Vcal_i/\Vcal_{i-1})$, whereas the maximal slope of a subsheaf of $\Vcal/\Vcal_{i+1}\otimes\Kcal$ is 
$$\mu(\Vcal_{i+2}/\Vcal_{i+1}\otimes\Kcal)=\mu_{i+2}+2g-2=\mu_{i+1}=\mu_i-(2g-2)<\mu_i\ .$$
Hence, the map above must be zero, and $\nabla\Vcal_i\subset \Vcal_{i+1}\otimes \Kcal$.  By irreducibility of the connection, $\Vcal_i/\Vcal_{i-1}\to \Vcal_{i+1}/\Vcal_i\otimes\Kcal$ is nonzero.  Since these are line bundles with the same degree, this map is an isomorphism.  Therefore, conditions (i) and (ii) in Definition \ref{def:oper} are satisfied. This completes the proof.
\end{proof}

\subsection{The Eichler-Shimura isomorphism} Let us return in more detail to Example \ref{ex:fuchsian}.
 For $q\in \frac{1}{2}\ZBbb$, let $V_q$ denote the $2q-1$ dimensional irreducible representation of $\SL_2(\CBbb)$. Let  $\rho:\pi\to \SL_2(\CBbb)$ be the (lift of the) monodromy of a projective connection on $X$.  
 We can realize the local system $\Vbold_\rho$ in $\Kcal^{-1/2}$ for some choice of spin structure.  For $q\geq 3/2$, let $\Vbold_q$ denote the local system obtained by composing $\rho$ with the representation $V_q$:
$$
\rho^{(n)}:\pi\lra\SL_2(\CBbb)\lra \SL(V_q)\ .
$$
  Then $\Vbold_q$ is realized in $\Kcal^{1-q}$, and we have
\begin{equation} \label{eqn:fuchsian-oper}
0\lra \Vbold_q \lra \Kcal ^{1-q}
 \stackrel{D}{\xrightarrow{\hspace*{.75cm}}}  \Kcal ^q \lra 0\ .
\end{equation}
Since $q\geq 3/2$, $H^0(X,\Kcal^{1-q})\simeq H^1(X, \Kcal^q)^\ast =\{0\}$.  This implies $H^0(X,\Vbold_q)=H^2(X,\Vbold_q)=\{0\}$, and the long exact sequence associated to \eqref{eqn:fuchsian-oper} becomes
$$
0\lra H^0(X,\Kcal^{q}) \stackrel{\delta}{\xrightarrow{\hspace*{.5cm}}}   H^1(X,\Vbold_q) 
\lra  H^1(X,\Kcal^{1-q}) \lra 0\ .
$$
The coboundary map $\delta$ is called {\bf Eichler integration}.
The reason for the terminology is the following: if $\omega$ is a global holomorphic section of $\Kcal^q$, then on sufficiently small open sets $U_i$ we can solve the inhomogeneous equation $Dy_i=\omega\bigr|_{U_i}$.  If we set ${\bf v}_{ij}=y_i-y_j$, then $\{{\bf v}_{ij}\}$ is a $1$-cocycle with values in $\Vbold_q$ which represents $\delta\omega$.

In any case,
it follows that we have an isomorphism (cf.\ \cite{Eichler57, Shimura59, Gunning67})
\begin{equation} \label{eqn:eichler}
H^1(X,\Vbold_q)\simeq H^0(X,\Kcal^{q}) \oplus (H^0(X,\Kcal^{q}))^\ast\ .
\end{equation}

Eq.\ \eqref{eqn:eichler} can be used to describe the tangent space to the Betti moduli space at $[\rho^{(n)}]$ (this was explained to me by Bill Goldman \cite{Goldman}).  From Weil's description of the tangent space,
\begin{equation} \label{eqn:weil}
T_{[\rho^{(n)}]}\Mfrak_B^{(n)}\simeq H^1(X, \End\Vbold_q)\ .
\end{equation}
Now representations of $\SL_2(\CBbb)$ are self-dual: $V_q^\ast\simeq V_q$.  By the Clebsch-Gordon rule for decomposition of tensor product representations, we have
$$
\End V_q= (V_q\otimes V_q^\ast)_{\tr=0}\simeq  (V_q\otimes V_q)_{\tr=0}=\bigoplus_{\stackrel{j=2}{j\in \ZBbb}}^{2q-1} V_j
$$
(note that the trivial representation $V_{3/2}$ is eliminated by the traceless condition). 
This decomposition translates into one for the local system. It follows that
$$
H^1(X, \End\Vbold_q)=\bigoplus_{\buildrel {j=2}\over {j\in \ZBbb}}^{2q-1} H^1(X,\Vbold_j)\ .
$$
Combining this with eqs.\ \eqref{eqn:eichler} and \eqref{eqn:weil} we obtain
$$
T_{[\rho^{(n)}]}\Mfrak_B^{(n)}\simeq
\bigoplus_{j=2}^n H^0(X,\Kcal^j)\oplus (H^0(X,\Kcal^j))^\ast\ .
$$
This should be compared with \eqref{eqn:higgs-fuchs}!

%%%%%%%%%%%%%%%%%%%%%%%%%%%%%%%%%%%%%%%%%%%%%%%%%
%%%% BIBLIOGRAPHY %%%%%%%%%%%%%%%%%%%%%%%%%%%%%%%%%%%%

%\bibliographystyle{plain}
%\bibliography{ref}

\begin{thebibliography}{1}
\frenchspacing

\bibitem{Atiyah57} M. Atiyah,
 Complex analytic connections in fibre bundles. Trans. Amer. Math. Soc. 85 (1957), 181--207.

\bibitem{AtiyahBott82} M. Atiyah and R. Bott, The Yang-Mills equations over Riemann surfaces.  Phil.\ Trans.\ R.
Soc.\ Lond.\ A  308 (1982), 523--615.

\bibitem{BeilinsonDrinfeld05} A. Beilinson and V. Drinfeld, Opers.
arXiv:math/0501398

\bibitem{BenZviFrenkel01a} D. Ben-Zvi and E. Frenkel, Spectral curves, opers and integrable systems. Publ. Math. Inst. Hautes Etudes Sci. No. 94 (2001), 87--159. 

\bibitem{BenZviFrenkel01b} D. Ben-Zvi and E. Frenkel, 
Vertex algebras and algebraic curves. Mathematical Surveys and Monographs, 88. American Mathematical Society, Providence, RI, 2001. xii+348 pp.

\bibitem{BenZviBiswas04} D. Ben-Zvi and I. Biswas, Opers and theta functions. Adv. Math. 181 (2004), no. 2, 368--395.


\bibitem{Boalch01} P. Boalch,
Symplectic manifolds and isomonodromic deformations. Adv. Math. 163 (2001), no. 2, 137--205.

\bibitem{Bolibrukh02} A.  Bolibrukh,
The Riemann-Hilbert problem on a compact Riemann surface. (Russian) Tr. Mat. Inst. Steklova 238 (2002), Monodromiya v Zadachakh Algebr. Geom. i Differ. Uravn., 55--69; translation in Proc. Steklov Inst. Math. 2002, no. 3 (238), 47--60.

\bibitem{BradlowGothenPrada} S. Bradlow, O. Garc\'ia-Prada, P. Gothen,
Maximal surface group representations in isometry groups of classical Hermitian symmetric spaces. Geom. Dedicata 122 (2006), 185--213.


\bibitem{BurgerIozziWienhard} M. Burger, A. Iozzi, and A. Wienhard,
Higher Teichm\"uller spaces: From SL(2,R) to other Lie groups, to appear in the Handbook of Teichm\"uller theory.
 


\bibitem{Corlette88} K.  Corlette,
Flat $G$-bundles with canonical metrics.
J. Diff. Geom. 28 (1988), 361--382.


\bibitem{Daskalopoulos92} G. Daskalopoulos, The topology of the space of stable bundles on a Riemann surface.  J.\ Diff.\
Geom.\  36 (1992), 699--746.



\bibitem{DWWW} G. Daskalopoulos, J. Weitsman,  R. Wentworth, and G. Wilkin,
Morse theory and hyperk\"ahler Kirwan surjectivity for Higgs bundles. J. Differential Geom. 87 (2011), no. 1, 81--115.

\bibitem{DaskalWentworth04} G. Daskalopoulos and R. Wentworth,
Convergence properties of the Yang-Mills flow on K\"ahler surfaces. J. reine angew. Math. 575 (2004), 69--99.

\bibitem{DWW} G. Daskalopoulos, R. Wentworth, and G. Wilkin,
Cohomology of SL(2,C) character varieties of surface groups and the action of the Torelli group. Asian J. Math. 14 (2010), no. 3, 359--383

\bibitem{Deligne70} P.
Deligne,  Equations diff\'erentielles \`a points singuliers
r\'eguliers.  Lecture Notes in Mathematics, Vol. 163.
Springer-Verlag, Berlin-New York, 1970. iii+133 pp.

\bibitem{Itzykson91}
P. Di Francesco, C.  Itzykson, and J.-B. Zuber.  Classical
$W$-algebras. Comm. Math. Phys. 140 (1991), no. 3, 543--567


\bibitem{Donaldson83} S. Donaldson,
A new proof of a theorem of Narasimhan and Seshadri.  J. Diff. Geom.  18  (1983),  no. 2, 269--277.

\bibitem{Donaldson85} S. Donaldson,  Anti self-dual Yang-Mills connections over complex algebraic surfaces and
stable vector bundles.  Proc. London Math. Soc.  50 (1985), 1--26.

\bibitem{Donaldson87} S.  Donaldson,
Twisted harmonic maps and the
self-duality equations.  Proc. London Math. Soc.   55 (1987),
127--131.


\bibitem{EellsSampson64} J.  Eells and J.  Sampson,
Harmonic mappings of Riemannian manifolds.
Amer.  J.  Math.  86 (1964),
109--160.

\bibitem{Eichler57} M. Eichler,
Eine Verallgemeinerung der Abelschen Integrale. Math. Z. 67 (1957), 267--298.

\bibitem{EsnaultViehweg99}
H. Esnault and E. Viehweg, Semistable bundles on curves and irreducible representations of the fundamental group. Algebraic geometry: Hirzebruch 70 (Warsaw, 1998), 129--138, Contemp. Math., 241, Amer. Math. Soc., Providence, RI, 1999.

\bibitem{Frenkel07}
E.
Frenkel, Lectures on the Langlands program and conformal field theory. Frontiers in number theory, physics, and geometry. II, 387Ð533, Springer, Berlin, 2007. 

\bibitem{Goldman} W. Goldman, The Eichler-Shimura isomorphism.  Unpublished note.

\bibitem{GriffithsHarris78} P. Griffiths and J. Harris,
Principles of algebraic geometry. Pure and Applied Mathematics. Wiley-Interscience, New York, 1978. xii+813 pp.

\bibitem{Gunning66} R. Gunning, 
 Lectures on Riemann surfaces. Princeton Mathematical Notes Princeton University Press, Princeton, N.J. 1966 iv+254 pp.

\bibitem{Gunning67} R. Gunning, Special coordinate coverings of Riemann surfaces. 
Math. Ann. 170 1967 67--86.

\bibitem{Guichard11} O. Guichard,
Aspects topologiques et g\'eom\'etriques des repr\'esentations de groupes de surfaces,
Habilitation 2011.

\bibitem{Hamilton75} R. Hamilton, 
Harmonic maps of manifolds with boundary. Lecture Notes in Mathematics, Vol. 471. Springer-Verlag, Berlin-New York, 1975. i+168 pp.


\bibitem{HarderNarasimhan74} G. Harder and M. Narasimhan, 
On the cohomology groups of moduli spaces of vector bundles on curves.
Math. Ann. 212 (1974/75), 215--248. 




\bibitem{Hartman67} P. Hartman, On homotopic harmonic maps. Can. J. Math. 
19 (1967), 673--687.

\bibitem{HawleySchiffer66} 
N. Hawley and M. Schiffer, Half-order differentials on Riemann surfaces. Acta Math. 115 (1966), 199--236.


\bibitem{Hejhal75}
D.
Hejhal, Monodromy groups for higher-order differential equations. Bull. Amer. Math. Soc. 81 (1975), 590--592.

\bibitem{Hejhal78}
D.
Hejhal,
 Monodromy groups and Poincar\'e series.
 Bull. Amer. Math. Soc. 84 (1978), no. 3, 339--376.

\bibitem{Hitchin87a} N. Hitchin, The self-duality equations on a Riemann surface.  Proc. London Math. Soc. (3)  55  (1987),  no. 1, 59--126.

\bibitem{Hitchin87b} N. Hitchin, Stable bundles and integrable systems.  Duke Math. J.  54  (1987),  no. 1, 91--114.

\bibitem{Hitchin90}
N. Hitchin, Harmonic maps from a 2-torus to the 3-sphere. J. Differential Geom. 31 (1990), no. 3, 627--710.

\bibitem{Hitchin92} N. Hitchin,
 Lie groups and Teichm\"uller space.
Topology 31 (1992), no. 3, 449--473.

\bibitem{Hong01}
M.-C. Hong, Heat flow for the Yang-Mills-Higgs field and the Hermitian
Yang-Mills-Higgs metric. Ann. Global Anal. Geom., 20(1) (2001), 23--46.

\bibitem{JoshiPauly09} 
K. Joshi and C. Pauly,
Hitchin-Mochizuki morphism, Opers and Frobenius-destabilized vector bundles over curves.
	arXiv:0912.3602v2



\bibitem{Jost94} J. Jost, Equilibrium maps between metric spaces.
Calc. Var. 2 (1994), 173--204.

\bibitem{JostYau91} 
      J. Jost and S.-T. Yau,
Harmonic maps and group representations.
   In \emph{Differential Geometry},
 Pitman Monographs Pure Appl. Math. 52,
 B. Lawson and K. Tenenblat, eds., 
  1991, 241--259.



\bibitem{Kirwan84} F. Kirwan, Cohomology of quotients in symplectic and algebraic geometry. Mathematical Notes, 31. Princeton University Press, Princeton, NJ, 1984.

\bibitem{Kobayashi87} S. Kobayashi, Differential geometry of complex vector bundles. Publications of the Mathematical Society of Japan, 15. Kano Memorial Lectures, 5. Princeton University Press, Princeton, NJ; Iwanami Shoten, Tokyo, 1987. 


\bibitem{KorevaarSchoen93}
N. Korevaar and R. Schoen, Sobolev
spaces and harmonic maps for metric space targets.
Comm. Anal. Geom.  1 (1993), 561--659.

\bibitem{KorevaarSchoen97}
 N. Korevaar and R. Schoen, Global existence theorems for
harmonic maps to non-locally compact spaces. Comm. Anal. Geom.  
5 (1997), 213--266.



\bibitem{Labourie91} F. Labourie,
Existence
d'applications harmoniques tordues \`a valeurs dans les
vari\'et\'es
\`a courbure n\'egative. Proc. Amer. Math. Soc.  111 (1991), 877--882.


\bibitem{LubotzkyMagid85} A. Lubotzky and A. Magid, Varieties of representations of finitely generated groups.  Mem. Amer. Math. Soc.  58  (1985),  no. 336.

%\bibitem{Moser64} J. Moser,
%A Harnack inequality for parabolic differential equations. Commun.  Pure Applied Math. 17  (1964), no. 2, 101--134.

\bibitem{NarasimhanSeshadri65} M. Narasimhan and C. Seshadri, Stable and unitary vector bundles on a compact Riemann surface.  Ann. of Math.  82  (1965), no. 2, 540--567.

\bibitem{Nitsure91} N. Nitsure,
 Moduli space of semistable pairs on a curve.  Proc. London Math. Soc. (3)  62  (1991),  no. 2, 275--300.

\bibitem{Ohtsuki82} M. Ohtsuki,
On the number of apparent singularities of a linear differential equation. Tokyo J. Math. 5 (1982), no. 1, 23--29.
%
\bibitem{Uniformization} Henri Paul de Saint-Gervais,
Uniformisation des surfaces de Riemann.  Retour sur un th\'eor\`eme centenaire.
 ENS \'Editions, Lyon, 2010. 544 pp.

\bibitem{Schoen83} R. Schoen, Analytic aspects of the harmonic map
problem. Seminar on Non-linear Partial Differential Equations
(S.S. Chern, ed.), MSRI Publ. 2, Springer-Verlag, New York 1983,
321--358.

\bibitem{Seshadri67}
C. Seshadri,
Space of unitary vector bundles on a compact Riemann surface.
Ann. of Math. (2) 85 (1967), 303--336. 

\bibitem{Shimura59} G. Shimura,
Sur les int\'egrales attach\'ees aux formes automorphes. 
J. Math. Soc. Japan 11 (1959), 291--311. 

\bibitem{Simpson87} C. Simpson, Systems of Hodge bundles and uniformization.  Harvard thesis, 1987.


\bibitem{Simpson88} C. Simpson,
Constructing variations of Hodge structure using Yang-Mills theory and applications to uniformization.  J. Amer. Math. Soc.  1  (1988),  no. 4, 867--918. 

\bibitem{Simpson90} C. Simpson,
 Harmonic bundles on noncompact curves. J. Amer. Math. Soc. 3 (1990), no. 3, 713--770.

\bibitem{Simpson92} C. Simpson,
Higgs bundles and local systems.  Inst. Hautes Etudes Sci. Publ. Math.  No. 75 (1992), 5--95.

\bibitem{Simpson94a} C. Simpson,
Moduli of representations of the fundamental group of a smooth projective variety. I.  Inst. Hautes Etudes Sci. Publ. Math.  No. 79 (1994), 47--129.

\bibitem{Simpson94b} C. Simpson,
Moduli of representations of the fundamental group of a smooth projective variety. II. Inst. Hautes Etudes Sci. Publ. Math. No. 80 (1994), 5--79 (1995)

\bibitem{Simpson10} C. Simpson,
Iterated destabilizing modifications for vector bundles with connection. Vector bundles and complex geometry, 183--206, Contemp. Math., 522, Amer. Math. Soc., Providence, RI, 2010. 

\bibitem{Teleman60}
C. Teleman.
Sur les structures fibr\'ees osculatrices d'une surface de
Riemann. 
Comment. Math. Helv. 34 (1960), 175--184. 

\bibitem{Uhlenbeck82} K.
Uhlenbeck, Connections with $L^{p}$ bounds on
curvature. Comm. Math. Phys. 83 (1982), no. 1, 31--42.

\bibitem{UhlenbeckYau86} K. Uhlenbeck and S.-T. Yau,  On the existence of Hermitian-Yang-Mills connections in stable vector bundles. Comm. Pure Appl. Math.  39 (1986), S257--S293.

\bibitem{Weil38} A. Weil, G\'en\'eralisation de fonctions abeliennes, J. Math. Pures. Appl. 17 (1938), 47-87.

\bibitem{WentworthWilkin11} R. Wentworth and G. Wilkin, Cohomology of $\U(2,1)$ representation varieties of surface groups. arXiv:1109.0197

\bibitem{Wilczynski62}
E.
Wilczynski, Projective differential geometry of curves
and ruled surfaces. Chelsea Publishing Co., New York 1962
viii+298 pp.

\bibitem{Wilkin08} G. Wilkin,  Morse theory for the space of Higgs bundles, 
Comm. Anal. Geom.  16 (2008), 283--332.


\end{thebibliography}

\end{document}